\documentclass[a4paper,10pt,reqno]{amsart}

\usepackage[T1]{fontenc}			
\usepackage[utf8]{inputenc}
\usepackage[italian, english]{babel}
\usepackage{amsmath, amsfonts, amssymb, amsthm, amscd, mathtools}
\usepackage{geometry}
\usepackage[square,numbers]{natbib}	
\usepackage{bbm, bm}
\usepackage{xspace} 
\usepackage{enumerate, enumitem}
\usepackage{hyperref}
\usepackage{xcolor}
\usepackage{comment}

\newcommand{\bbF}{{\ensuremath{\mathbb F}} }

\newcommand{\bbP}{{\ensuremath{\mathbb P}} }

\newcommand{\bbW}{{\ensuremath{\mathbb W}} }

\newcommand{\cA}{{\ensuremath{\mathcal A}} }
\newcommand{\cB}{{\ensuremath{\mathcal B}} }
\newcommand{\cC}{{\ensuremath{\mathcal C}} }

\newcommand{\cE}{{\ensuremath{\mathcal E}} }
\newcommand{\cF}{{\ensuremath{\mathcal F}} }

\newcommand{\cI}{{\ensuremath{\mathcal I}} }

\newcommand{\cL}{{\ensuremath{\mathcal L}} }

\newcommand{\cO}{{\ensuremath{\mathcal O}} }

\newcommand{\cR}{{\ensuremath{\mathcal R}} }
\newcommand{\cS}{{\ensuremath{\mathcal S}} }

\newcommand{\cX}{{\ensuremath{\mathcal X}} }

\newcommand{\cZ}{{\ensuremath{\mathcal Z}} }

\newcommand{\dC}{{\ensuremath{\mathrm C}} }
\newcommand{\dD}{{\ensuremath{\mathrm D}} }

\newcommand{\R}{\mathbb{R}}

\renewcommand{\P}{\mathbb{P}}
\newcommand{\E}{\mathbb{E}}

\newcommand{\ind}{\ensuremath{\mathbf{1}}}

\DeclarePairedDelimiterX{\inprod}[2]{\langle}{\rangle}{#1, #2}
\DeclareMathOperator{\tr}{tr}

\newcommand{\changeepsilon}
{
\let\temp\epsilon
\let\epsilon\varepsilon
\let\varepsilon\temp
}

\newcommand{\changetheta}
{
\let\temp\theta
\let\theta\vartheta
\let\vartheta\temp
}
\newcommand{\changephi}
{
\let\temp\phi
\let\phi\varphi
\let\varphi\temp
}

\theoremstyle{plain}
\newtheorem{theorem}{Theorem}[section]
\newtheorem*{theorem*}{Theorem}
\newtheorem{lemma}[theorem]{Lemma}
\newtheorem*{lemma*}{Lemma}

\newtheorem*{proposition*}{Proposition}

\theoremstyle{definition}
\newtheorem{definition}{Definition}[section]
\newtheorem{assumption}{Assumption}[section]
\newtheorem{hypothesis}[assumption]{Hypothesis}

\theoremstyle{remark}
\newtheorem{remark}{Remark}[section]
\newtheorem*{remark*}{Remark}

\definecolor{darkviolet}{rgb}{0.58, 0.0, 0.83}

\definecolor{gre}{rgb}{0.03,0.50,0.03}

\numberwithin{equation}{section}

\hypersetup{pdftitle={Multi dimensional ergodic singular control},pdfauthor={Alessandro Calvia, Federico Cannerozzi, Giorgio Ferrari},hidelinks,%
pdfstartview={XYZ null null 1.4},colorlinks=true,linkcolor=blue,urlcolor=blue}

\newcommand{\mail}[1]{\href{mailto:#1}{\normalfont\texttt{#1}}}

\makeatletter
\def\@setthanks{\vspace{-\baselineskip}\def\thanks##1{\@par##1\@addpunct.}\thankses}
\makeatother

\title[Multi-dimensional Ergodic Singular Stochastic Control]{Optimal Policy Characterization for a Class of Multi-Dimensional Ergodic Singular Stochastic Control Problems}
\author[A.~Calvia]{Alessandro~Calvia\textsuperscript{\MakeLowercase{a},1}}
\thanks{\noindent \textsuperscript{a} Department of Mathematics, Politecnico di Milano, Milan (Italy).}
\author[F.~Cannerozzi]{Federico~Cannerozzi\textsuperscript{\MakeLowercase{b},2}}
\thanks{\noindent \textsuperscript{b} Bielefeld University, Center for Mathematical Economics (IMW), Bielefeld (Germany).}
\author[G.~Ferrari]{Giorgio~Ferrari\textsuperscript{\MakeLowercase{b},3}}
\thanks{\noindent
\noindent \textsuperscript{1} E-mail: \mail{alessandro.calvia@polimi.it}.
\\
\noindent \textsuperscript{2} E-mail: \mail{federico.cannerozzi@uni-bielefeld.de}.
\\
\noindent \textsuperscript{3} E-mail: \mail{giorgio.ferrari@uni-bielefeld.de}}

\date{\today}

\begin{document}

\changephi
\changetheta
\changeepsilon
\allowdisplaybreaks
	
\begin{abstract}
In ergodic singular stochastic control problems, a decision-maker can instantaneously adjust the evolution of a state variable using a control of bounded variation, with the goal of minimizing a long-term average cost functional. The cost of control is proportional to the magnitude of adjustments. This paper characterizes the optimal policy and the value in a class of multi-dimensional ergodic singular stochastic control problems. These problems involve a linearly controlled one-dimensional stochastic differential equation, whose coefficients, along with the cost functional to be optimized, depend on a multi-dimensional uncontrolled process $Y$. We first provide general verification theorems providing an optimal control in terms of a Skorokhod reflection at $Y$-dependent free boundaries, which emerge from the analysis of an auxiliary Dynkin game. We then fully solve two two-dimensional optimal inventory management problems. To the best of our knowledge, this is the first paper to establish a connection between multi-dimensional ergodic singular stochastic control and optimal stopping, and to exploit this connection to achieve a complete solution in a genuinely two-dimensional setting.
\end{abstract}
\maketitle


\smallskip

\noindent \textbf{AMS 2020:} 93E03; 60G40; 49L20; 35R35; 90B05.

\smallskip

\bigskip

\section{Introduction} \label{sec:intro}

\subsection*{State of the art} In singular stochastic control problems, a decision-maker can instantaneously adjust the dynamics of an underlying state process via a control process that has paths of bounded variation. The control's action may cause discontinuities in the state process and may exhibit (and typically does exhibit) singular behavior, in the sense that the (random) Borel-measure on the time axis induced by the optimal control process may be singular with respect to the Lebesgue measure. 

A rich body of literature addresses infinite time-horizon discounted problems as well as finite time-horizon problems and their applications. In such problems, the optimal policy typically involves reflecting the state process at the topological boundary (free boundary) of the so-called no-action (also known as continuation) region, where the optimally controlled process should be maintained with minimal effort. The construction of the optimal control process is then framed in terms of the solution to a Skorokhod reflection problem (see \cite{boryc2016characterization,kruk2000optimal,kruk2000optimal-b} and the more recent \cite{dianetti2023spa} for a review of the literature). Under suitable conditions, the boundary of the no-action region in a singular stochastic control problem can be identified with the free boundary arising in an auxiliary optimal stopping problem, whose value function is given by the derivative (in the direction of the optimally controlled state) of the singular stochastic control problem's value function. See, among others, \cite{baldursson1996irreversible,karatzas1984connections} for the connection between monotone problems and optimal stopping; \cite{boetius1998connections,karatzas2001connections} for the relation between bounded-variation problems and Dynkin games (zero-sum games of optimal stopping); and \cite{de2018stochastic} for a connection between nonzero-sum games of singular stochastic control and games of optimal stopping. Thanks to the connection with optimal stopping, the characterization of the optimal policy is closely linked to the study of obstacle problems, which are more tractable than variational inequalities with gradient constraints, as those appearing in infinite time-horizon discounted or finite time-horizon singular stochastic control problems. Today, there exists a satisfactory understanding of those classes of singular stochastic control problems, and significant progress has been made in characterizing the optimal policy in various multi-dimensional settings (see, e.g., \cite{chiarolla1994optimal,chiarolla2000inflation,de2017optimal,ferrari2018optimal,soner1989regularity}).

On the other hand, singular stochastic control problems with an ergodic (long-time average) performance criterion have been studied far less than their infinite-time horizon discounted or finite-time horizon counterparts. They arise, for example, in optimal inventory management (see, e.g., \cite{helmes2017continuous,taksar1985average}), { in approximation of queuing systems under heavy traffic (see, e.g., \cite{arapostathis2015ergodic,arapostathis2016ergodic,Budhiraja2011ergodic} and \cite{taksar1991queueing,taksar1992queueing} for the discounted case)}, in applications where it is crucial to consider the payoffs received by successive generations, such as the exploitation of a natural resource \cite{hening2022optimal,cohen2022optimal,liang2020ergodic}, and also appear in the asymptotic expansion of optimal consumption-investment problems with small transaction costs (see \cite{possamai2015homogenization,soner2013homogenization}).
We also refer to \cite{cohen2024existence} for the existence of relaxed solutions to multi-dimensional stationary singular stochastic control problems and their mean-field game counterpart, to \cite{kurtz1998markov} for the linear programming formulation of ergodic singular control problems, as well as to \cite{christensen2024data} for data-driven rules in ergodic reflection problems.

In ergodic singular stochastic control problems, the dynamic programming equation (DPE) takes the form of a variational problem for the pair $(v,\lambda^{\star})$, where $v$ is the so-called potential function and $\lambda^{\star}$ is the value of the problem, which is constant due to the ergodic setting -- see \cite{hynd2012eigenvalue,menaldi1992singular,menaldi1999optimal} for studies of such equations via methods from the theory of partial differential equations (PDEs). In one-dimensional settings, the dynamic programming equation reduces to an ordinary differential equation with a derivative constraint, allowing explicit solutions via the "guess-and-verify" method. Typically, the no-action region is an interval on the real line (possibly unbounded) and $\lambda^{\star}$ can be expressed in terms of the problem's data and the free boundary. For explicitly solvable ergodic problems in one dimension, see \cite{hening2022optimal,kunwai2022ergodic,liang2020ergodic,taksar1985average,weerasinghe2002stationary,weerasinghe2007abelian}, among others. Additionally, \cite{lokka2011model,lokka2013long} address two-dimensional fully degenerate problems that remain solvable using the "guess-and-verify" approach. However, the absence of a probabilistic representation for the potential function $v$ makes, to the best of our knowledge, the characterization of the optimal singular control policy an open problem in truly multi-dimensional settings.

\subsection*{Our contributions} In this paper, we contribute to this open question by providing a characterization of the optimal policy in a class of multi-dimensional singular stochastic control problems. This is achieved by establishing a novel connection to optimal stopping. Our setting considers a one-dimensional It\^o-diffusion $X$ that is linearly controlled by a bounded-variation process (the singular control) and has drift and volatility coefficients modulated by an uncontrolled (factor) process $Y$ valued in $\mathbb{R}^d$, $d\geq 1$. This latter process also affects the long-term average expected cost functional that the decision-maker aims to minimize.  

We first provide a general verification theorem (see Theorem \ref{thm:verification} below) that allows us to recover the value of the problem as the limit superior of the expected long-term average of the process $(\lambda(Y_t))_{t\geq0}$, where the Borel-measurable function $\lambda:\mathbb{R}^d \to \mathbb{R}$ appears, together with a function $V:\mathbb{R}^d \times \mathbb{R} \to \R$, as the solution to an auxiliary PDE with gradient constraints. It is important to note that this PDE is \emph{not} the dynamic programming equation associated with the ergodic problem, meaning that the pair $(V,\lambda)$ does \emph{not} represent the potential function and the problem’s value, respectively. In what follows, we shall refer to $V$ as the pseudo-potential function. Furthermore, it is important to note that our results hold even in the case of a process $Y$ that is not ergodic and thus does not admit a stationary distribution (see Remark \ref{rmk:ergodicity} below).

In Theorem \ref{thm:verification}, we also show that a control process (if it exists) that keeps the state process \`a la Skorokhod in the region where the gradient constraints are not binding is, in fact, optimal. To identify this region and thereby characterize the optimal control, Theorems \ref{thm:HJB_sol} and \ref{thm:HJB_sol_weak} demonstrate (under different regularity conditions) that $(V,\lambda)$ from Theorem \ref{thm:verification} can be constructed by analyzing an auxiliary optimal stopping problem. Specifically, this is a zero-sum optimal stopping game with value $U$. In particular, $V$ is given as the integral of $U$ with respect to $x$, and $\lambda(y)$ can be determined in terms of the free boundaries associated with $U$ and the problem’s data. It is worth noticing that a connection to optimal stopping in this spirit had already been observed in \cite{karatzas1983class,taksar1985average} in the one-dimensional setting (see also \cite{cannerozzi2024cooperation,cao2023stationary,dianetti2023ergodic} in the context of one-dimensional stationary mean-field games of singular controls with scalar interaction). However, such a relation is obtained via direct computations based on the fact that the underlying process is one-dimensional, thus leading to a dynamic programming equation for the problem's value which is an explicitly solvable ODE with derivative constraint. To the best of our knowledge, this is the first paper to establish a connection between optimal stopping and ergodic singular stochastic control in multiple dimensions and to exploit it as a road-map for characterizing the value and the optimal policy.

The general verification theorems \ref{thm:verification}, \ref{thm:HJB_sol}, and \ref{thm:HJB_sol_weak} are then tested in two case studies arising in the field of inventory control.
We consider two different inventory models with stochastic mean-reversion levels.
The two examples differ in their mean-reversion level dynamics and observability specifications, such that, in the first case, the weaker conditions of Theorem \ref{thm:HJB_sol_weak} are met, while in the second case study, Theorem \ref{thm:HJB_sol} applies.

The first case study treated in Section \ref{sec:application:partial_observation} deals with the optimal management of an inventory for which the level of mean-reversion is modulated by a two-state continuous-time Markov chain, which, however, is not observable by the decision maker. This leads to an ergodic singular stochastic control problem under partial observation, which we address by studying the equivalent separated problem. In the latter, the state process is a truly two-dimensional diffusion, whose first component is given by the dynamics of the inventory (which can be instantaneously increased or decreased via a control of bounded variation), and whose second component is the filter or belief process. Given that these two components are driven by the same Brownian motion (the so-called innovation process), the diffusion term of the state process is degenerate. Following the receipt of Theorems \ref{thm:verification} and \ref{thm:HJB_sol_weak}, we identify a Dynkin game of optimal stopping whose value $U$ is expected to coincide with the derivative of the pseudo-potential function in the direction of the inventory component of the state process. The analysis of such a Dynkin game is anything but standard. In fact, the aforementioned degeneracy of the state process leads to the fact that the second-order differential operator appearing in the variational inequality associated with $U$ is of parabolic type. By introducing a change of variables that expresses the differential operator in its canonical form, and studying the optimal stopping problem in those new coordinates, we are able to show that $U \in \dC^1(\mathbb{R}^2)$, and that two bounded, non-increasing belief-dependent free boundaries $a_{\pm}$ trigger the saddle point of the stopping game.
{ The free boundaries $a_\pm$ yield the optimal control process, as it is constructed as the process that reflects the (optimally controlled) inventory process $X^{\star}$ at the belief-dependent free boundaries $a_{\pm}$ so that $a_+(Y_t) \leq X^{\star}_t \leq a_-(Y_t)$ $d\P \otimes dt$-almost everywhere. The results follows from the application of Theorems \ref{thm:verification} and \ref{thm:HJB_sol_weak}, which can be applied thanks to the regularity properties of $U$.}
It is worth noting that the continuous differentiability of $U$ is proven by suitably employing techniques as in \cite{de2020global}, after proving that the boundary points are probabilistically regular for the state process. The proof of this fact, in turn, hinges on the Feller property of the state process, which we show by proving that the infinitesimal generator of the state process, though not uniformly elliptic, is hypoelliptic, as it satisfies the H\"ormander conditions (see also \cite{ernst2024gapeev}).

In the second case-study, we deal with a nondegenerate problem (in the sense that the involved second-order differential operator is uniformly elliptic) for which the conditions of Theorems \ref{thm:verification} and \ref{thm:HJB_sol} are met. In this example, the stochastic mean-reversion level of the stored good evolves as a mean-reverting diffusion itself.
In this case, the Dynkin game which is expected to be associated (in the sense of Theorem \ref{thm:HJB_sol}) to the ergodic singular stochastic control problem has value function $U \in \bbW^{2,\infty}_{loc}(\R^2)$ and its saddle point is given by the entry times of the underlying state process to the regions where $U$ equates the two obstacles (the stopping regions). The boundaries of those sets can be described in terms of two non-increasing curves (the free boundaries) which are finite and Lipschitz-continuous.
The Lipschitz-continuity of the free boundaries is functional in the application of Theorem \ref{thm:HJB_sol} as it allows to obtain the required transversality condition. Moreover, the Lipschitz property of $a_\pm$ is also an interesting result \textit{per se}, given that in certain obstacle problems the Lipschitz property is the preliminary regularity needed to upgrade -- via a bootstrapping procedure and suitable technical conditions -- the regularity of the free boundary to $\dC^{1,\alpha}$-regularity, for some $\alpha \in (0,1)$ and eventually to $\dC^{\infty}$-regularity (see the introduction of \cite{deangelis_stabile2019lipschitz} for a discussion and literature review on this).

The rest of this paper is organized as follows: in Section \ref{sec:model}, we define the ergodic stochastic singular control problem, state the assumptions and prove the verification theorems.
In particular, {in} Section \ref{sec:1stverthm} we introduce the PDE for the pair pseudo-potential function-value profile $(V,\lambda)$, and we state and prove a preliminary verification theorem, while in Section \ref{sec:Dynkinconn} we identify the auxiliary Dynkin game, and we show how to use its value function to build a solution $(V,\lambda)$ and an optimal control.
In Section \ref{sec:rmks_thms}, we gather some remarks on the previous results.
Section \ref{sec:application} contains the application of previous results to inventory managements problems. In particular, in Section \ref{sec:application:partial_observation} an inventory control problem with partially observable demand is considered, while Section \ref{sec:application:full_observation} addresses an inventory control problem with fully observable demand.


\subsection{Notation}\label{sec:notation}
For $n \geq 1$, we denote by $\R^{n \times n}$ the set of $n \times n$ real-valued matrices.
{ We follow the convention that $\R^d$ is identified by column-vectors, i.e. $d \times 1$ matrices.}

For any $n, k \geq 1$ and any $\cO \subseteq \R^n$, we denote by $\dC^{k}(\cO)$ the set of real-valued $k$-times continuously differentiable functions.
When $k=\infty$, we mean that the function is infinitely many times differentiable.
Analogously, we write $\dC^{k}_b(\R^d)$ to denote the space of function $f \in \dC^k(\cO)$ bounded and with bounded derivatives.
We denote by $\dC^\infty_c(\cO)$ the set of infinitely many times real-valued functions with compact support.
We indicate by $L^{\infty}_{loc}(\cO)$ the set of all locally bounded functions with respect to Lebesgue measure.
We denote by $\bbW^{2,\infty}_{loc}(\cO)$ the Sobolev space of all functions $f \in L^{\infty}_{loc}(\cO)$ such that the partial derivatives up to the second order exist in the weak sense and are in $L^{\infty}_{loc}(\cO)$.
For any $f$ in $\dC^2(\cO)$ or $\bbW^{2,\infty}_{loc}(\cO)$, we denote the (weak) derivative with respect to $x_1$ and $x_2$ by $f_{x_1 x_2}$.

For $n_1, \,n_2 \geq 1$ and $\cO \subseteq \R^{n_1} \times \R^{n_2}$, we denote by $\dC^{2,1}(\cO)$ the set of real-valued continuous functions on $\cO$ that are continuously differentiable twice with respect to the first variable and once with respect to the second. We write $\dC^{2,1}_b(\cO)$ when the considered function is bounded, with bounded derivatives.

\section{The Problem}
\label{sec:model}

Let $d \geq 1$.
Let $(\Omega, \cF, \bbF \coloneqq (\cF_t)_{t \geq 0}, \bbP)$ be a complete filtered probability space, with $\bbF$ satisfying the usual assumptions, on which $\bbF$-adapted standard Brownian motion $W$ and $d$-dimensional Brownian motion $B$ are defined.
We assume that $W$ and $B$ are correlated, in the sense that there exists a vector $\rho=(\rho_i)_{i=1}^d$, $-1 \leq \rho_i \leq 1$, so that the { cross-variation} $\langle W, B_i\rangle_t$ is equal to $\rho_i t$, for any $t \geq 0$, for any $i=1,\dots,d$\footnote{For a possible construction, consider $(B^i)_{i=1}^d$ independent standard Brownian motions, and $(\rho_i)_{i=1}^d$ so that $\sum_{i=1}^d \rho_i^2 = 1$. Then, $W=\sum_{i=1}^d\rho_i B^i$ is a standard Brownian motion so that $\langle W, B_i\rangle_t = \rho_i t$.}.
Let $\xi$ be a singular control, i.e. a process $\xi:[0,\infty) \times \Omega \to \R$ which is $\bbF$-adapted, $\xi$ càdlàg, $\xi_{0-} = 0$ $\P$-a.s. and whose total variation in any interval $[0,T]$, $T>0$, is finite, i.e.\ $\vert \xi \vert_{[0,T]} < \infty$.
Moreover, we assume that $\E[\vert \xi \vert_{[0,T]}]< \infty$ for any $T > 0$.
We identify $\xi$ with its positive and negative parts given by the Jordan decomposition, i.e.\ $\xi_t = \xi^+_t - \xi^-_t$ $\P$-a.s. with $\xi^\pm$ non-decreasing and with { disjoint supports}.

\smallskip
Let $(b,\sigma): \R \times \R^d \to \R \times \R$ and $(\eta,\zeta): \R^d \to \R^d \times \R^{d \times d}$ be measurable.
For any $(x,y) \in \R \times \R^d$ and singular control $\xi$, we consider $(X^\xi,Y)$ to be the solution of the pair of stochastic differential equations
\begin{equation}\label{eq:SDE}
\left\{ \begin{aligned}
& dX^\xi_t = b(X^\xi_t, Y_t) \, dt + \sigma(X^\xi_t, Y_t) \, dW_t + d\xi^+_t - d\xi^-_t, && X^\xi_{0^-} = x, \\
& dY_t = \eta(Y_t)dt + \zeta(Y_t)dB_t, && Y_0= y.
\end{aligned} \right.
\end{equation}
Whenever necessary, to stress the dependence of the solution to \eqref{eq:SDE} on the initial conditions $(x,y) \in \R \times \R^d$,
we write $\E_{x,y}[\cdot]$ to denote the expectation under $\P_{x,y}(\cdot)=\P(\cdot \vert X^{\xi}_{0^-}=x, Y_0 = y)$.
Analogously, we write $\E_{y}[\cdot]$ 
to denote the expectation under $\P_{y}(\cdot)=\P(\cdot \vert Y_0 = y)$.
Conditions on $b,\sigma,\eta,\zeta$ appear in Assumption \ref{assumptions} below.

\smallskip
In the sequel, the following class of admissible controls, satisfying a suitable growth condition, will be employed:
\begin{definition}\label{def:admissible_policies}
A singular control $\xi$ is admissible if 
\begin{equation}\label{eq:controls:admissibility}
    \varlimsup_{T \to \infty} \frac{1}{T}\E_{x,y}[ \vert X^\xi_T \vert ] = 0.
\end{equation}
We denote by $\cB$ the set of admissible controls.
\end{definition}
We then consider the long-time average (ergodic) singular stochastic control problem
\begin{equation}\label{eq:ctrl_pb}
\inf_{\xi \in \cB} \varlimsup_{T \to +\infty} \frac{1}{T} \, \E_{x,y}\left[\int_0^T c(X_t^\xi, Y_t) \, dt + K_+ \xi^+_T + K_- \xi^-_T\right],
\end{equation}
where $c \colon \R \times \R^d \to [0,+\infty)$ is a suitable measurable function and $K_+, K_- > 0$ are fixed constants.
We say that $\xi^\star$ is optimal for the ergodic singular control problem if the associated cost functional as in  \eqref{eq:ctrl_pb} is minimal.

\smallskip
Let $(X^0,Y)$ be the uncontrolled process, solving the SDE
\begin{equation}\label{eq:SDE:X0}
\left\{ \begin{aligned}
& dX^0_t = b(X^0_t, Y_t) \, dt + \sigma(X^0_t, Y_t) \, dW_t , && X^0_0 = x \in \mathbb{R}, \\
& dY_t = \eta(Y_t)dt + \zeta(Y_t)dB_t, && Y_0= y \in \R^d,
\end{aligned} \right.
\end{equation}
and denote by $\cL_{(X^0,Y)}$ its infinitesimal generator, i.e.
\begin{multline}\label{eq:generator:X0}
    \cL_{(X^0,Y)}f(x,y) = b(x,y)f_x(x,y) + \frac{1}{2}\sigma^2(x,y)f_{xx}(x,y) + \sum_{i=1}^d \eta_i(y)f_{y_i}(x,y) \\
    + \frac{1}{2}\sum_{i,j=1}^d a_{ij}(y)f_{y_i y_j}(x,y) + \sum_{i,j=1}^d { \rho_j \sigma(x,y)\zeta_{ij}(y) f_{x y_i}(x,y) }
\end{multline}
where $f \in \dC^2_b(\R \times \R^d)$ and $a(y) := (\zeta \zeta^\top)(y)$.
We denote by $(\rho\zeta): \R^d \to \R^d$ the function given by { $(\rho\zeta)_i(y) = \sum_{j=1}^d\rho_j\zeta_{ij}(y)$}, so that the generator $\cL_{(X^0,Y)}$ can be expressed equivalently in matrix form as
\begin{multline}\label{eq:generator:X0_matrix}
    \cL_{(X^0,Y)}f(x,y) = b(x,y)f_x(x,y) + \frac{1}{2}\sigma^2(x,y)f_{xx}(x,y) +  \eta(y)\nabla_{y}f(x,y) \\
    + \frac{1}{2}\tr \big( a(y)\nabla^2_y f(x,y) \big) + \sigma(x,y)(\rho\zeta)(y)\nabla_{y}f_x(x,y).
\end{multline}

We make the following standard assumptions on the functions involved so far:
\begin{assumption}\label{assumptions}
\mbox{}
\begin{enumerate}
    \item $b(x,y)$ and $\sigma(x,y)$ are Lipschitz-continuous with at most linear growth; moreover, the partial derivatives $b_x(x,y)$ and $\sigma_x(x,y)$ exist and they are continuous, and $(\sigma \sigma_x)(x,y)$ is locally Lipschitz, jointly in $(x,y)$;
    \item $\eta(y)$ and $\zeta(y)$ are Lipschitz-continuous with at most linear growth; moreover, $\sigma_x(x,y)\zeta_{ij}(y)$ is locally Lipschitz jointly in $(x,y)$;
    \item $c(x,y)$ is jointly continuous in $(x,y)$ and continuously differentiable with respect to $x$.
    Moreover, there exist $p \geq 1$ and $\kappa > 0$ so that $\vert c(x,y) \vert \leq \kappa (1 + \vert x \vert^p + \vert y \vert^p ) $ for any $(x,y) \in \R \times \R^d$.
\end{enumerate}
\end{assumption}

{ Points 1 and 2 in Assumptions 2.1 are standard and are related to existence and uniqueness of stochastic differential equations (see \eqref{eq:SDE:X0} and \eqref{eq:X_hat_0}). The differentiability assumptions on $b(x,y)$ and $c(x,y)$ guarantee that the auxiliary Dynkin game (see \eqref{eq:dynkin}) is well-defined.}

\subsection{A preliminary verification theorem}\label{sec:1stverthm}

In order to state and prove the first verification theorem for problem \eqref{eq:ctrl_pb} (see Theorem \ref{thm:verification} below), we introduce the following partial differential equation (PDE) for the pair $(V,\lambda)$, where $V \colon \R \times \R^d \to \R$ and $\lambda \colon \R^d \to \R$:
\begin{equation}\label{eq:HJB}
\min\{\cL_{(X^0,Y)} V(x,y) + c(x,y) - \lambda(y), -V_x(x,y) + K_-, V_x(x,y) + K_+\}=0, \; (x,y) \in \R \times \R^d,
\end{equation}
{ and we divide $\R \times \R^d$ in the following subsets $\cI$, $\cA_+$ and $\cA_-$:}
\begin{equation}\label{eq:HJB:waiting_exert}
\begin{aligned}
\cA_+ & \coloneqq \{(x,y) \in \R \times \R^d \colon V_x(x,y) + K_+ \leq 0\}, \\
\cA_- & \coloneqq \{(x,y) \in \R \times \R^d \colon K_- - V_x(x,y) \leq 0\}, \\
{  \cI } & \coloneqq \{ { (x,y) \in \R \times \R^d \colon -K_+ < V_x(x,y) < K_-} \}.
\end{aligned}
\end{equation}
We say that $(V,\lambda)$ is a solution to \eqref{eq:HJB} if $V \in \bbW^{2,\infty}_{loc}(\R \times \R^d)$, $\lambda \in L^\infty_{loc}(\R^d)$ and \eqref{eq:HJB} holds true for a.e. $(x,y) \in \R \times \R^d$. 
Notice that, if $(V,\lambda)$ is a solution to \eqref{eq:HJB}, then $V_x \in \dC(\R \times \R^d)$ by Sobolev's embedding.
It follows then that $\cA_+$ and $\cA_-$ are closed and $\cI$ is open.
Moreover, as $(V,\lambda)$ solves equation \eqref{eq:HJB} a.e., we have that $V_x(x,y) + K_+ \geq 0$ for almost all $(x,y)$, and similarly $-V_x(x,y) + K_- \geq 0$ a.e.
This fact, paired with the definition of $\cA_\pm$ and the continuity of $V_x$, implies that $\cA_+ = \{ (x,y): \, V_x(x,y) + K_+ = 0 \}$, $\cA_- = \{ (x,y): \, V_x(x,y) - K_- = 0 \}$ and 
\begin{equation}\label{eq:general:equality_almost_everywhere}
    \cL_{(X^0,Y)} V(x,y) + c(x,y) - \lambda(y) = 0 \quad \text{ for a.e. } (x,y) \in \cI.
\end{equation}
The following verification theorem establishes the relationship of the pair $(V,\lambda)$ with the value of problem \eqref{eq:ctrl_pb} and the optimal control $\xi^\star$.
\begin{theorem}
\label{thm:verification}
Recall \eqref{eq:HJB:waiting_exert}. Let $(V,\lambda)$ be a solution to \eqref{eq:HJB} such that $V \in \dC^2(\cI)$, { $\lambda \in \dC(\R^d)$}, $\vert V(x,y) \vert \leq \kappa (1 + \vert x \vert )$ for some $\kappa > 0$ and $(\lambda(Y_t))_{t \in [0,T]}$ is $d\P \otimes dt$ integrable, for any $T > 0$.
Then,
\begin{equation}\label{eq:verification:value_geq_equal}
\varlimsup_{T \to +\infty} \frac 1T \, \E_{y}\left[\int_0^T \lambda(Y_s) \, ds\right] \leq \inf_{\xi \in \cB} \varlimsup_{T \to +\infty} \frac{1}{T} \, \E_{x,y}\left[\int_0^T c(X_t^\xi, Y_t) \, dt + K_+ \xi^+_T + K_- \xi^-_T\right].
\end{equation}
Moreover, suppose that there exists an admissible control $\xi^\star \in \cB$ such that { $(X_t^{\xi^\star}, Y_t) \in \cI$, $\P$-a.s., for all $t \geq 0$}, and such that, for all $t \geq 0$,
\begin{equation*}
\xi^{\star, +}_t = \int_{[0,t]} \ind_{(X_s^{\xi^\star}, Y_s) \in \cA_+} \, d\xi^{\star, +}_s, \qquad
\xi^{\star, -}_t = \int_{[0,t]} \ind_{(X_s^{\xi^\star}, Y_s) \in \cA_-} \, d\xi^{\star, -}_s, \quad \mathbb{P}-\text{a.s.}
\end{equation*}
Then
\begin{multline}\label{eq:verification:value_leq_equal}
\varlimsup_{T \to +\infty} \frac 1T \, \E_{y}\left[\int_0^T \lambda(Y_s) \, ds\right] = \varlimsup_{T \to +\infty} \frac{1}{T} \, \E_{x,y}\left[\int_0^T c(X_t^{\xi^\star}, Y_t) \, dt + K_+ \xi^{\star,+}_T + K_- \xi^{\star,-}_T\right] \\
\geq \inf_{\xi \in \cB} \varlimsup_{T \to +\infty} \frac{1}{T} \, \E_{x,y}\left[\int_0^T c(X_t^\xi, Y_t) \, dt + K_+ \xi^+_T + K_- \xi^-_T\right].
\end{multline}
\end{theorem}
\begin{proof}
We argue as in \cite[Chapter VIII, Theorem 4.1]{fleming_soner2003controlled} (see also \cite[Theorem 2.4]{bandini_gozzi2022optimal_dividend}).
For each $m \geq 1$, we consider the standard mollifier $\varphi_m(x,y)= m^{-(d+1)}\varphi(mx,my)$ with $\varphi \in \dC_c^\infty(B_1(0))$, $\varphi \geq 0$, $\int_{\R^{d+1}} \varphi(x,y) dxdy=1$, where $B_1(0)$ is the ball in $\R^{d+1}$ centered in zero with radius one, so that  $\varphi_m(x,y) \in \dC_c^\infty(B_{1/m}(0))$.
Then, we define $(V^m)_{m \geq 1} \subset C^{\infty}(\R^{d+1})$ by convolution as  $V^m := V \ast \varphi_m$.
Since $V \in \bbW^{2,\infty}_{{loc}}(\R \times \R^d)$, $V \in \dC^1(\R \times \R^d)$ by Sobolev embedding. Thus, for any compact set $K \subset \R \times \R^d$, it holds
\begin{equation}\label{filtered:verification_thm:approximation_first_order}
    \lim_{k \rightarrow \infty}||V^m -V||_{L^\infty(K)}=0, \quad \lim_{k \rightarrow \infty}||\dD(V^m -V)||_{L^\infty(K)}=0, \;\; \forall \, \dD \in \{ \partial_x, \, \partial_{y_i}, \, i=1,\dots,d\}.
\end{equation}
Since the second-order derivatives are not continuous over $\R \times \R^d$, we can not conclude that the second-order derivatives of $V^m$ converge to the corresponding second-order derivatives of $V$ uniformly on every compact subset of $\R \times \R^d$.
However, by using the definition of weak derivative and the fact that $V \in \bbW^{2,\infty}_{loc}(\R \times \R^d)$, it holds  $\dD(V^m) = (\dD V )\ast \varphi_m$ for any $\dD = \partial_{zu}$, with $z,u=x,y_1,\dots,y_d$.
Then, exploiting the continuity of the coefficients of $\cL_{(X^0,Y)}$ we deduce that
\begin{equation}\label{filtered:verification_thm:approximation_generator}
    \epsilon_{m,K} \coloneqq \sup_{(x,y) \in K }\vert \cL_{(X^0,Y)} V^m(x,y) - (\cL_{(X^0,Y)} V \ast \varphi_m ) (x,y) \vert \overset{m \to \infty}{\longrightarrow} 0,
\end{equation}
for any compact $K \subseteq \R \times \R^d$.
Let now $c^m \coloneqq c \ast \varphi_m$ and $\lambda^m \coloneqq \lambda \ast \varphi_m$ and notice that \eqref{eq:HJB} implies
\begin{equation}\label{eq:thm_verifica:var_ineq_convoluted}
    (\cL_{(X^0,Y)}V) \ast \varphi_m(x,y) + c^m(x,y) - \lambda^m(x,y) \geq 0 \quad \forall \, (x,y) \in \R \times \R^d.
\end{equation}
As $\lambda$ is a function of $y$ only, the convolution $\lambda \ast \varphi_m$ is actually just a function of $y$, as it can be directly verified from the definition of convolution.
By triangulating with $(\cL_{(X^0,Y)}V)\ast\varphi_m$, using  \eqref{filtered:verification_thm:approximation_generator} and \eqref{eq:thm_verifica:var_ineq_convoluted}, we find
\begin{equation}\label{eq:thm_verifica:final_inequality}
\begin{aligned}
\inf_{(x,y) \in K } & \Big( \cL_{(X^0,Y)} V^m (x,y) + c^m(x,y) - \lambda^m(y) \Big) \\
& \geq { \inf_{(x,y) \in K } \Big( \cL_{(X^0,Y)} V^m(x,y) - (\cL_{(X^0,Y)} V) \ast \varphi_m (x,y) \Big) } \\
& +  \inf_{(x,y) \in K } \Big( ( \cL_{(X^0,Y)} V ) \ast \varphi_m (x,y) + c^m(x,y) - \lambda^m(y) \Big)  \geq -\epsilon_{m,K}\overset{m \to \infty}{\longrightarrow} 0.    
\end{aligned}
\end{equation}
Let then $(x,y) \in \R\times\R^d$ and let $(K_n)_{n \geq 1}$ be a sequence of compact sets such that $K_{n} \subset K_{n+1}$ and $\lim_{n \to \infty} K_n = \R \times \R^d$.
Also, set $\tau_{n} := \inf\{t \geq 0: \,  (X^{\xi}_t,Y_t) \notin K_n \} ${, and denote by $(X^{\xi}_{\cdot \land \tau_{n}^-},Y_{\cdot \land \tau_{n}^-})$ the left-continuous version of the pair of processes $(X^\xi, Y)$ stopped at $\tau_n$.}
By applying Dynkin's formula to $(V^m(X_t,Y_t))_{t \geq 0}$, under $\P_{x,y}$ we obtain 
\begin{align*}
    \E_{x,y} [ &  V^m(X^{\xi}_{t \land {\tau_{n}^-}},Y_{t \land {\tau_{n}^-}})] - V^m(x,y) \\
    & = \E_{x,y}\bigg[\int_0^{t \land \tau_{n}}\cL_{(X^0,Y)}V^m(X^{\xi}_s,Y_s)ds + \int_{{[0, t \land \tau_{n})}} V^m_x(X^{\xi}_{s-},Y_{s-})d\xi_s \\
    & \quad + \sum_{{0 \leq s < t \land \tau_{n}} } \Big( V^m(X^{\xi}_{s},Y_{s}) - V^m(X^{\xi}_{s-},Y_{s-}) - V^m_x(X^{\xi}_{s-},Y_{s-})\Delta\xi_s \Big) \bigg] \\
    & = \E_{x,y}\bigg[\int_0^{t \land \tau_{n}}\cL_{(X^0,Y)}V^m(X^{\xi}_s,Y_s)ds + \int_0^{t\land\tau_{n}} V^m_x(X^{\xi}_{s-},Y_{s-})d\xi^{c}_s  \\
    & \quad + \sum_{{0 \leq s < t \land \tau_{n}} } \Big( V^m(X^{\xi}_{s},Y_{s}) - V^m(X^{\xi}_{s-},Y_{s-}) \Big) \bigg],
\end{align*}
where $\xi^{c}$ denotes the continuous part of $\xi$.
As $Y$ is continuous and $X^{\xi}_{s} = X^{\xi}_{s-} \pm \Delta\xi^\pm_s$, we deduce that, for any $s \geq 0$, it holds
\[
V^m(X^{\xi}_{s},Y_{s}) - V^m(X^{\xi}_{s-},Y_{s-}) = \int_0^{\Delta\xi^{+}_{s}}V^m_x(X^{\xi}_{s-}+z,Y_{s})dz - \int_0^{\Delta\xi^{-}_s}V^m_x(X^{\xi}_{s-}-z,Y_{s})dz,
\]
$\P_{x,y}$-a.s. Hence, 
\begin{equation}\label{eq:filtered:thm:optimal_control:equality_tau_n}
\begin{aligned}
     \E_{x,y} & [V^m(X^{\xi}_{t \land {\tau_{n}^-}},Y_{t \land {\tau_{n}^-}})] - V^m(x,y) \\
     & = \E_{x,y}\bigg[\int_0^{t \land \tau_{n}}\cL_{(X^0,Y)}V^m(X^{\xi}_s,Y_s)ds + \int_0^{t \land \tau_{n}} V^m_x(X^{\xi}_{s},Y_{s})d\xi^{c}_s  \\
     & + \sum_{{0 \leq s < t \land \tau_{n}} } \Big(\int_0^{\Delta\xi^{+}_{s}}V^m_x(X^{\xi}_{s-}+z,Y_{s})dz - \int_0^{\Delta\xi^{-}_s}V^m_x(X^{\xi}_{s-}-z,Y_{s})dz \Big) \bigg].
\end{aligned}
\end{equation}
Given that $(X^\xi_{s \land {\tau_{n}^-}},Y_{s \land {\tau_{n}^-}})_{s \geq 0}$ belongs to the compact $K_n$, $\P_{x,y}$-a.s., for $n$ large enough, we add and subtract $c^m(X^{\xi}_t,Y_t) - \lambda^m(Y_t)$ in the first integral in the right-hand side of \eqref{eq:filtered:thm:optimal_control:equality_tau_n}, invoke \eqref{eq:thm_verifica:final_inequality} and use the bounds on the partial derivative $V^m_x$ given by \eqref{eq:HJB}, to deduce
\begin{equation}\label{eq:filtered:thm:optimal_control:inequality_tau_nm}
\begin{aligned}
     \E_{x,y} & [V^m(X^{\xi}_{t \land {\tau_{n}^-}},Y_{t \land {\tau_{n}^-}})] - V^m(x,y) \\
     & \geq \E_{x,y}\bigg[\int_0^{t \land \tau_{n}} \big(\lambda^m(Y_s) - \epsilon_{m,K_n} - c^m(X^{\xi}_s,Y_s) \big)ds  \bigg] + \E_{x,y}\bigg[\int_0^{t \land \tau_{n}} V^m_x(X^{\xi}_{s},Y_{s})d\xi^{c}_s\bigg] \\
     & + \E_{x,y}\bigg[ { \sum_{{0 \leq s < t \land \tau_{n}}} \Big(\int_0^{\Delta\xi^{+}_{s}}V^m_x(X^{\xi}_{s-}+z,Y_{s})dz - \int_0^{\Delta\xi^{-}_s}V^m_x(X^{\xi}_{s-}-z,Y_{s})dz \Big) }\bigg].
\end{aligned}
\end{equation}
Now we aim at taking the limit as $m \uparrow \infty$.
Note that, given the continuity of $c$, we have $c^m(x,y) \to c(x,y)$ for any $x \in \R$. This implies $c^m(X^{\xi}_s,Y_s) \to c(X^{\xi}_s,Y_s)$ for every $ s \geq 0$, $\P_{x,y}$-a.s.
Moreover, as { $\lambda \in \dC(\R^d)$, we have $\lambda^m(y) \to \lambda(y)$ for any $y \in \R^d$, which implies $\lambda^m(Y_s) \to \lambda(Y_s)$ for every $s \geq 0$, $\P_{x,y}$-a.s. }
Finally, since the process $(X^{\xi}_{s \land {\tau_{n}^-}},Y_{s \land {\tau_{n}^-}})_{s \geq 0}$ belongs to the compact $K_n$, we can safely invoke \eqref{filtered:verification_thm:approximation_first_order} and the fact that $(\lambda^m)_{m \geq 1}$ are uniformly bounded on compact sets of $\R^d$ to deduce that as $m \uparrow \infty$:
\begin{equation}\label{eq:filtered:thm:optimal_control:inequality_tau_n}
\begin{aligned}
     \E_{x,y} & [V(X^{\xi}_{t \land {\tau_{n}^-}},Y_{t \land {\tau_{n}^-}})] - V(x,y) \\
     & \geq \E_{x,y}\bigg[\int_0^{t \land \tau_{n}} \big(\lambda(Y_s) - c(X^{\xi}_s,Y_s) \big)ds + \int_0^{t \land \tau_{n}} V_x(X^{\xi}_{s},Y_{s})d\xi^{c}_s \bigg]\\
     & \,\,\, + \E_{x,y}\bigg[ {\sum_{{0 \leq s < t \land \tau_{n}}} \Big(\int_0^{\Delta\xi^{+}_{s}}V_x(X^{\xi}_{s-}+z,Y_{s})dz - \int_0^{\Delta\xi^{-}_s}V_x(X^{\xi}_{s-}-z,Y_{s})dz \Big)} \bigg] \\
     & \geq \E_{x,y}\bigg[\int_0^{t \land \tau_{n}}\lambda(Y_s)ds\bigg] - \E_{x,y}\bigg[\int_0^{t \land \tau_{n}} c(X^{\xi}_s,Y_s)ds + K_+\xi^+_{t \land {\tau_{n}^-}} + K_-\xi^-_{t \land {\tau_{n}^-}}  \bigg]
\end{aligned}
\end{equation}
where in the last inequality we used the bounds on the partial derivatives of $V$ given by $\eqref{eq:HJB}$.
We conclude by taking the limit as $n \uparrow \infty$. In order to do that, we first notice that $\tau^n \to \infty$, $\P$-a.s. Upon observing that $\vert V(X^{\xi}_{t \land {\tau_{n}^-}},Y_{t \land {\tau_{n}^-}}) \vert \leq \kappa(1+\vert X^{\xi}_{t \land {\tau_{n}^-}} \vert )$, we can apply the dominated convergence theorem, which (after dividing by $t>0$) yields
\begin{multline*}
     \frac{1}{t}\E_{y}\left[\int_0^t \lambda(Y_s) ds \right] \leq  \frac{1}{t}\E_{x,y}\bigg[\int_0^{t} c(X^{\xi}_s,Y_s) ds +K_+ \xi^{+}_{t } + K_-\xi^{-}_{t} \bigg] \\
     + \frac{1}{t}\bigg(\E_{x,y} [V(X^{\xi}_{t},Y_{t})]  -  V(x,y) \bigg).
\end{multline*}
Given that $\lvert V(x,y) \rvert \leq \kappa(1+\vert x \vert)$ and that, for any admissible $\xi$, one has $\varlimsup_{t \to \infty}\frac{1}{t}\E_{x,y} [ \vert X^{\xi}_{t} \vert] =0$, we take limits as $t \uparrow \infty$ in the equation above and we deduce \eqref{eq:verification:value_geq_equal}.

To obtain the reverse inequality, consider the policy $\xi^\star \in \cB$ as in the statement of the Theorem.
As $V \in \dC^2(\cI)$ and $\cI$ is open, $\cL_{(X^0,Y)}V^m \to \cL_{(X^0,Y)}V$ pointwise on any compact $K \subseteq \cI$, with $\cL_{(X^0,Y)}V^m$ being uniformly bounded on $K$.
By construction, we have both that $\xi^\star$ only activates when $(X^{\xi^\star}_{s-},Y_{s-}) \in \cA_+ \cup \cA_-$ and that $(X^{\xi^\star}_{s\land {\tau_{n}^-}},Y_{s\land {\tau_{n}^-}})_{s \geq 0}$ takes values only in $K_n \cap \cI$.
Therefore, for $\xi = \xi^\star$, we can invoke the dominated convergence theorem in \eqref{eq:filtered:thm:optimal_control:equality_tau_n} to send $m \to \infty$, to get
\begin{equation*}
\begin{aligned}
     \E_{x,y} & [V(X^{\xi^\star}_{t \land {\tau_{n}^-}},Y_{t \land {\tau_{n}^-}})] - V(x,y) \\
     & = \E_{x,y}\bigg[\int_0^{t \land \tau_{n}}\cL_{(X^0,Y)}V(X^{\xi^\star}_s,Y_s)ds + \int_0^{t \land \tau_{n}} V_x(X_{s},Y_{s})d\xi^{\star,c}_s  \\
     & + \sum_{{0 \leq s < t \land \tau_{n}} } \Big(\int_0^{\Delta\xi^{\star,+}_{s}}V_x(X^{\xi^\star}_{s-}+z,Y_{s})dz - \int_0^{\Delta\xi^{\star,-}_s}V_x(X^{\xi^\star}_{s-}-z,Y_{s})dz \Big) \bigg] \\
     & = \E_{x,y}\bigg[\int_0^{t \land \tau_{n}}\lambda(Y_s)ds\bigg] - \E_{x,y}\bigg[\int_0^{t \land \tau_{n}} c(X^{\xi^\star}_s,Y_s)ds + K_+\xi^{\star,+}_{t \land {\tau_{n}^-}} + K_-\xi^{\star,-}_{t \land {\tau_{n}^-}}  \bigg].
\end{aligned}
\end{equation*}
{ Notice that in the last line we relied on the equality $\cL_{(X^0,Y)} V(x,y) + c(x,y)- \lambda(y) = 0 $ for any $(x,y) \in \cI$. 
Indeed, the equality holds true for a.e. $(x,y) \in \cI$ by \eqref{eq:general:equality_almost_everywhere}. Since, by assumption, $V \in \dC^2(\cI)$ and $\lambda \in \dC(\R^d)$, the right-hand side is continuous, and we conclude that the equality holds for any $(x,y) \in \cI$.}
By invoking again the dominated convergence theorem, we take first the limit with respect to $n \to \infty$ and for $t \to \infty$, to deduce 
\begin{multline}
\varlimsup_{T \to \infty}\frac{1}{T}\E_{y}\left[\int_0^T \lambda(Y_t)dt \right]  = \varlimsup_{T \to \infty}\frac{1}{T}\E_{x,y}\bigg[\int_0^{T} c(X^{\xi^\star}_s,Y_s) ds +K_+ \xi^{\star,+}_{T} + K_-\xi^{\star,-}_{T} \bigg] \\
{ \geq} \inf_{\xi \in \cB} \varlimsup_{T \to +\infty} \frac{1}{T} \, \E_{x,y}\left[\int_0^T c(X_t^\xi, Y_t) \, dt + K_+ \xi^+_T + K_- \xi^-_T\right],
\end{multline}
where the last inequality follows since $\xi^\star$ is admissible.
\end{proof}

\begin{remark}\label{rmk:ergodicity}
Setting $\lambda^\star(y) \coloneqq \varlimsup_{T \to \infty}\frac{1}{T}\E_y[\int_0^T \lambda(Y_t) dt]$, Theorem \ref{thm:verification} gives that
\[
\inf_{\xi \in \cB} \varlimsup_{T \to +\infty} \frac{1}{T} \, \E_{x,y}\left[\int_0^T c(X^{\xi}_t, Y_t) \, dt + K_+ \xi^+_T + K_- \xi^-_T\right] = \lambda^\star(y).
\]
This holds even in the case when $Y$ is not recurrent. 
In the case that $Y$ admits a stationary distribution $p^Y_\infty$ on $\R^d$, then
\[
\varlimsup_{T \to \infty}\frac{1}{T}\E_y\left[\int_0^T \lambda(Y_t) dt\right] = \int_{\R^d} \lambda(y) p_\infty^Y(dy) \eqqcolon \lambda^\star,
\]
and thus $\lambda^\star(y) \equiv \lambda^\star$ is the value of the ergodic control problem, independently of $y \in \R^d$.
We stress that this representation is not always true nor necessary, as $Y$ does not need to be ergodic for Theorem \ref{thm:verification} to hold true.
Indeed, provided that the pair $(V,\lambda)$ satisfies appropriate regularity and growth properties, the optimality of $\xi^\star$ and the representation of the optimal payoff in terms of a value profile $\lambda(y)$ follow just from analytical arguments applied to a solution $(V,\lambda)$ of equation \eqref{eq:HJB}.
\end{remark}

\begin{remark}\label{rmk:potential_function}
Even in the case that the pair $(X^{\xi},Y)$, $\xi \in \cB$, is an ergodic process, the PDE in \eqref{eq:HJB} is \emph{not} the dynamic programming equation for the ergodic control problem, which is instead given by
\begin{equation}\label{eq:PDE_potential_value}
\min\{\cL_{(X^0,Y)} \overline{V}(x,y) + c(x,y) - \lambda^\star, -\overline{V}_x(x,y) + K_-, \overline{V}_x(x,y) + K_+\}=0,\,\,(x,y) \in \R \times \R^d,
\end{equation}
where $\lambda^\star$ is the value of the problem.
{ As a consequence, the function $V$ in Theorem~\ref{thm:verification} is \emph{not} the potential function of the problem and $\lambda(y)$ is \textit{not} the value of the problem. In particular, being $\lambda(y)$ dependent on $y$, we can not deduce uniqueness of the value $\lambda^\star$ from a solution to the partial differential equation \eqref{eq:HJB}, so that the value of the control problem \eqref{eq:ctrl_pb} is expected to be dependent on the initial position $y$.
Nevertheless, if the process $Y$ admits a stationary distribution, Theorem \ref{thm:verification} provides a connection between any value profile $\lambda(y)$ and the value of the problem $\lambda^\star$, as noticed in Remark \ref{rmk:ergodicity}.} 
\end{remark}

\begin{remark}
We notice that the admissibility condition \eqref{eq:controls:admissibility} is not new in literature (see, e.g., \cite[Equation (38)]{jack2006ergodic}, \cite[Equation (28)]{lokka2013long}). In these works, the authors prove that any control whose cost functional is finite should satisfy \eqref{eq:controls:admissibility}, by relying on some explicit growth assumptions on the instantaneous cost.
At our level of generality this is not feasible, since our instantaneous cost depends both on the controlled process and on the factor process (which is not required to satisfy any integrability assumption).
The price to pay is to restrict to those strategies that satisfy the growth condition  \eqref{eq:controls:admissibility}.
Alternatively, one could be tempted to consider as admissible those strategies such that $\varlimsup_{T \to \infty}\frac{1}{T}\E[\vert V(X^\xi_T,Y_T) \vert ] = 0$, where $(V,\lambda)$ is a solution to \eqref{eq:HJB}. This would be consistent with other works in literature (see, e.g., \cite[{ Chapter VIII, Equation (4.1)}]{fleming_soner2003controlled} for the discounted setting).
Nevertheless, as in general the solution $(V,\lambda)$ of \eqref{eq:HJB} is not unique (see Theorems \ref{thm:HJB_sol} and \ref{thm:HJB_sol_weak}), this would make the class of controls dependent of the particular pair $(V,\lambda)$.
Therefore, we opt for the admissibility condition \eqref{eq:controls:admissibility}.
\end{remark}


\subsection{The connection to a Dynkin game}
\label{sec:Dynkinconn}

In order to identify the pair $(V,\lambda)$ in Theorem~\ref{thm:verification}, we associate an auxiliary Dynkin game to the original singular stochastic control problem \eqref{eq:ctrl_pb}.
Let $(\widehat{X}, \widehat{Y}) = (\widehat{X}_t, \widehat{Y}_t)_{t \geq 0}$ be given by
\begin{equation}\label{eq:X_hat_0}
    \left\{ \begin{aligned}
    d\widehat{X}_t & = \big(b(\widehat{X}_t, \widehat{Y}_t) + \sigma\sigma_x(\widehat{X}_t, \widehat{Y}_t)\big)dt + \sigma(\widehat{X}_t, \widehat{Y}_t) d W_t , && \widehat{X}_{0} = x \in \R, \\
    d\widehat{Y}_t & = \big(\eta(\widehat{Y}_t) + \sigma_x(\widehat{X}_t, \widehat{Y}_t)(\rho\zeta)(\widehat{Y}_t)\big)dt + \zeta(\widehat{Y}_t)dB_t, && \widehat{Y}_0 = y \in \R^d,
    \end{aligned}  \right.
\end{equation}
whose generator is given by
\begin{multline}\label{eq:generator:X_hat_0}
    \cL_{(\widehat{X},\widehat{Y})}f(x,y) = \big( b(x,y) + (\sigma\sigma_x)(x,y) \big)f_x(x,y) + \frac{1}{2}\sigma^2(x,y)f_{xx}(x,y) \\
     + \big(\eta(y) + \sigma_x(x,y) (\rho \zeta) (y) \big)\nabla_y f(x,y) + \frac{1}{2}\tr\big( a(y)\nabla^2_y f(x,y) \big) + \sigma(x,y)(\rho\zeta)(y) \nabla_y f_{x}(x,y)
\end{multline}
for any $f \in \dC^2(\R \times \R^d)$.
Note that, under Assumptions \ref{assumptions}, $(\widehat{X}^{x,y},\widehat{Y}^{x,y})$ is a strong Markov process.
When needed, we stress the dependence of $(\widehat{X},\widehat{Y})$ { on} the initial point $(x,y) \in \R \times \R^d$ by writing $(\widehat{X}^{x,y},\widehat{Y}^{x,y})$.
With a slight abuse of notation, we write $\E_{x,y}[\cdot]$ to denote the conditional expectation given $\widehat{X}_0 = x$ and $\widehat{Y}_0 = y$.
We introduce the auxiliary Dynkin game through the (upper) value function
\begin{multline}\label{eq:dynkin}
U(x,y) \coloneqq \inf_{\theta \geq 0} \sup_ {\tau \geq 0} \E_{x,y}\left[\int_0^{\tau \land \theta} \exp\left\{\int_0^t b_x(\widehat{X}_s,\widehat{Y}_s) \, ds\right\} c_x(\widehat{X}_t,\widehat{Y}_t) \, dt \right. \\
\left. - \ind_{\tau < \theta} \exp\left\{\int_0^\tau b_x(\widehat{X}_s,\widehat{Y}_s) \, ds\right\} K_+ + \ind_{\theta < \tau} \exp\left\{\int_0^\theta b_x(\widehat{X}_s,\widehat{Y}_s) \, ds\right\} K_- \right],
\end{multline}
where $\tau$, $\theta$ are stopping times of filtration generated by the Brownian motions.
The analysis required to study the Dynkin game \eqref{eq:dynkin} is typically strongly dependent on the structure of the underlying Markov process $(\widehat{X},\widehat{Y})$ and of the instantaneous cost $c_x(x,y)$.
Therefore, at this stage, we limit ourselves to assume that the value function $U(x,y)$ satisfies the following properties, from which we are going to characterize an optimal control for the ergodic problem \eqref{eq:ctrl_pb}.
\begin{hypothesis}\label{conj:main}
There exist two measurable functions $a_+, a_- \colon \R^d \to \R$ such that:
\begin{enumerate}[label=(\Roman*),wide]
    \item  $\sup_{y \in \R^d}a_+(y) < \inf_{y \in \R^d}a_-(y)$.
    \item It holds
    \begin{align}
        c_x(x,y) + K_- b_x(x,y) \geq 0, \quad \forall (x,y) \text{ s.t. } x \geq a_-(y), \label{eq:main:sign_S_-} \\
        c_x(x,y) -K_+ b_x(x,y) \leq 0, \quad \forall (x,y) \text{ s.t. } x \leq a_+(y). \label{eq:main:sign_S_+}
    \end{align}
    \item Setting $\cC \coloneqq \{ (x,y) \in \R \times \R^d: \, a_+(y) < x < a_-(y) \}$, $U \in \dC^2(\cC) \cap \dC(\R \times \R^d)$ and it solves the free-boundary problem
    \begin{equation}\label{eq:U_var_ineq}
    \left\{ \begin{aligned}
    \cL_{(\widehat{X},\widehat{Y})}U(x,y) + c_x(x,y) + b_x(x,y)U(x,y) & = 0, && \text{if } a_+(y) < x < a_-(y), \\
    U(x,y) & = -K_+, && \text{if } x \leq a_+(y), \\
    U(x,y) & = K_-, && \text{if } x \geq a_-(y).
    \end{aligned} \right.
    \end{equation}
    for all $(x,y) \in \R \times \R^d$.
\end{enumerate}
\end{hypothesis}
Hypothesis \ref{conj:main} should be read in the following way:
the Dynkin game \eqref{eq:dynkin} should have a saddle point $(\tau^*,\theta^*)$, { given by the entry times in the sets $\cS_+ \coloneqq \{(x,y) \in \R \times \R^d: \, x \leq a_+(y)\}$ and $\cS_- \coloneqq \{(x,y) \in \R \times \R^d: \, x \geq a_-(y)\}$ respectively}, so that the functions $a_\pm$ are expected to be the boundaries of the continuation and stopping regions.
If one could rely on the semi-harmonic characterization of the Dynkin game's value function \eqref{eq:U_var_ineq}, property (II) would follow.
{ We notice that property (II) is not uncommon in ergodic singular control problems. In particular, in one-dimensional ergodic control problems, similar conditions are imposed on the problem data to ensure that the optimal policy is of barrier type (see, e.g., \cite[Equation (20)]{lokka2013long}, \cite[Assumption 2.7(ii)]{liang2020ergodic}, \cite[Assumption 5(ii)]{cao2023stationary}).}
Finally, property (III) would tell that $U$ is a classical solution of the pointwise free-boundary problem \eqref{eq:U_var_ineq}.
In following Section \ref{sec:application} we will provide two examples in which Hypothesis \ref{conj:main} is satisfied.

\smallskip
We can heuristically derive the differential problem \eqref{eq:U_var_ineq} for $U$ starting from a pseudo-potential function $V$, in the following way:
Suppose that we are given a pair $(V,\lambda)$ solution to \eqref{eq:HJB} and that there exist two measurable functions $a_\pm:\R^d \to \R$ so that the set $\cI$ in \eqref{eq:HJB:waiting_exert} can be expressed as $\{(x,y): \, a_+(y) < x < a_-(y)\}$.
If $V$ is regular enough, then $V_x(x,y)$ satisfies \eqref{eq:U_var_ineq} with $\cC = \cI$ and $\cS_\pm = \cA_\pm$.
The state constraint in the inaction regions $\cA_+$ and $\cA_-$ are straightforward to see; as for the behavior of $V_x$ in the action regions $\cI$, it is enough to take the $x$ derivative of the term $\cL_{(X^0,Y)}V + c { - \lambda} = 0$ to see that $V_x$ should satisfy the PDE in \eqref{eq:U_var_ineq}.
This justifies both the presence of the new discount term $b_x(x,y)$ in \eqref{eq:dynkin} and \eqref{eq:U_var_ineq} and the disappearance of the value profile $\lambda(y)$.

In the following Theorems \ref{thm:HJB_sol} and \ref{thm:HJB_sol_weak}, we revert this reasoning: we rigorously build a solution $(V,\lambda)$ to \eqref{eq:HJB} starting from the value function $U$ of the Dynkin game \eqref{eq:dynkin}, provided that Hypothesis \ref{conj:main} holds true and $U$ enjoys some additional regularity.

\begin{theorem}\label{thm:HJB_sol}
Let $U$ be the value function of the auxiliary Dynkin game \eqref{eq:dynkin}, and suppose that Hypothesis \ref{conj:main} holds true.
Moreover, suppose that $U \in \bbW^{2,\infty}_{loc}(\R \times \R^d)$.
Let $\alpha \in (\sup_y a_+(y), \inf_y a_-(y))$ and set
\begin{align}
V(x,y) & \coloneqq \int_\alpha^x U(x',y) dx', \label{eq:main:built_potential}  \\
\lambda(y) & \coloneqq c(\alpha,y) + U(\alpha,y) b(\alpha,y) + \frac{1}{2}\sigma^2(\alpha,y) U_{x}(\alpha,y) + \sigma(\alpha,y)(\rho\zeta)(y)U_y(\alpha,y). \label{eq:main:built_value}
\end{align}
Then, { $(V,\lambda)$ is a solution to~\eqref{eq:HJB}.}
In addition, suppose that $V \in \dC^2(\cC)$, $(\lambda(Y_t))_{t \in [0,T]}$ is $d\P \otimes dt$ integrable for any $T > 0$, and that there exists $\xi^\star\in \cB$ such that, for almost all $t \geq 0$,
\begin{multline}\label{eq:main:optimal_control}
(X_t^{\xi^\star}, Y_t) \in \{(x,y) \colon a_+(y) \leq x \leq a_-(y)\}, \\
\xi^{\star, +}_t = \int_{[0,t]} \ind_{X_s^{\xi^\star} \leq a_+(Y_s) } \, d\xi^{\star, +}_s, \quad \xi^{\star, -}_t = \int_{[0,t]} \ind_{X_s^{\xi^\star} \geq a_-(Y_s)} \, d\xi^{\star, -}_s,
\end{multline}
$\P_{x,y}$-a.s. { and $\P_{x,y}((X^{\xi^\star}_t,Y_t) \in \cC) = 1$ for any $t \geq 0$}. Then, $\xi^\star$ is optimal and it holds 
\begin{equation}\label{eq:HJB_sol:value_representation}
\varlimsup_{T \to \infty}\frac{1}{T}\E_{y}\left[\int_0^T \lambda(Y_t)dt \right] = \inf_{\xi \in \cB} \varlimsup_{T \to +\infty} \frac{1}{T} \, \E_{x,y}\left[\int_0^T c(X_t^\xi, Y_t) \, dt + K_+ \xi^+_T + K_- \xi^-_T\right].
\end{equation}
\end{theorem}
\begin{proof}
We notice that, by choosing either $\theta = 0$ or $\tau = 0$ in the functional appearing in~\eqref{eq:dynkin}, we have $-K_+ \leq U(x,y) \leq K_-$ for all $(x,y) \in \R \times \R^d$, so that $V$ has at most linear growth in $x$ uniformly in $y$.
We notice that $V \in \bbW^{2,\infty}_{loc}(\R \times \R^d)$ as so does $U$.
By construction, { $V_x(x,y) = U(x,y)$, $V_{xx}(x,y) = U_{x}(x,y)$, $V_{x y_i}(x,y) = U_{y_i}(x,y)$, $V_{y_i}(x,y) = \int_{\alpha}^x U_{y_i}(x',y)dx'$ and $V_{y_i y_j}(x,y) = \int_{\alpha}^x U_{y_i y_j}(x',y)dx'$ for a.e. $(x,y) \in \R \times \R^d$, where derivatives are to be understood in the weak sense}.
Fix now $(x,y) \in \R \times \R^d$.
{ As every term admits at least a weak derivative, Lebesgue differentiation theorem yield}
\begin{align*}
& \cL_{(X^0,Y)} V(x,y) + c(x,y) \\
& = \frac{1}{2} \sigma^2(x,y) U_x(x,y) + b(x,y) U(x,y) + \sigma(x,y)(\rho \zeta)(y)\nabla_{y}U(x,y)  \\
& \;\;\; + \eta(y)\nabla_y V(x,y) + \frac{1}{2}\tr\big( a(y) \nabla^2_y V(x,y) \big) + c(x,y) \\
& \, = \int_{\alpha}^x  \dfrac{\partial}{\partial z} \left[\frac{1}{2} \sigma^2(z,y) U_x(z,y) + b(z,y) U(z,y) + \sigma(z,y)(\rho \zeta)(y)\nabla_{y}U(z,y) + c(z,y)   \right] dz \\
& \;\;\; + \int_{\alpha}^x \left( \eta(y)\nabla_{y} U(z,y) + \frac{1}{2}\tr\big( a(y) \nabla^2_y U(z,y)\big) \right)dz \\
& \;\;\; + U(\alpha,y) b(\alpha,y) + \frac{1}{2}\sigma^2(\alpha,y) U_{x}(\alpha,y) + \sigma(\alpha,y)(\rho\zeta)(y)\nabla_{y} U(\alpha,y) + c(\alpha,y) \\
& \, = \int_{\alpha}^x  \bigg(\cL_{(\widehat{X},\widehat{Y})}U(z,y) + b_x(z,y)U(z,y) + c_x(z,y) \bigg) dz + \lambda(y),
\end{align*}
where in the last line we used the definition of $\lambda(y)$ in \eqref{eq:main:built_value}.
Thus, we get
\begin{equation}\label{eq:thm:HJB_sol:differential_identity}
\begin{aligned}
\cL_{(X^0,Y)} & V(x,y) + c(x,y) - \lambda(y) \\
& = - \int_{a_+(y)}^{\alpha} \left(\cL_{(\widehat{X},\widehat{Y})}U(z,y) + b_x(z,y)U(z,y) + c_x(z,y)\right) dz\\
& \; + \underbrace{\int_{a_+(y)}^x \left(\cL_{(\widehat{X},\widehat{Y})}U(z,y) + b_x(z,y)U(z,y) + c_x(z,y)\right) dz}_{\eqqcolon I(x,y)} = I(x,y),
\end{aligned}
\end{equation}
as the first integral is equal to $0$ for any $y \in \R^d$, using \eqref{eq:U_var_ineq} upon noticing that $[a_+(y) , \alpha] \subseteq [a_+(y), a_-(y)]$ for { any} $y \in \R^d$.
By the same reasoning,  $I(x,y) = 0$ for all $(x,y)$ such that $a_+(y) \leq x \leq a_-(y)$.
Next, let us consider $x \geq a_-(y)$.
As $U(x,y) \equiv K_-$ for any $x \geq a_-(y)$, \eqref{eq:main:sign_S_-} implies that $U(x,y)$ verifies
\begin{equation}\label{eq:subm}
\cL_{(\widehat{X},\widehat{Y})}U(x,y) + c_x(x,y) + b_x(x,y)U(x,y) \geq 0, \quad \text{for all } (x,y) \in \R \times \R^d \text{ s.t. } x \geq a_-(y).
\end{equation}
Then, equation \eqref{eq:subm} implies that $I(x,y) \geq 0$, if { $x \geq a_-(y)$}.
Analogously, if $x \leq a_+(y)$, then
\eqref{eq:main:sign_S_+} implies that $U(x,y)$ verifies
\begin{equation}\label{eq:supm}
\cL_{(\widehat{X},\widehat{Y})}U(x,y) + c_x(x,y) + b_x(x,y)U(x,y) \leq 0, \quad \text{for all } (x,y) \in \R \times \R^d \text{ s.t. } x \leq a_+(y).
\end{equation}
Equation \eqref{eq:supm} then entails that $I(x,y) \geq 0$, if $x \leq a_+(y)$.
Therefore,
\begin{equation*}
\cL_{(X^0,Y)} V(x,y) + c(x,y) - \lambda(y) \geq 0, \quad \text{for all } (x,y) \in \R \times \R^d.
\end{equation*}
As $V_x(x,y) = U(x,y)$, \eqref{eq:U_var_ineq} gives that the derivative constraints
\begin{equation*}
V_x(x,y) + K_+ \geq 0, \quad K_- - V_x(x,y) \geq 0, \quad \text{for all } (x,y) \in \R \times \R^d,
\end{equation*}
are satisfied. Hence, $(V,\lambda)$ solves \eqref{eq:HJB}.
Moreover, we notice that the sets $\cI$, $\cA_+$ and $\cA_-$ defined by \eqref{eq:HJB:waiting_exert} are given by $\cC$, $\cS_+$ and $\cS_-$ respectively.
To conclude, we note that the assumptions of Theorem \ref{thm:verification} are verified.
By assumption, $V \in \dC^2(\cC)$, $(\lambda(Y_t))_{0 \leq t \leq T}$ is $d\P\otimes dt$ integrable, $\xi^\star$ is admissible.
Moreover, as already noticed, $V(x,y)$ has at most linear growth in $x$ uniformly in $y$ and belongs to $\bbW^{2,\infty}_{loc}(\R \times \R^d)$.
Finally, $y \mapsto \lambda(y)$ is continuous, as $U \in \dC^1(\R \times \R^d)$ by Sobolev embedding, and so it is locally bounded.
Thus, Theorem \ref{thm:verification} yields the optimality of the control $\xi^\star$ and the representation \eqref{eq:HJB_sol:value_representation}.
\end{proof}

\begin{remark}\label{rmk:why_no_potential}
The reader may wonder why this argument does not work directly for a pair $(\overline{V},\lambda^\star)$ solution to \eqref{eq:PDE_potential_value}, since the derivative $\overline{V}_x$ of the potential function is expected to satisfy the free-boundary problem \eqref{eq:U_var_ineq} as well.
The reason is given by the proof of Theorem \ref{thm:HJB_sol} itself, as relating $\cL_{(X^0,Y)}V + c$ and $\cL_{(\widehat{X},\widehat{Y})}U +b_x U + c_x$ leaves us with a reminder term $\lambda(y)$, dependent on $y \in \R^d$, instead of a constant $\lambda \in \R$, which implies that the couple $(V(x,y),\lambda(y))$ (with $V$ being given by \eqref{eq:main:built_potential}) satisfies the auxiliary \eqref{eq:HJB} rather than the dynamic programming equation \eqref{eq:PDE_potential_value}.
\end{remark}

In many interesting cases, it is not possible to verify $U \in \bbW^{2,\infty}_{loc}(\R \times \R^d)$.
Therefore, the theorem above is not directly applicable.
Nevertheless, we can state a weaker version of Theorem \ref{thm:HJB_sol}, which bypasses the need of having locally bounded second order derivatives.
\begin{theorem}\label{thm:HJB_sol_weak}
Let $U$ be the value function of the auxiliary Dynkin game \eqref{eq:dynkin}, and suppose that Hypothesis~\ref{conj:main} holds true.
Moreover, suppose that $U$ is so that $\cL_{(\widehat{X},\widehat{Y})} U \in L^\infty_{loc}(\R \times \R^d)$.
Let $\alpha \in (\sup_y a_+(y), \inf_y a_-(y))$ and set
\begin{align}
V(x,y) & \coloneqq \int_\alpha^x U(x',y) dx', \label{eq:main:built_potential_weak}\\
\lambda(y) & \coloneqq c(\alpha,y) + U(\alpha,y) b(\alpha,y) + \frac{1}{2}\sigma^2(\alpha,y) U_{x}(\alpha,y) + \sigma(\alpha,y)(\rho\zeta)(y)U_y(\alpha,y). \label{eq:main:built_value_weak}
\end{align}
Suppose that $V \in \bbW^{2,\infty}_{loc}(\R \times \R^d)$, $\lambda \in L^\infty_{loc}(\R^d)$ and that it holds
\begin{multline}\label{thm:HJB_sol_weak:differential_identity}
    \cL_{(X^0,Y)} V(x,y) + c(x,y) { - } \lambda(y) \\
    = \int_{a_+(y)}^x \Big(\cL_{(\widehat{X},\widehat{Y})} U(x',y) + b_x(x',y) U(x',y) + c_x(x',y) \Big)dx'.
\end{multline}
Then, $(V, \lambda)$ is a solution to ~\eqref{eq:HJB}.
In addition, suppose that $V \in \dC^2(\cC)$, { $\lambda \in \dC(\R^d)$}, $(\lambda(Y_t))_{t \in [0,T]}$ is $d\P \otimes dt$ integrable for any $T > 0$, and that there exists $\xi^\star \in \cB$ such that, for almost all $t \geq 0$,
\begin{multline}\label{eq:thm:HJB_sol_weak:optimal_ctrl}
(X_t^{\xi^\star}, Y_t) \in \{(x,y) \colon a_+(y) \leq x \leq a_-(y) \}, \\
\xi^{\star, +}_t = \int_{[0,t]} \ind_{X_s^{\xi^\star} \leq a_+(Y_s) } \, d\xi^{\star, +}_s, \quad \xi^{\star, -}_t = \int_{[0,t]} \ind_{X_s^{\xi^\star} \geq a_-(Y_s)} \, d\xi^{\star, -}_s,
\end{multline}
$\P_{x,y}$-a.s. { and $\P_{x,y}((X^{\xi^\star}_t,Y_t) \in \cC) = 1$ for any $t \geq 0$}. Then, $\xi^\star$ is optimal and it holds 
\begin{equation}\label{eq:HJB_sol_weak:value_representation}
\varlimsup_{T \to \infty}\frac{1}{T}\E_{y}\left[\int_0^T \lambda(Y_t)dt \right] = \inf_{\xi \in \cB} \varlimsup_{T \to +\infty} \frac{1}{T} \, \E_{x,y}\left[\int_0^T c(X_t^\xi, Y_t) \, dt + K_+ \xi^+_T + K_- \xi^-_T\right].
\end{equation}
\end{theorem}
The proof is completely analogous to the proof of Theorem \ref{thm:HJB_sol}, once noticed that the { right-hand} side of \eqref{thm:HJB_sol_weak:differential_identity} is exactly the term $I(x,y)$ defined in \eqref{eq:thm:HJB_sol:differential_identity}, and it is therefore omitted.
We just notice that, in the lower regularity framework of Theorem \ref{thm:HJB_sol_weak}, we cannot obtain the regularity of $V$ from the one of $U$, so that equation \eqref{eq:thm:HJB_sol:differential_identity} cannot be inferred.
A situation like this one may occur in many different situations. We provide a relevant example in Section \ref{sec:application:partial_observation}.

\subsection{Remarks on Theorems \ref{thm:HJB_sol} and \ref{thm:HJB_sol_weak}}\label{sec:rmks_thms}
We make a couple of remarks on Theorems \ref{thm:HJB_sol} and \ref{thm:HJB_sol_weak}.
According to their statements, for any $\alpha$ in between the free-boundaries $a_+$ and $a_-$, there exist a different solution $(V,\lambda)$ to \eqref{eq:HJB}.
This is not a surprise, as equation \eqref{eq:HJB} is not the dynamic programming equation for the potential, but instead a PDE for a pseudo-potential function, as already noticed in Remark \ref{rmk:potential_function}.
Nevertheless, the optimal control $\xi^\star$ identified by Theorems \ref{thm:HJB_sol} and \ref{thm:HJB_sol_weak} is the same for any pair $(V,\lambda)$: it is enough to notice that its definition relies only on the free-boundaries of the variational inequality satisfied by the value function $U(x,y)$; thus, it is independent of the particular value $\alpha$ in the definition of $(V,\lambda)$.
This observation has important consequences.
Denote by $\lambda(y;\alpha)$ the value profile given by \eqref{eq:main:built_value} associated to $\alpha \in (\sup a_+(y),\inf a_-(y))$, in order to highlight the dependence on the level $\alpha$.
As the optimal control $\xi^\star$ given by \eqref{eq:main:optimal_control} is the same for any $\alpha$, the verification Theorem \ref{thm:verification} and, in particular, equations \eqref{eq:verification:value_geq_equal} and \eqref{eq:verification:value_leq_equal}), imply that
\begin{equation}\label{eq:main:constant_value}
    \inf_{\xi \in \cB} \varlimsup_{T \to +\infty} \frac{1}{T} \, \E_{x,y}\left[\int_0^T c(X_t^\xi, Y_t) \, dt + K_+ \xi^+_T + K_- \xi^-_T\right] = \varlimsup_{T \to \infty} \frac{1}{T}\E_{y}\left[\int_0^T \lambda(Y_t;\alpha) dt \right]
\end{equation}
for any $\alpha \in (\sup_{y \in \R^d} a_+(y),\inf_{y \in \R^d} a_-(y))$.

It is possible to characterize the difference between two value profiles in terms of the derivatives with respect to $y$ of the value function $U$.
Indeed, let $\sup a_+(y) \leq \alpha_1 < \alpha_2 \leq \inf a_-(y)$.
Upon noticing that the map $\alpha \mapsto \lambda(y;\alpha)$ is continuously differentiable  for any fixed $y$, standard computations yield
\begin{equation}\label{eq:main:lambda_difference}
\begin{aligned}
    \lambda & (y;\alpha_2)-\lambda(y;\alpha_1) \\
    & = \int_{\alpha_1}^{\alpha_2} \dfrac{\partial}{\partial z} \left[ c(z,y) + U(z,y) b(z,y) + \frac{1}{2}\sigma^2(z,y) U_{x}(z,y) + \sigma(z,y)(\rho\zeta)(y)U_y(z,y)  \right] dz \\
    & = - \int_{\alpha_1}^{\alpha_2} \left( \eta(y)\nabla_{y}U(z,y) +\frac{1}{2}\tr\big(a(y)\nabla^2_{y}U(z,y)\big)\right) dz,
\end{aligned}    
\end{equation}
where we have added and subtracted $\eta(y)\nabla_{y}U(z,y) +\frac{1}{2}\tr\big(a(y)\nabla^2_{y}U(z,y)\big)$ and used the fact that $(z,y) \in \cC$ for any $\alpha_1 \leq z \leq \alpha_2$.
Thus, the existence of different value profiles is only due to the presence of the factor process $Y$.

\smallskip
Equations \eqref{eq:main:constant_value} and \eqref{eq:main:lambda_difference} allow to recover the usual representation of the value of the problem in terms of the boundaries of the inaction region for one-dimensional singular control problems.
Indeed, suppose that $X^0$ is an ergodic diffusion and that the coefficients $\eta$ and $\zeta$ in \eqref{eq:SDE} are constantly null.
This amounts to consider the deterministic constant process $Y_t = y$ appearing in the coefficients of $X^\xi$ and in the instantaneous cost as a fixed deterministic parameter.
In particular, this implies that the free-boundaries $a_\pm(y)$ depend on $y$ only parametrically, so that we can define $(V,\lambda)$ in \eqref{eq:main:built_potential} and \eqref{eq:main:built_value} by choosing any $\alpha \in (a_+(y),a_-(y))$.
Equation \eqref{eq:main:lambda_difference} then implies
\[
\lambda(y;\alpha_1) = \lambda(y;\alpha_2) \quad \forall a_+(y) \leq \alpha_1 \leq \alpha_2 \leq a_-(y).
\]
By continuity of $\lambda(y;\alpha)$ with respect to $\alpha$, one gets 
\begin{equation}\label{eq:main:relation_1d}
c(a_-(y),y) + K_-b(a_-(y),y) = \lim_{\alpha \uparrow a_-}\lambda(y;\alpha) = \lim_{\alpha \downarrow a_+}\lambda(y;\alpha) = c(a_+(y),y) - K_+b(a_+(y),y),
\end{equation}
where we used the fact that $U(a_\pm(y),y)=\mp K_\pm$ and $U_x(a_\pm(y),y) = 0$ for any $y \in \R^d$.
Equation \eqref{eq:main:relation_1d} is well-known in literature, as it gives a relation between the value of the problem $\lambda(y)$ and the free-boundaries $a_\pm(y)$, to be used in the the guess-and-verify approach to impose the smooth-fit condition on the potential function (see e.g. \cite[pp. 5-6]{alvarez2018stationary} and in particular equation (2.16) therein).

\section{Applications to inventory models}\label{sec:application}

As applications of our results, we consider two mean-reverting inventory models with stochastic mean-reversion level.
In the following, $X=(X^{\xi}_t)_{t \geq 0}$ is the net inventory process, which captures the difference between regular and customer demands.
The firm controls its inventory level by a process of bounded variation $\xi=(\xi^+,\xi^-) \in \cB$, possibly subject to further restrictions.
Controls are exercised on the net inventory process to maintain the inventory at desired positions.
The dynamics of the inventory process will be given by
\begin{equation}\label{eq:application:inventory}
dX^{\xi}_t = (Y_t - \delta X_t)dt + \sigma dW_t + d\xi^{+}_t - d\xi^{-}_t, \quad X^\xi_0 = x.
\end{equation}
Here, $\delta > 0$ is the depreciation rate and $Y=(Y_t)_{t \geq 0}$ is the mean-reversion level, which we will be subject to further specifications.
We note that positive and negative net-inventory levels mean inventory and backlog, respectively.

Brownian models for the (net) inventory process are nowadays classical; see, e.g., the review \cite{perera2023inventory_models} and the references therein.
{ The behavior of the (net) inventory process is assumed to be mean-reverting as in \cite{cadenillas2010optimal,hu_liu2018mean_reverting,liu_bensoussan2019ergodic_inventory}. In particular, mean-reversion can be thought of as an effect of deterioration (see \cite{goyal2001recent,raafat1991survey}).}
Finally, we note that the ergodic optimization criterion has nowadays a long history in the inventory management literature (among others, see \cite{dai2013brownian_inventory,helmes2017continuous,helmes2018weak_inventory,liu_bensoussan2019ergodic_inventory,taksar1985average,yao2015optimal_brownian_inventory}).

The two models we consider below in Section \ref{sec:application:partial_observation} and \ref{sec:application:full_observation} differ both in the dynamics of mean reversion process $Y$, as well as in the information available to the firm.
The model of Section \ref{sec:application:partial_observation} features partial information, which we address by studying the associated full-information (separated) problem.
This will lead us to deal with a degenerate state process, in that the generator of the underlying state process $(\widehat{X},\widehat{Y})$ \eqref{eq:X_hat_0} of the auxiliary Dynkin game is not uniformly elliptic.
As a consequence, we will not be able to deduce the $\bbW^{2,\infty}_{loc}$ regularity of the Dynkin game value function, and hence we will rely on Theorem \ref{thm:HJB_sol_weak} to find an optimal control.
In particular, we will show that the assumptions of Theorem \ref{thm:HJB_sol_weak} are satisfied by introducing a proper transformation of the Dynkin game's value function, which will be used to show the regularity of the pseudo-potential function $V$ in \eqref{eq:main:built_potential_weak}.

On the contrary, in Section \ref{sec:application:full_observation}, we deal with a non-degenerate problem, in the sense that the generator of the state process $(\widehat{X},\widehat{Y})$ of the Dynkin game is uniformly-elliptic.
By relying on the results of \cite{chiarolla2000inflation}, we will show that $U \in \bbW^{2,\infty}_{loc}$, so that Theorem \ref{thm:HJB_sol} can be applied to build a solution $(V,\lambda)$ of equation \eqref{eq:HJB} and an optimal control $\xi^\star$ { satisfying} \eqref{eq:main:optimal_control}.

\subsection{Inventory control with partially observable mean-reversion level}\label{sec:application:partial_observation}
We here consider an inventory management problem with unknown { mean-reversion level}.
The setting is closely related to the problem investigated in \cite{federico2023inventory_unknown_demand}, with two major differences.
Differently from \cite{federico2023inventory_unknown_demand}, the drift of the inventory process features an unobservable two-dimensional Markov chain instead of an unobservable Bernoulli random variable, implying that the demand is not only unobservable but it changes over time, at unobservable jump times.
Most importantly, instead of working under the discounted optimization criterion, we work with the ergodic one.

\smallskip
Let $(\Omega,\cF,\bbF = (\cF_t)_{t \geq 0},\P)$ be a complete filtered probability space capturing all the uncertainty of our setting. $\bbF$ denotes the \textit{full information} filtration.
Consider a real-valued $\bbF$-Brownian motion $W$, and let $\varepsilon = (\varepsilon_t)_{t \geq 0}$ be a two-state Markov chain, with state space $\cE = \{ 1,2 \}$ and rate transition matrix (or \mbox{$Q$-matrix}, see, e.g.,~\citep{norris:markovchains}) $\begin{bmatrix} -\lambda_1 & \lambda_1 \\ \lambda_2 & -\lambda_2\end{bmatrix}$, where $\lambda_1, \, \lambda_2 > 0$. In particular, this implies (cf.~\citep[Theorem~2.8.2]{norris:markovchains}) that, for all $t \geq 0$,
\begin{equation*}
    \P(\varepsilon_{t + \Delta t} = j \vert \varepsilon_t = i) = \left\{ \begin{aligned}
        & \lambda_i \Delta t + o(\Delta t), && j \neq i, \\
        & 1 - \lambda_i \Delta t + o(\Delta t), && j = i,
    \end{aligned} \right.
\end{equation*}
as $\Delta t \to 0$, uniformly in $t$. We assume that $\varepsilon$ and $W$ are independent.

\smallskip
Let $m: \cE \to \R$ be equal to $m_1$ if $j = 1$ and $m_2$ if $j = 2$, with $m_1 > m_2$.
Referring to \eqref{eq:application:inventory}, we suppose that, for any $\xi=(\xi^+,\xi^-) \in \cB$, $Y_t = m_{\varepsilon_t}$, so that the inventory process follows the dynamics
\begin{equation}\label{application:dynamics}
    d X^{\xi}_t = (m_{\varepsilon_t} - \delta X^{\xi}_t)dt + \sigma dW_t + d\xi^+_t - d\xi^-_t, \quad X^{\xi}_{0-}=x,
\end{equation}
where $\delta$ and $\sigma$ are positive constants.

Denote by $X^0$ the solution of \eqref{application:dynamics} with control $\xi \equiv 0$.
The \textit{information filtration} is given by the completed natural filtration $\bbF^{X^0}$ of the process $X^0$.
The set $\cB$ of admissible controls for the control problem under partial information is defined as the set of singular controls $\xi$ so that $\xi$ is $\bbF^{X^0}$-adapted.
Notice that the information filtration $\bbF^{X^0}$ is strictly included in the full information filtration $\bbF$. In particular, the values of $\varepsilon$ are not observable, and they can just be inferred from the observation of $X^0$.

Let $c:\R \to \R$ be a continuous function and let $K_+, \, K_- >0$.
Then, the problem is to determine
\begin{equation}\label{eq:application:value}
    \inf_{{\xi \in \cB}} \varlimsup_{T \to \infty}\frac{1}{T}\E_{x,y}\left[\int_0^T c(X^{\xi}_t)dt +K_+\xi^+_T +K_-\xi^-_T \right],
\end{equation}
where $\E_{x,y}[\cdot]$ denotes the expectation under $\P_{x,y}(\cdot)$, which is the probability measure on $(\Omega, \cF)$ under which $X_{0^-}^\xi = x$ and $\epsilon_0$ has distribution $\begin{bmatrix} y \\ 1-y \end{bmatrix}$, for $(x,y) \in \R \times [0,1]$. More precisely, $\bbP_{x,y}(\epsilon_0 = 1) = y$.

In Section \ref{sec:application:partial_observation:filtering}, we will rely on classical results from filtering theory (see, e.g., \cite{bensoussan1992partially_observed,liptser2001statistics1}) in order to reduce the problem to an equivalent two-dimensional full information stochastic singular control problem, so that we are back in the setting of Section \ref{sec:model}.
Then, we will solve the equivalent full information problem in Section \ref{sec:application:partial_observation:filtered_problem}.

\smallskip
We make the following assumptions on the instantaneous cost function $c$:
\begin{assumption}\label{application:assumptions_cost}
The instantaneous cost function $c : \R \to \R$ belongs to { $\dC^\infty(\R)$} and it is strictly convex.
Moreover, there exist constants $p \geq 2$, $\alpha_0,\alpha_1,\alpha_2 > 0$ so that, for any $x \in \R$, it holds
\begin{equation}\label{eq:cost_function:growth_derivative}
\begin{aligned}
    & 0 \leq c(x) \leq \alpha_0(1+\vert x \vert^p), \\
    & \vert c'(x) \vert  \leq \alpha_1(1+\vert x \vert^{p-1}), \\
    & \vert c''(x) \vert  \leq \alpha_2(1+\vert x \vert^{p-2}).
\end{aligned}
\end{equation}
Finally, $\lim_{x \to \pm \infty} c'(x) = \pm\infty$.
\end{assumption}
{ The property $c \in \dC^\infty(\R)$ will be later needed in order to apply the regularity results of \cite{peskir2025weak} in Lemma \ref{filtered:lemma:weak_solution}.}
Notice that Assumption \ref{application:assumptions_cost} include the benchmark case of the quadratic cost $c(x)=(x-\bar{x})^2$, where $\bar{x} \in \R$ is a fixed constant, target level of the inventory.

\subsubsection{Derivation of the separated problem}\label{sec:application:partial_observation:filtering}
In this section, we derive the so-called separated problem (for a reference, see the pioneer work \cite{wonham1968separation}) for the partial information singular control problem.
In this way, we reduce Problem \eqref{eq:application:value} to a complete information setting.

\smallskip
We consider the \textit{filter} $\Pi = (\Pi_t)_{t \geq 0}$, i.e. the process defined by
\[
\Pi_t = \P( \varepsilon_t = 1 \vert \cF^{X^0}_t ) = \E[ \ind_{ \{ \varepsilon_t = 1 \} } \vert \cF^{X^0}_t].
\]
It is well-known that $\Pi$ provides the best mean-square estimate of the law of $\varepsilon$ on the basis of the observation filtration $\bbF^{X^0}$. That is, for any $f:\cE \to \R$, it holds
\[
\E[f_{\varepsilon_t} \vert \cF^{X^0}_t] = \Pi_t f_1 + (1-\Pi_t)f_2 \eqqcolon f(\Pi_t).
\]
By \cite[Exercise 3.27]{bain_crisan2009filtering} (see also \cite[Theorem 9.1]{liptser2001statistics1}), $\Pi$ is the unique strong solution of the following stochastic differential equations:
\[
    d\Pi_t = (\lambda_2 - (\lambda_1 + \lambda_2)\Pi_t) dt + \frac{m_1 - m_2}{\sigma}\Pi_t(1-\Pi_t)d B_t, \quad \Pi_{0} = y,
\]
where $y=\P(\varepsilon_0 = 1) \in [0,1]$ is the prior belief on the initial state of the process and $B = (B_t)_{t \geq 0}$ is the \textit{innovation process}, defined by
\begin{equation}\label{eq:filtered:innovation_process}
    B_t \coloneqq W_t - \int_0^t \sigma^{-1}\left( m(\Pi_s) - m_{\varepsilon_s} \right)ds, \quad t \geq 0.
\end{equation}
In particular, $B$ is an $\bbF^{X^0}$-Brownian motion.
In the following, we set $\gamma \coloneqq \frac{m_1 - m_2}{\sigma}$, which is strictly positive as $m_1 > m_2$ by assumption.
The process $X^0$ is an It\^{o}-process with respect to the innovation process so that the dynamics of the pair $(X^0,\Pi)$ are coupled by the equations
\begin{equation}\label{eq:filtering:filtered_dynamics}
    \left\{ \begin{aligned}
    dX^{0}_t & =  (m(\Pi_t) -\delta X^{0}_t )dt + \sigma d B_t, && X^0_{0} = x, \\
    d\Pi_t & = (\lambda_2 - (\lambda_1 + \lambda_2)\Pi_t) dt + \gamma\Pi_t(1-\Pi_t)dB_t, && \Pi_{0} = y.
    \end{aligned}  \right.
\end{equation}
As $\Pi$ is bounded by definition, the system of stochastic differential equations \eqref{eq:filtering:filtered_dynamics} admits a pathwise unique solution.
Moreover, by continuously extending the coefficients of the diffusion $\Pi$ to constants outside of $[0,1]$, $\Pi$ can be regarded as the solution of a stochastic differential equation with bounded Lipschitz coefficients, so that {\cite[Theorem~5.4.20]{karatzas_shreve}} implies that the pair $(X^0,\Pi)$ is strong Markov.
We regard the state space of $(X^0,\Pi)$ to be $\R \times (0,1)$. Indeed, notice that the boundary points $0$ and $1$ are entrance-not-exit for the process $\Pi$, as can be shown by applying Feller's test for explosion (see \cite[Section 5.5]{karatzas_shreve}).
In other words, $\Pi$ can start from $y=0$ or $y=1$, but it cannot reach any of these points in finite time.
Hence, in our subsequent analysis, we shall exclude the values $y=0$ of $y=1$ from our analysis.

By using the filter $\Pi$, we introduce the following problem: Find an optimal control $\xi^{\star} \in \cB$ such that, for any $(x,y) \in \R \times (0,1)$, 
\begin{multline}\label{eq:filtered:value}
    \inf_{\xi \in \cB} \varlimsup_{T \to \infty}\frac{1}{T}\E_{x,y}\left[\int_0^T c(X^{\xi}_t)dt +K_+\xi^+_T +K_-\xi^-_T \right] \\
    = \varlimsup_{T \to \infty}\frac{1}{T}\E_{x,y}\left[\int_0^T c(X^{\xi^{\star}}_t)dt +K_+\xi^{\star,+}_T +K_-\xi^{\star,-}_T \right]
\end{multline}
under the constraint
\begin{equation}\label{eq:filtering:controlled_dynamics}
    \left\{ \begin{aligned}
    dX^{\xi}_t & = (m(\Pi_t) -\delta X^{\xi}_t )dt + \sigma d B_t + d\xi_t, && X^\xi_{0^-} = x, \\
    d\Pi_t & = (\lambda_2 - (\lambda_1 + \lambda_2)\Pi_t) dt + \gamma\Pi_t(1-\Pi_t)dB_t, && \Pi_{0} = y.
    \end{aligned}  \right.
\end{equation}
Thanks to the strong uniqueness of the solutions to \eqref{application:dynamics} and \eqref{eq:filtering:controlled_dynamics}, it is standard to see that \eqref{eq:application:value} under \eqref{application:dynamics} is equivalent to \eqref{eq:filtered:value} under \eqref{eq:filtering:controlled_dynamics} (see, e.g., \cite[Proposition 3.11]{callegaro2020optimalreduction}, for a similar setting).
According to such an equivalence, we now aim at solving the cost minimization problem \eqref{eq:filtered:value} under the constraint \eqref{eq:filtering:controlled_dynamics}.

\smallskip
By using the notation of Section \ref{sec:model}, we have that the factor process $Y$ is now one-dimensional and it is given by the filter $\Pi$.
The coefficient $(\eta,\zeta)$ are given by
\begin{equation*}
\eta(y) \coloneqq (\lambda_2 - (\lambda_1+\lambda_2)y), \quad \zeta(y) \coloneqq \gamma y (1-y).
\end{equation*}
The coefficients $(b,\sigma)$ are given by
\begin{equation*}
b(x,y) \coloneqq m_2 + (m_1-m_2)y -\delta x, \quad \sigma(x,y) \coloneqq \sigma > 0.
\end{equation*}
The processes $X$ and $Y$ are run by the same Brownian motion $B$, given by the innovation process \eqref{eq:filtered:innovation_process}, so that the correlation is $\rho = 1$.
We notice that Assumption \ref{assumptions} is satisfied.

\begin{remark}
    The choice of singular controls adapted to $\bbF^{X^0}$, i.e., to the filtration generated by the uncontrolled observed process may seem rather odd. Indeed, the standard requirement is that controls are adapted with respect to the observed process, which in our case is $X^\xi$. However, since this process is controlled, the available information changes with the control itself, i.e., the issue of \textit{circularity of information} arises. From a technical point of view this poses an obstacle to the formulation of our control problem, which is bypassed precisely by choosing controls adapted to $\bbF^{X^0}$, in the original problem. 
    
    This turns out to be equivalent to selecting controls adapted to $\bbF^B$, as the two filtrations coincide (cf.~\citep[p.~31]{liptser2001statistics2}). Indeed, on the one hand, the innovation process $B$ is by construction an $\bbF^{X^0}$-Wiener process, and hence $\bbF^B \subseteq \bbF^{X^0}$; on the other hand, the system of SDEs~\eqref{eq:filtering:controlled_dynamics} with $\xi \equiv 0$, which is driven by $B$ alone and where $\Pi$ is $\bbF^{X^0}$-adapted by construction, admits a unique pathwise strong solution, and hence $\bbF^{X^0} \subseteq \bbF^{B}$. Therefore, in our model the controller is able to observe the noise acting on the fully observable system~\eqref{eq:filtering:controlled_dynamics}, given by the innovation process. In particular, we see that admissible singular controls in $\cB$ are also $\bbF^B$-adapted.
    The equivalence of the two filtrations reveals that the formulation of the separated problem is the standard one for a singular control problem under complete observation, in which we choose controls to be adapted to the noise driving the system. In this context the noise is the innovation, which represents the new information content that we can extract from the observation as time passes.
    
    Note that, differently from our model, whenever the observation is not directly controlled, but the control acts indirectly on it through the signal process, it is possible to use a change of measure technique to solve this issue, see, e.g.,~\citep[Chapter~8]{bensoussan1992partially_observed}. We can also observe that, while the (completed) natural filtration generated by $X^\xi$ is contained in $\bbF^{X^0}$, for any possible choice of admissible $\xi$, the converse may not be true in general. This property could be verified by filtration $\bbF^{X^{\xi^\star}}$, with $\xi^\star$ the optimal control, but it needs to be checked \textit{a posteriori}.
\end{remark}

\subsubsection{The auxiliary Dynkin game for the separated problem}\label{sec:application:partial_observation:filtered_problem}
As described in Section \ref{sec:model}, the first step to solve the ergodic control problem \eqref{eq:filtered:value} (and thus the original one \eqref{eq:application:value}) is to identify the auxiliary Dynkin game \eqref{eq:dynkin} and to show that Hypothesis \ref{conj:main} holds true.
We first establish that the auxiliary Dynkin game has a value and that there exists a saddle point $(\tau^*,\theta^*)$, which is characterized in terms of the functions $a_-$ and $a_+$, the so called \textit{free-boundaries} of the PDE satisfied by $U$.
Then, we show that Hypothesis \ref{conj:main} is satisfied. In addition, we show that $U$ is $\dC^1$ over the whole state space.

\smallskip
As the volatility of $X^0$ is constant, the underlying Markov process of the Dynkin game $(\widehat{X},\widehat{Y})$ defined by \eqref{eq:X_hat_0} is given by $(X^0,\Pi)$ itself.
To ease the notation, we set $\cO = \R \times (0,1)$ to denote the state space of $(X^0,\Pi)$.
When needed, in order to simplify the notation, we write $(X^{x,y},\Pi^{x,y})$ to stress the dependence on the initial position $(x,y) \in \cO$.

The Dynkin game is given by 
\begin{equation}\label{eq:filtered:U}
U(x,y) \coloneqq \inf_{\tau \geq 0} \sup_ {\theta \geq 0} \E\left[\int_0^{\tau \land \theta} e^{-\delta t} c' (X^{x,y}_t)dt +K_- e^{-\delta \tau}\ind_{\tau < \theta} - K_+ e^{-\delta \theta}\ind_{\theta < \tau} \right],
\end{equation}
where $\tau$, $\theta$ are stopping times of the filtration generated by the innovation process $B$ in \eqref{eq:filtered:innovation_process}.
Occasionally, we will use the notation
\begin{equation}\label{eq:filtered:zero_sum_functional}
M_{(x,y)}(\tau,\theta) \coloneqq \E\left[\int_0^{\tau \land \theta} e^{-\delta t} c' (X^{x,y}_t)dt +K_- e^{-\delta \tau}\ind_{\tau < \theta} - K_+ e^{-\delta \theta}\ind_{\theta < \tau} \right],
\end{equation}
so that $U(x,y) = \inf_{\tau} \sup_{\theta} M_{(x,y)}(\tau,\theta)$.
For later use, we state a simple integrability result, whose proof is omitted.
\begin{lemma}\label{filtered:lemma:finiteness}
Let $(X^{x,y},\Pi^{x,y})$ { be the solution of} \eqref{eq:filtering:filtered_dynamics} starting from $(x,y) \in \mathcal{O}$.
Then, $X^{x,y}$ has the following explicit representation:
\begin{equation}\label{eq:explicit_representation}
    X^{x,y}_t = e^{-\delta t}\left( x + \int_0^t e^{\delta s}m(\Pi^{y}_s)ds + \sigma\int_0^t e^{\delta s}dB_s \right).
\end{equation}
Moreover, for any $q > 0$ there exists a positive constant $\kappa$, dependent of $q$, $m_1$, $m_2$, $\delta$ and $\sigma$ but independent of $(x,y)$, such that
\begin{equation}\label{eq:filtered:bound}
\E\left[\int_0^{\infty} e^{-\delta t} \vert X^{x,y}_t \vert^q dt\right] \leq \kappa(1+ \vert x \vert^{q}).
\end{equation}
\end{lemma}
In particular, Lemma \ref{filtered:lemma:finiteness} and Assumptions \ref{application:assumptions_cost} imply that $M_{(x,y)}(\tau,\theta)$ is finite for any pair of stopping times $(\tau,\theta)$.

\smallskip
The next lemma shows that the Dynkin game defined by \eqref{eq:filtered:U} has a value and both the sup-player and the inf-player have an optimal stopping strategy.
As a consequence, we get a first characterization of the optimal stopping times and we deduce that we can exchange supremum and infimum in the definition of $U$.
We define the regions
\begin{equation}\label{eq:filtered:continuation_stopping1}
\begin{aligned}
    \mathcal{S}_+ & \coloneqq \{ (x,y) \in \cO: \; U(x,y) \leq - K_+ \}, \\
    \mathcal{S}_- & \coloneqq \{ (x,y) \in \cO: \; U(x,y) \geq K_- \}, \\
    \mathcal{C} & \coloneqq \cO \setminus (\mathcal{S}_+ \cup \mathcal{S}_-) = \{ (x,y) \in \cO: \; -K_+ < U(x,y) < K_- \}.
\end{aligned}
\end{equation}

\begin{lemma}\label{filtered:lemma_value_game}
The Dynkin game \eqref{eq:filtered:U} has a value.
Moreover, the stopping times
\begin{equation}\label{eq:filtered:optimal_stopping_times}
    \tau^* \coloneqq \inf\{t \geq 0: \; (X^0_t,\Pi_t) \in \mathcal{S}_-\}, \quad \theta^* \coloneqq \inf\{t \geq 0: \; (X^0_t,\Pi_t) \in \mathcal{S}_+\}
\end{equation}
are optimal strategies for the inf-player and the sup-player, respectively.
\end{lemma}
\begin{proof}
The proof consists in an application of \cite[Theorem 2.1]{ekstrom_peskir2008SICON}.
Consider the following functions
\begin{equation}\label{filtered:lemma_value_game:auxiliary_functions}
\begin{aligned}
& \Phi(x,y) \coloneqq \E\left[\int_0^\infty e^{-\delta s} c'(X^{x,y}_s)ds \right], && G_1(t,x,y) \coloneqq e^{-\delta t}(-K_+ - \Phi(x,y)), \\
& G_2(t,x,y) \coloneqq e^{-\delta t}(K_- - \Phi(x,y)), && G_3(t,x,y) \coloneqq -e^{-\delta t} \Phi(x,y).
\end{aligned}
\end{equation}
Notice that $\Phi$ is measurable and finite by Lemma \ref{filtered:lemma:finiteness}, and thus so are $G_i$ for any $i=1,2,3$. Moreover,  $G_1 \leq G_3 \leq G_2$.
Denote by $(X^{t,x,y}_s,\Pi^{t,x,y}_s)_{s \geq t}$ the solution of \eqref{eq:filtering:filtered_dynamics} starting from $(x,y) \in \cO$ at time $t \geq 0$ and define the auxiliary Dynkin game
\begin{multline}\label{eq:filtered:lemma_value:auxiliary_dynkin}
\Tilde{u}(t,x,y) \coloneqq \inf_{\tau} \sup_{\theta} \E\Big[  \ind_{\tau < \theta} G_2(t+\tau,X^{t,x,y}_{t+\tau},\Pi^{t,x,y}_{t+\tau}) \\
+ \ind_{\theta < \tau} G_1(t+\theta,X^{t,x,y}_{t+\theta},\Pi^{t,x,y}_{t+\theta}) + \ind_{\tau = \theta}G_3(t+\tau,X^{t,x,y}_{t+\tau},\Pi^{t,x,y}_{t+\tau})\Big].
\end{multline}
Notice that the process $\cZ = (t,X^{x,y}_t,\Pi^{x,y}_t)_{t \geq 0}$ is  continuous and strong-Markov.
Moreover, by exploiting Lemma \ref{filtered:lemma:finiteness}, the explicit representation of $X^{x,y}$ in \eqref{eq:explicit_representation} and the boundedness of $\Pi^y$, we have
\begin{multline}
    \E[\sup_{t \geq 0} \vert G_1(t,X^{x,y}_t,\Pi^{x,y}_t)\vert] \leq K_+ + \E[\sup_{t \geq 0} \vert e^{-\delta t}\Phi(X^{x,y}_t,\Pi^{x,y}_t)\vert] \\
    \leq \kappa(1 +\E[\sup_{t \geq 0} e^{-\delta t} \vert X^{x,y}_t\vert^{p-1}]) < \infty,
\end{multline}
and analogously for $G_2$ and $G_3$.
Hence, by \cite[Theorem 2.1]{ekstrom_peskir2008SICON}, the stopping times
\begin{align*}
    & \tau^* = \inf\{ s \geq 0: \Tilde{u}(t+s,X^{x,y}_{t+s},\Pi^{x,y}_{t+s}) = G_2(t+s,X^{x,y}_{t+s},\Pi^{x,y}_{t+s}) \}, \\
    & \theta^* = \inf\{ s \geq 0:  \Tilde{u}(t+s,X^{x,y}_{t+s},\Pi^{x,y}_{t+s}) = G_1(t+s,X^{x,y}_{t+s},\Pi^{x,y}_{t+s}) \}
\end{align*}
are a saddle point for the Dynkin game \eqref{eq:filtered:lemma_value:auxiliary_dynkin}.
We now prove that $(\tau^*,\theta^*)$ are optimal for the game with value $U(x,y)$ as well.
We start by noticing that $(X^{0,x,y}_s,\Pi^{0,x,y}_s)=(X^{x,y}_s,\Pi^{x,y}_s)$ for any $s \geq 0$ and that $\Tilde{u}(t,x,y) = e^{-\delta t} u(x,y)$, where
\begin{multline*}
u(x,y) = \inf_{\tau} \sup_{\theta} \E\big[  \ind_{\tau < \theta} G_2(\tau,X^{x,y}_\tau,\Pi^{x,y}_\tau) + \ind_{\theta < \tau} G_1(\theta,X^{x,y}_\theta,\Pi^{x,y}_\theta) \\
+ \ind_{\tau = \theta}G_3(\tau,X^{x,y}_\tau,\Pi^{x,y}_\tau)\big].
\end{multline*}
Next, we notice that, by the strong Markov property, it holds for any pair $(\tau,\theta)$ 
\begin{multline}\label{eq:filtered:lemma_value_game:equality_values}
    M_{(x,y)}(\tau,\theta) = \Phi(x,y) + \E\big[  \ind_{\tau < \theta} G_2(\tau,X^{x,y}_\tau,\Pi^{x,y}_\tau) + \ind_{\theta < \tau} G_1(\theta,X^{x,y}_\theta,\Pi^{x,y}_\theta) \\
    + \ind_{\tau = \theta}G_3(\tau,X^{x,y}_\tau,\Pi^{x,y}_\tau)\big].
\end{multline}
Finally, by taking the infimum over $\tau$ and supremum over $\theta$ in \eqref{eq:filtered:lemma_value_game:equality_values}, we deduce $U(x,y) = \Phi(x,y) + u(x,y)$.
This implies that $(\tau^*,\theta^*)$ can be expressed as in \eqref{eq:filtered:optimal_stopping_times}.
\end{proof}

The second step is to prove the following preliminary regularity properties of $U$:
\begin{lemma}\label{filtered:lemma_regularity_U}
\mbox{}
\begin{enumerate}[label=(\roman*)]
    \item \label{filtered:lemma_regularity_U:continuity} $U$ is jointly continuous in $\cO$.
    \item \label{filtered:lemma_regularity_U:monotonicity} For any fixed $y \in (0,1)$, $x \mapsto U(x,y)$ is non-decreasing and, for any fixed $x \in \R$, the map $y \mapsto U(x,y)$ is non-decreasing.
\end{enumerate}
\end{lemma}
\begin{proof}
Let $( (x_n,y_n) )_{n \geq 1}$ be a sequence in $\cO$ converging to $(x,y) \in \cO$.
Fix $\epsilon > 0$ and let $\bar\theta$ so that
\begin{multline}\label{filtered:lemma_regularity_U:leq_bound}
U(x,y) =  \inf_{\tau} \sup_{\theta} M_{(x,y)}(\tau,\theta) =  \sup_{\theta}\inf_{\tau} M_{(x,y)}(\tau,\theta) \\
\leq \inf_{\tau} M_{(x,y)}(\tau,\bar\theta) { + \epsilon} \leq M_{(x,y)}(\tau,\bar\theta) { + \epsilon}
\end{multline}
for any $\tau$, where { second equality} holds by Lemma \ref{filtered:lemma_value_game}.
Analogously, let $\tau^n$ so that 
\begin{equation}\label{filtered:lemma_regularity_U:geq_bound}
U(x_n,y_n) =  \inf_{\tau} \sup_{\theta} M_{(x_n,y_n)}(\tau,\theta) \geq \sup_{\theta} M_{(x_n,y_n)}(\tau^n,\theta) - \epsilon \geq M_{(x_n,y_n)}(\tau^n,\theta) - \epsilon \quad \forall \theta.
\end{equation}
By taking the differences, we get
\begin{align*}
    U(x,y) & - U(x_n,y_n) \leq M_{(x,y)}(\tau^n,\bar{\theta}) - M_{(x_n,y_n)}(\tau^n,\bar{\theta}) { +2\epsilon} \\
    & = \E\left[\int_0^{\tau^n \land \bar{\theta}} e^{-\delta t} \left( c' (X^{x,y}_t) - c'(X^{x_n,y_n}_t)\right) dt \right] { +2\epsilon} \\
    & \leq \E\left[\int_0^\infty e^{-\delta t} \left\vert c' (X^{x,y}_t) - c'(X^{x_n,y_n}_t)\right\vert dt \right] { +2\epsilon}.
\end{align*}
By the same reasoning, exchanging the roles of $(x,y)$ and $(x_n,y_n)$ in \eqref{filtered:lemma_regularity_U:leq_bound} and \eqref{filtered:lemma_regularity_U:geq_bound}, we bound the absolute value of the difference by
\begin{align*}
    \vert U(x,y) & - U(x_n,y_n) \vert \leq \E\left[\int_0^\infty e^{-\delta t} \left\vert c' (X^{x,y}_t) - c'(X^{x_n,y_n}_t)\right\vert dt \right] { +2\epsilon }.
\end{align*}
Given that the coefficients of the pair $(X^0,\Pi)$ are smooth functions with bounded derivatives, the flow $(x,y) \mapsto (X^{x,y}_t,\Pi^{x,y}_t)$ is a diffeomorphism for every $t > 0$ by \cite[Theorem 13.8]{rogerswilliams_vol2}.
As $c'$ is continuous as well, we deduce that the integrand converges to $0$, $\P$-a.s., for any $t \geq 0$. 
By Lemma \ref{filtered:lemma:finiteness}, we invoke the dominated convergence theorem to get $\displaystyle\limsup_{n \to \infty}\vert U(x,y) - U(x_n,y_n) \vert \leq 2\epsilon$, with integrable majorant the function
\[
  \alpha_1\Big( 2 + \frac{2^{p-1}m_1}{\delta}+ 2^{p-1}\Big\vert e^{-\delta t} \int_0^t e^{\delta s} dB_s  \Big\vert^{p-1} + 2^{p-1} \Big\vert  e^{-\delta t}( \Bar{x} -\frac{m_1}{\delta} ) \Big\vert^{p-1} \Big).
\]
where we have set $\Bar{x} = \max_{n \geq 1}(x,x_n)$.
Since $\epsilon$ is arbitrary, we get point \ref{filtered:lemma_regularity_U:continuity}.

\smallskip
As for point \ref{filtered:lemma_regularity_U:monotonicity}, fix $y \in (0,1)$, let $x_1 \leq x_2$.
By equation \eqref{eq:explicit_representation}, it holds $X^{x_1,y}_t \leq X^{x_2,y}_t$, for any $t \geq 0$, $\P$-a.s. Since $c'$ is strictly increasing, $c'(X^{x_1,y}_t) \leq c'(X^{x_2,y}_t)$, for any $ t \geq 0$, $\P$-a.s., which implies that, for any $(\tau,\theta)$, $M_{(x_1,y)}(\tau,\theta) \leq M_{(x_2,y)}(\tau,\theta)$, and so the conclusion holds for $U(x,y)$ as well.
Fix now $x \in \R_+$, $0 \leq y_1 \leq y_2 \leq 1$. As \eqref{eq:filtering:filtered_dynamics} admits a unique strong solution, by \cite[Theorem 1.1]{ikeda1977comparison} the map $y \mapsto \Pi^{x,y}_t$ is non-decreasing for any $t \geq 0$, $\P$, a.s. Since $m_1 > m_2$ by assumption, the map $y \mapsto m(\Pi^{x,y}_t)$ is increasing as well, and therefore \eqref{eq:explicit_representation} implies that $y \mapsto X^{x,y}_t$ is increasing for every $ t \geq 0$ $\P$-a.s. Thanks to the monotonicity of $c'$, we get \ref{filtered:lemma_regularity_U:monotonicity}.
\end{proof}

We define the following functions
\begin{equation}\label{eq:filtered:barriers}
    a_-(y) \coloneqq \inf\{ x \in \R: \; U(x,y) \geq K_- \}, \quad a_+(y) \coloneqq \sup\{ x \in \R: \; U(x,y) \leq -K_+ \},
\end{equation}
with the conventions $\sup \emptyset = -\infty$, $\inf\emptyset = +\infty$.
Then, by continuity and monotonicity of $U$ and exploiting the bounds $-K_+ \leq U(x,y) \leq K_-$, the sets $\cS_\pm$ and $\cC$ defined in \eqref{eq:filtered:continuation_stopping1} can be expressed in terms of $a_\pm$ as
\begin{equation}\label{eq:filtered:regions_U}
\begin{aligned}
    & \cS_+ = \{ (x,y) \in \cO: \, x \leq a_+(y) \}, \quad \cS_- = \{ (x,y) \in \cO: \, x \geq a_-(y) \}, \\
    & \cC = \{ (x,y) \in \cO: \, a_+(y) < x < a_-(y) \}.
\end{aligned}
\end{equation}

\begin{lemma}\label{filtered:lemma:regularity_boundary}
\mbox{}
\begin{enumerate}[resume,label=(\roman*)]
    \item \label{filtered:lemma:regularity_boundary:monotonicity} The maps $a_\pm(y)$ are non-increasing. Moreover, $a_-$ is right-continuous and $a_+$ is left-continuous.
    \item \label{filtered:lemma:regularity_boundary:separated} For any $y \in (0,1)$, it holds $a_+(y) \leq (c')^{-1}(-K_+\delta) < (c')^{-1}(K_-\delta) \leq a_-(y)$.
    \item\label{filtered_regularity_boundary:bounded_barriers} There exist two finite values $\underline{a}_+$, $\overline{a}_-$ so that $-\infty < \underline{a}_+ \leq a_+(y)$ and $a_-(y) \leq \overline{a}_- < + \infty$.
\end{enumerate}
\end{lemma}

\begin{proof}
Point \ref{filtered:lemma:regularity_boundary:monotonicity} follows from definition of $a_\pm$ and the monotonicity and continuity of $U$, ensured by Lemma \ref{filtered:lemma_regularity_U}.
As for point \ref{filtered:lemma:regularity_boundary:separated}, consider the processes
\begin{align}
    & \Big( e^{-\delta (t \land \theta^*)}U(X^{x,y}_{t \land \theta^*},\Pi^{x,y}_{t \land \theta^*}) + \int_0^{t\land \theta^*}e^{-\delta s}c'(X^{x,y}_s)ds \Big)_{t \geq 0},  \label{eq:lemma_regularity:subm} \\
    & \Big( e^{-\delta (t \land \tau^*)}U(X^{x,y}_{t \land \tau^*},\Pi^{x,y}_{t \land \tau^*}) + \int_0^{t\land \tau^*}e^{-\delta s}c'(X^{x,y}_s)ds \Big)_{t \geq 0}, \label{eq:lemma_regularity:supm}
\end{align}
which, by \cite[Theorem 2.1]{peskir2008optimal_stopping_games}, are a sub-martingale and super-martingale respectively.
Let now $(x_o,y_o) \in \cS_-$.
By using the sub-martingale property of the process \eqref{eq:lemma_regularity:subm} and the bound $U(x,y) \leq K_-$, we deduce
\begin{multline*}
K_- = U(x_o,y_o) \leq \E\left[ e^{-\delta (t \land \theta^*)}U(X^{x_o,y_o}_{t \land \theta^*},\Pi^{x_o,y_o}_{t \land \theta^*}) + \int_0^{t\land \theta^*}e^{-\delta s}c'(X^{x_o,y_o}_s)ds \right] \\
{ \leq } \E\left[ K_- e^{-\delta(t \land \theta^*)} + \int_0^{t\land \theta^*}e^{-\delta s}c'(X^{x_o,y_o}_s)ds  \right] \leq K_- + \E\left[  \int_0^{t\land \theta^*}e^{-\delta s}(c'(X^{x_o,y_o}_s) -K_-\delta)ds  \right]
\end{multline*}
for any $t \geq 0$. This implies
\begin{equation}\label{lemma:regularity_boundary:sign_S_-}
    0 \leq \lim_{t \to 0}\frac{1}{t}\E\left[ \int_0^{t}e^{-\delta s}(c'(X^{x_o,y_o}_s) -K_-\delta)\ind_{s < \theta^*} ds \right] = c'(x_o) -K_-\delta, \quad \forall (x_o,y_o) \in \cS_-,
\end{equation}
where we used dominated convergence theorem, by using as majorant
\[
K_-\delta + \alpha_2+ \alpha_2\sup_{0 \leq s \leq 1} e^{-\delta s} \vert X^{x_o,y_o}_s \vert^{p-1},
\]
which is clearly integrable.
Let $(c')^{-1}$ be the inverse of $c'$, which exists as $c$ is strictly convex by Assumption \ref{application:assumptions_cost}.
Then, \eqref{lemma:regularity_boundary:sign_S_-} implies
\[
a_-(y) \geq \inf\{ x \in \R: c'(x) \geq K_-\delta \} = \inf\{ x \in \R: x \geq (c')^{-1} (K_-\delta) \} = (c')^{-1}(K_-\delta).
\]
Analogously, for any $(x_o,y_o) \in \cS_+$, by relying on the super-martingale property of the process \eqref{eq:lemma_regularity:supm}, we deduce 
\begin{equation}\label{lemma:regularity_boundary:sign_S_+}
    0 \geq c'(x_o) { +K_+\delta}, \quad \forall (x_o,y_o) \in \cS_+,
\end{equation}
which in turns implies ${ a_+(y) \leq} (c')^{-1}(-\delta K_+)$.
Moreover, by using again the strict convexity of $c$, we deduce $ (c')^{-1}(-\delta K_+) < (c')^{-1}(\delta K_-) $.
This proves point \ref{filtered:lemma:regularity_boundary:separated}.

\smallskip
To prove point \ref{filtered_regularity_boundary:bounded_barriers}, let $\underline{X}$ and $\overline{X}$ be, respectively, the solutions of
\begin{align}
    & d\underline{X}^{x}_t = (m_2 -\delta\underline{X}^{x}_t )dt + \sigma dB_t, \quad \underline{X}^{x}_0 = x, \label{eq:filtered:lower_X} \\
    & d\overline{X}^{x}_t = (m_1 -\delta\overline{X}^{x}_t )dt + \sigma dB_t, \quad \overline{X}^{x}_0 = x, \label{eq:filtered:upper_X}
\end{align}
and notice that, since $m_2 \leq m(\Pi_t) \leq m_1$, by \cite[Theorem 1.1]{ikeda1977comparison} we have $\underline{X}^{x}_t \leq X^{x,y}_t \leq \overline{X}^{x}_t$ for every $t \geq 0$ $\P$-a.s. for any $(x,y) \in \cO$.
Consider the following Dynkin games:
\begin{align*}
    & \underline{u}(x) \coloneqq \inf_{\tau} \sup_ {\theta} \E\left[\int_0^{\tau \land \theta} e^{-\delta t} c' (\underline{X}^x_t)dt +K_- e^{-\delta \tau}\ind_{\tau < \theta} - K_+ e^{-\delta \theta}\ind_{\theta < \tau} \right], \\
    & \overline{u}(x) \coloneqq \inf_{\tau} \sup_ {\theta} \E\left[\int_0^{\tau \land \theta} e^{-\delta t} c' (\overline{X}^x_t)dt +K_- e^{-\delta \tau}\ind_{\tau < \theta} - K_+ e^{-\delta \theta}\ind_{\theta < \tau} \right].
\end{align*}
By the same reasoning as in Lemmata \ref{filtered:lemma_value_game} and \ref{filtered:lemma_regularity_U}, it is possible to prove that $\underline{u}$ and $\overline{u}$ are continuous and non-decreasing.
Moreover, there exist two saddle points $(\underline{\tau},\underline{\theta})$ and  $(\overline{\tau},\overline{\theta})$ for $\underline{u}$ and $\overline{u}$, respectively, given by
\begin{align*}
    & \underline{\tau} \coloneqq \inf\{ t \geq 0: \underline{u}(\underline{X}_t) \geq K_- \}, \quad \underline{\theta} \coloneqq \inf\{ t \geq 0: \underline{u}(\underline{X}_t) \leq -K_+ \}, \\
    & \overline{\tau} \coloneqq \inf\{ t \geq 0: \overline{u}(\overline{X}_t) \geq K_- \}, \quad \overline{\theta} \coloneqq \inf\{ t \geq 0: \overline{u}(\overline{X}_t) \leq -K_+ \}.
\end{align*}
We focus on $\underline{\tau}$ and $\overline{\theta}$.
As for $U(x,y)$, set
\[
\underline{a}_+ \coloneqq \sup\{x \in \R: { \overline{u}(x)} \leq -K_+ \}, \quad \overline{a}_- \coloneqq \inf\{x \in \R: { \underline{u}(x)} \geq K_- \},
\]
so that ${ \underline{\tau}} = \inf\{t \geq 0: { \underline{X}_t }\geq \overline{a}_- \}$ and ${ \overline{\theta}} = \inf\{t \geq 0: { \overline{X}_t }\leq \underline{a}_+ \}$.
Notice that, since $c'$ is strictly increasing, it holds $\underline{u}(x) \leq U(x,y) \leq \overline{u}(x)$ for all $(x,y) \in \cO$, which implies that
\[
a_+(y) = \sup\{x \in \R : U(x,y) \leq -K_+ \} \geq \sup\{x \in \R : { \overline{u}(x)} \leq -K_+ \} = \underline{a}_+.
\]
To prove $\underline{a}_+ > -\infty$, we show $\{x \in \R:\, { \overline{u}(x)} \leq -K_+ \}$ is nonempty.
Suppose not.
Then, for any $(x,y) \in \cO$, it holds $-K_+ < { \overline{u}(x)}$, which implies that ${ \overline{\theta}} = \infty$ $\P$-a.s.
Therefore, 
\begin{multline*}
    -K_+ < { \overline{u}(x)} = \inf_{\tau}\E\left[\int_0^{\tau} e^{-\delta t} c' ({ \overline{X}^x_t})dt {  + K_- e^{-\delta \tau}\ind_{\tau < \infty} } \right] \\
    \leq \E\left[\int_0^\infty e^{-\delta t} c' (\overline{X}^x_t)dt \right] = \E\left[\int_0^\infty e^{-\delta t} c' ({ e^{-\delta t} x} + \overline{X}^0_t)dt \right],
\end{multline*}
for all $x \in \R$, where $\overline{X}^0$ denotes the solution of \eqref{eq:filtered:upper_X} starting from $x=0$ at $t = 0$.
As $c'$ is strictly increasing and $\lim_{x \to -\infty} c'(x)=-\infty$ { by Assumption \ref{application:assumptions_cost}}, the monotone convergence theorem yields
\begin{multline*}
    -K_+ < \lim_{x \to -\infty} \E\left[\int_0^\infty e^{-\delta t} c' ( { e^{-\delta t} } x + \overline{X}^0_t)dt \right] \\
    = \E\left[\int_0^\infty e^{-\delta t} \lim_{x \to -\infty}\left( c' ( { e^{-\delta t} } x + \overline{X}^0_t) \right) dt \right] = -\infty,
\end{multline*}
thus getting a contradiction.
The proof of $\overline{a}_- < +\infty$ is dealt analogously.
\end{proof}
Observe that Lemma \ref{filtered:lemma:regularity_boundary} implies that points (I) and (II) in Hypothesis \ref{conj:main} are satisfied by $a_\pm(y)$.
We now show that $U$ solves the free-boundary problem \eqref{eq:U_var_ineq} and that $U \in \dC^2(\cC)$, thus showing that point (III) in Hypothesis \eqref{conj:main} holds true.
On top of those properties, we will also prove that $U \in \dC^1(\cO)$, making an important step towards the application of Theorem \ref{thm:HJB_sol_weak}.

\smallskip
We notice that the generator $\cL_{(X^0,\Pi)}$ is degenerate, as a result of the fact that the diffusions $X^0$ and $\Pi$ are run by the same Brownian motion $B$.
Thus, classical PDE interior results based on Schauder's estimates do not hold.
Instead, we will show that $\mathcal{L}_{(X^0,\Pi)}$ is hypoelliptic (see, e.g., \cite[Section 2.3]{nualart2006malliavin}), allowing us to deduce the regularity of $U$ in the open set $\cC$ from \cite[Corollary 7]{peskir2025weak}.
This is accomplished in the following way: By a proper change of variables, we identify a diffeomorphic parabolic differential operator $\cL_{(X^0,Z)}$, defined as the generator of a diffusion $(X^0,Z)$ whose second component $Z$ is a process of finite variation.
Then, borrowing ideas from Ernst and Peskir in \cite{ernst2024gapeev} (see also \cite{peskir2024detection_multiple,peskir2024detection_ou}), we verify that $\cL_{(X^0,Z)}$ satisfies H\"{o}rmander's condition, ensuring the hypoellipticity of the operator $\cL_{(X^0,\Pi)}$ itself.

To show that $U \in \dC^1(\cO)$ (global smooth-fit property), in Lemma \ref{filtered:lemma:feller} we rely on the parabolic H\"{o}rmander's condition to show that $(X^0,\Pi)$ is strong Feller, similarly to \cite[Proposition 4]{peskir2024detection_multiple}.
Then, we prove in { Lemma} \ref{filtered:lemma:probabilistic_regularity} that the boundary points of $\cS_\pm$ are probabilistically regular for their complement sets relatively to $(X^0,\Pi)$.
Thanks to these fine technical results, we are able to prove in Theorem \ref{filtered:thm:regularity_up_to_boundary} that the $\dC^1$-regularity of $U$ extends to the boundary of $\cC$.

\smallskip
In order to prove the global smooth-fit property in Theorem \ref{filtered:thm:regularity_up_to_boundary}, we will need the following technical assumption on the parameters, which will be in force throughout the rest of the section:

\begin{assumption}\label{application:assumptions_discount}
Consider $p \geq 2$ given in Assumption \ref{application:assumptions_cost}.
The parameters satisfy the following relation: if $p>2$, then 
\begin{equation}\label{filtered:assumption_delta:p_not_2}
\delta > \left(\gamma^2  - (\lambda_1+\lambda_2)\right) \vee \left(6\gamma^2  - 2(\lambda_1+\lambda_2)\right) \vee \Big( 2(\frac{p-1}{p})(2p-3)\gamma^2 - 2\frac{p-1}{p}(\lambda_1 + \lambda_2) \Big) \vee 0.
\end{equation}
If $p=2$, then
\begin{equation}\label{filtered:assumption_delta:p_equal_2}
\delta > \left(\gamma^2  - (\lambda_1+\lambda_2)\right) \vee 0.
\end{equation}
\end{assumption}
Assumption \ref{application:assumptions_discount} ensures that the discount factor in the auxiliary Dynkin game is large enough to ensure uniform estimates on the the first order derivatives.

\smallskip
The following Lemma provides the change of variable which transforms $(X^0,\Pi)$ into the process $(X^0,Z)$, with $Z$ of finite variation.
As the proof is a straightforward application of It\^{o}'s formula, we omit it.

\begin{lemma}\label{filtered:lemma:change_of_variables}
For any $(x,y) \in \cO$, consider the process
\begin{equation}\label{eq:filtered:definition_z}
    Z_t \coloneqq \frac{\sigma}{\gamma}\log\left( \frac{\Pi_t}{1-\Pi_t} \right) - X^0_t.
\end{equation}
Then, the pair $(X^0,Z)$ satisfies the equation
\begin{equation}\label{eq:filtering:filtered_dynamics_change_of_variables}
\left\{ \begin{aligned}
    d X^0_t & =  \mu(X^0_t,Z_t)dt + \sigma dB_t, && X^0_0 = x, \\
    dZ_t & = q(X^0_t,Z_t) dt, && Z_0 = z = \frac{\sigma}{\gamma}\log\left(\frac{y}{1-y}\right) -x,
\end{aligned}\right.
\end{equation}
where $(\mu,q):\R^2 \to \R$ are defined by
\begin{equation}\label{eq:filtered:coefficient_X0_Z}
\begin{aligned}
\mu(x,z) & = m\left(\frac{e^{\frac{\gamma}{\sigma}(z + x)}}{1+e^{\frac{\gamma}{\sigma}(z + x)}}\right)  -\delta x , \\
q(x,z) & = \sigma\gamma \left(\frac{e^{\frac{\gamma}{\sigma}(z + x)}}{1+e^{\frac{\gamma}{\sigma}(z + x)} }- \frac{1}{2} \right) + \frac{\sigma}{\gamma}(1+e^{-\frac{\gamma}{\sigma}(z + x)})(\lambda_2 - \lambda_1 e^{\frac{\gamma}{\sigma}(z + x)}) \\
& \;\; - m\left(\frac{e^{\frac{\gamma}{\sigma}(z + x)}}{1+e^{\frac{\gamma}{\sigma}(z + x)}}\right)  +\delta x.
\end{aligned}
\end{equation}
\end{lemma}
The infinitesimal generator $\cL_{( X^0, Z)}$ of $(X^0,Z)$ is of parabolic type, and it is given by
\begin{equation}\label{eq:filtered:generator_change_of_variables}
\mathcal{L}_{( X^0,Z)}f(x,z) = q(x,z) f_z(x,z) + \frac{1}{2}\sigma^2 f_{xx}(x,z) + \mu(x,z) f_x (x,z),
\end{equation}
for any $f \in \dC^{2,1}(\R^2)$.
Observe that the operators $\cL_{(X^0,\Pi)}$ and $\cL_{(X^0,Z)}$ are $\dC^\infty$-diffeomorphic in $\cO$.
Indeed, consider the diffeomorphism $\Psi: \R \times \R \to \R \times (0,1)$
\begin{equation}\label{eq:filtered:diffeomorphism}
    (x,y) = \Psi\left((x,z)\right) = \left( x, \frac{e^{\frac{\gamma}{\sigma}(z + x) }}{ 1+ e^{\frac{\gamma}{\sigma}(z + x)}}\right),
\end{equation}
whose inverse $\Psi^{-1}:\R \times (0,1) \to \R \times \R$ is given by
\[    (x,z) = \Psi^{-1}\left((x,y)\right) = \left(x,\frac{\sigma}{\gamma}\log\Big(\frac{y}{1-y}\Big) -x \right).
\]
Occasionally, we will write $(x,y(x,z))$ instead of $\Psi(x,z)$ and $(x,z(x,y))$ instead of $\Psi^{-1}(x,y)$.
By construction, $\Psi$ provides the $\dC^\infty$-diffeomorphism between $\cL_{(X^0,\Pi)}$ and $\cL_{(X^0,Z)}$.

\begin{remark}
For later use, we define the likelihood ratio process $\Phi = (\Phi_t)_{t \geq 0}$ as $\Phi_t = \frac{\Pi_t}{1-\Pi_t}$.
By It\^{o}'s formula, the transformed process $(X^0,\Phi)$ satisfies the following equations:
\begin{equation}\label{eq:filtered:likelihood_ratio}
\left\{ \begin{aligned}
dX^0_t & =  \left[ m\left(\frac{\Phi_t}{1 + \Phi_t}\right) -\delta X^0_t \right]dt + \sigma dB_t, \quad X^0_0 = x, \\
d\Phi_t & = \left[ (1+\Phi_t)(\lambda_2 - \lambda_1 \Phi_t) + \gamma^2 \frac{\Phi_t^2}{1+\Phi_t}\right]dt + \gamma\Phi_tdB_t, \quad \Phi_0 = \phi = \frac{y}{1-y}.
\end{aligned} \right.
\end{equation}
Notice that the map $y \mapsto \phi(y) \coloneqq \frac{y}{1-y}$ is a diffemorphism from $(0,1)$ to $\R_+$, with inverse $y(\phi) = \frac{\phi}{1+\phi}$.
In particular, the process $Z_t$ can be defined starting from $(X^0,\Phi)$ by setting
\[
Z_t = \frac{\sigma}{\gamma}\log\left(\Phi_t\right) - X^0_t.
\]
This equivalent representation will be used extensively in the following.
\end{remark}

\begin{lemma}\label{filtered:lemma:weak_solution}
$U \in \dC^2(\cC)$ and it satisfies \eqref{eq:U_var_ineq}.
\end{lemma}
\begin{proof}
$U(x,y)$ satisfies the constraints on $\cS_\pm$ by definition.
We deal now with the behavior of $U$ in $\cC$.

We recall that the flow $(x,y) \mapsto (X^{x,y}_t,\Pi^{x,y}_t)$ is a diffeomorphism for every $t > 0$ by \cite[Theorem 13.8]{rogerswilliams_vol2}.
Then, for any $f \in \dC^\infty_c(\cO)$ and $t > 0$, the map $(x,y) \mapsto \E[f(X^{x,y}_t,\Pi^{x,y}_t)]$ is $\dC^2_b(\cO)$.
Therefore, Assumption (3.5) in \cite{peskir2025weak} is satisfied, which ensures by Corollary 5 therein that $U$ satisfies
\begin{equation}\label{eq:filtered:unconstrained_PDE}
\cL_{(X^0,\Pi)}U(x,y) - \delta U(x,y) + c'(x) = 0
\end{equation}
inside $\cC$, where the derivatives appearing in \eqref{eq:filtered:unconstrained_PDE} are to be understood in the sense of Schwartz distribution.
To improve the regularity of $U$ and to show that $U$ satisfies \eqref{eq:filtered:unconstrained_PDE} in the classical sense in $\cC$, we show that the operator $\cL_{( X^0,Z)}$ defined by \eqref{eq:filtered:generator_change_of_variables} is hypoelliptic over $\R^2$. As $\cL_{( X^0,\Pi)}$ and $\cL_{( X^0,Z)}$ are $\dC^\infty$-diffeomorphic on $\cO$, this implies that $\cL_{( X^0,\Pi)}$ is hypoelliptic over $\cO$.
{ Given that $c' \in \dC^\infty(\R)$, \cite[Corollary 7]{peskir2025weak} yields that the solution $U \in \dC^\infty(\cC)$} and thus concludes the proof.

\smallskip
We write $\cL_{(X^0,Z)}$ in quadratic form as $\cL_{(X^0,Z)} = D_0 + D_1^2$, where
\begin{equation}\label{eq:lemma:weak_solution:quadratic_form}
    D_0 = \mu(x,z) \partial_x + q(x,z) \partial_z, \quad D_1 = \frac{\sigma }{\sqrt{2}} \partial_x.
\end{equation}
We identify each first-order operator with the vector of its coefficients.
For any $(x,z) \in \R^2$, consider the Lie algebra $Lie(D_0,D_1)(x,z)$, that is the linear subspace generated by $D_0$ and $D_1$ and closed with respect to the Lie bracket operation.
We show that $\cL_{(X^0,Z)}$ satisfies the H\"{o}rmander's condition, (as given by, e.g., condition (H) in \cite[Section 2.3.2]{nualart2006malliavin}), i.e. $\mathrm{dim}Lie(D_0,D_1)(x,z)=2$ for any $(x,z) \in \R^2$.
This implies that $\cL_{(X^0,Z)}$ is hypoelliptic and thus concludes the proof.

We follow the argument of \cite[Theorem 6, Step I.3]{ernst2024gapeev}.
As $D_1 = [\frac{\sigma}{\sqrt{2}},0]$ is constant, it follows by easy computations that considering $n$ times the Lie bracket of $D_1$ and $D_0$ gives 
\begin{equation}\label{eq:lemma:weak_solution:lie_brackets}
[[\dots[[D_0,D_1], \dots],D_1], D_1 ] = \left(\frac{\sigma}{\sqrt{2}}\right)^{n}\left(  \mu^{(n)}_x(x,z) \partial_x + q^{(n)}_x(x,z) \partial_z \right),
\end{equation}
where $f^{(n)}_x(x,z)$ denotes the $n$-th partial derivative of $f$ with respect to $x$.
By using again that $\sigma$ is a positive constant, we see that H\"{o}rmander's condition is verified if, for any $(x,z)$, 
we have either $q(x,z) \neq 0$ or $q^{(n)}_x(x,z) \neq 0$ for some $n \geq 1$.
Suppose not: let $(x_0,z_0) \in \R^2$ so that $q(x_0,z_0)=0$ and $q^{(n)}_x(x_0,z_0) = 0$ for every $n \geq 1$.
Then, as $q(x,z)$ is analytic on $\R^2$, this implies that the section $\R \ni x \mapsto q(x,z_0)$ is identically equal to $0$.
Thus, the process $(X^0,Z)$ solution to \eqref{eq:filtering:filtered_dynamics_change_of_variables} starting from any point { $(x,z_0)$} is so that $Z_t \equiv z_0$ $\P$-a.s.
Employing the process $(X^0,\Phi)$ solution to \eqref{eq:filtered:likelihood_ratio}, this implies that $z_0 \equiv Z_t = \frac{\sigma}{\gamma}\log(\Phi_t) - X^0_t$, i.e. $\log(\Phi_t) = \frac{\gamma}{\sigma}(X^0_t+z_0)$ for any $t \geq 0$.
We show that this is not possible, which leads to a contradiction and concludes the proof.
Indeed, by It\^{o}'s formula, one should have
\[
(1+\Phi_t^{-1})(\lambda_2 - \lambda_1 \Phi_t) + \gamma^2 \frac{\Phi_t^2}{1+\Phi_t} - \frac{\gamma^2}{2}=\frac{\gamma}{\sigma}m_2 +\gamma^2\frac{\Phi_t}{1+\Phi_t} - \delta \frac{\gamma}{\sigma} X^0_t.
\]
By multiplying both sides by $1+\Phi_t$ and by imposing the desired equality $\frac{\gamma}{\sigma}X^0_t = \log(\Phi_t) - \frac{\gamma}{\sigma}z_0$, this implies that the equality
\[
(1+\phi^{-1})(\lambda_2 - \lambda_1 \phi) - \frac{\gamma^2}{2}=\frac{\gamma}{\sigma}\left(m_2 -\delta z_0\right) - \delta \log(\phi).
\]
should hold true for every $\phi > 0$. As this is clearly not the case, the proof is concluded.
\end{proof}

\begin{remark}
By using the terminology of \cite{ernst2024gapeev}, we proved that the system $(X^0,\Phi)$ does not admit any \emph{trap curve}, in the sense of Definition 5 therein.
{ This implies that the diffusive behavior of the pair $(X^0,\Phi)$ is strong enough to make the process spread across the whole state space $\R \times \R_+$. In other words, the process can not be confined (or "trapped") in a one-dimensional manifold. Hypoellipticity of the infinitesimal generator $\cL_{(X^0,\Phi)}$ is a direct consequence of this genuinely diffusive behavior, as shown by Lemma \ref{filtered:lemma:weak_solution}.}
\end{remark}

Our next goal is to prove that $U \in \dC^1(\cO)$.
To this extent, denote by $\tau^*(x,y)$ and $\theta^*(x,y)$ the first entry times in the sets $\cS_+$ and $\cS_-$ respectively.
We need to ensure that, for any sequence $((x_n,y_n))_{n \geq 1} \subseteq \cC$, $(x_n,y_n) \to (x,y) \in \partial\cC =\partial \cS_+ \cup \partial\cS_-$, it holds, respectively,
\begin{equation}\label{eq:filtered:convergence_stopping_times}
\tau^*(x_n,y_n) \to 0, \:\text{ $\P$-a.s.,} \quad \text{or} \quad \theta^*(x_n,y_n) \to 0, \:\text{ $\P$-a.s.}
\end{equation}
In order to prove the above convergence, we show in Lemma \ref{filtered:lemma:feller} that the process $(X^0,\Pi)$ is strong Feller.
Next, we show in Lemma \ref{filtered:lemma:probabilistic_regularity} that the boundary points of $\cS_\pm$ are probabilistically regular for $\cS_\pm^c$ relatively to $(X^0,\Pi)$ respectively.
As the process is strong Feller, by \cite[Volume II, Chapter 13.1-2]{dynkin1965markov_processes} { probabilistic regularity} of $\partial\cS_\pm$ is equivalent to having \eqref{eq:filtered:convergence_stopping_times}.

\begin{lemma}\label{filtered:lemma:feller}
The process $(X^0,\Pi)$ is strong Feller.
\end{lemma}
\begin{proof}
We apply again the results of \cite{peskir2025weak} (see also \cite[Proposition 4.4]{peskir2024detection_multiple}).
Let $f:\cO \to \R$ be bounded and measurable and define $\Tilde{F}:(0,\infty) \times  \cO \to \R$ by $\Tilde{F}(t,x,y)\coloneqq \E[f(X^{x,y}_t,\Pi^{x,y}_t)]$.
By \cite[Section 5]{peskir2025weak} (see in particular Corollary 8), it holds $\partial_t \Tilde{F} = \cL_{(X^0,\Pi)}\Tilde{F}$ on $(0,\infty) \times \cO$ in the weak sense of Schwartz distributions.
We show that the operator $-\partial_t + \cL_{(X^0,\Pi)}$ is hypoelliptic over $(0,\infty) \times \cO$. Then, by \cite[Corollary 9]{peskir2025weak} we deduce $\Tilde{F} \in \dC^\infty((0,\infty) \times \cO)$, which implies, in particular, that for any $t > 0$ fixed we have $(x,y)\mapsto \E[f(X^{x,y}_t,\Pi^{x,y}_t)]$ is continuous, thus proving the strong Feller property.

To prove the hypoellipticity, we show that the operator $-\partial_t + \cL_{(X^0,Z)}$ satisfies the parabolic H\"{o}rmander's condition. This implies that $-\partial_t + \cL_{(X^0,Z)}$ is hypoelliptic and therefore so is $-\partial_t + \cL_{(X^0,\Pi)}$.
We express $-\partial_t + \cL_{(X^0,Z)}$ in quadratic form as $\bar{D}_0 + \bar{D}_1^2$, where
\[
\bar{D}_0= -\partial_t + \mu(x,z)\partial_x + q(x,z)\partial_z, \quad \bar{D}_1= \frac{\sigma}{\sqrt{2}}\partial_x.
\]
As in the proof of Lemma \ref{filtered:lemma:weak_solution}, we identify $\bar{D}_0$ and $\bar{D}_1$ with the vector of their coefficients, so that $\bar{D}_0 = (-1,\mu,q)$ and $\bar{D}_1 = (0,\frac{\sigma}{\sqrt{2}},0)$.
Therefore, parabolic H\"{o}rmander's condition holds if $\mathrm{dim}Lie(\bar{D}_0,\bar{D}_1)(t,x,z) = 3$ for any $(t,x,z) \in (0,\infty)\times\cO$.
We notice $\bar{D}_0 = (-1,D_0)$ and $\bar{D}_1 = (0,D_1)$, with $D_0$ and $D_1$ given by \eqref{eq:lemma:weak_solution:quadratic_form}.
As $D_0$ and $D_1$ do not depend on time, the $n$ times commutator of $\bar{D}_0$ and $\bar{D}_1$ is again given by \eqref{eq:lemma:weak_solution:lie_brackets}.
Fix $(x,z) \in \R^2$.
As showed in the proof of Lemma \ref{filtered:lemma:weak_solution} above, for any $(x,z) \in \R^2$ there exists $\bar{n} \geq 0$ so that $q^{(\bar{n})}(x,z) \neq 0$.
Thus, for such $\bar{n}$, we easily see that the three vectors $\bar{D}_0$, $\bar{D}_1$ and the $\bar{n}$ times commutator of $\bar{D}_1$ and $\bar{D}_0$ are linearly independent, thus proving that H\"{o}rmander's condition is satisfied.
\end{proof}

By definition, (see, e.g., \cite[Definition 4.2.9]{karatzas_shreve}), the boundary points $\partial \cC = \partial\cS_+ \cup \partial\cS_-$ are probabilistically regular for $\cS_\pm$ relatively to $(X^0,\Pi)$ if the random times
\[
\hat{\sigma}_-(x,y) \coloneqq \inf\{t > 0: \, (X^{x,y}_t,\Pi^{x,y}_t) \in \cS_- \}, \quad \hat{\sigma}_+(x,y) \coloneqq \inf\{t > 0: \, (X^{x,y}_t,\Pi^{x,y}_t) \in \cS_+ \}
\]
are such that
\begin{equation}\label{eq:filtered:lemma_regularity:equal_zero}
\P(\hat{\sigma}_-(x,y) = 0)=1 \:\: \forall (x_o,y_o) \in \partial\cS_-,  \quad \P(\hat{\sigma}_+(x,y) = 0) = 1 \:\:  \forall (x_o,y_o) \in \partial\cS_+.
\end{equation}
This is proved in the following Lemma.

\begin{lemma}\label{filtered:lemma:probabilistic_regularity}
Every point $(x_o,y_o) \in \partial\cS_\pm$ is probabilistically regular, i.e. \eqref{eq:filtered:lemma_regularity:equal_zero} holds true.
\end{lemma}
\begin{proof}
Let $(x_o,y_o) \in \partial\cS_-$.
Recall the definition of the likelihood ratio process $\Phi$ in \eqref{eq:filtered:likelihood_ratio} and of the upper boundary $a_-$ in \eqref{eq:filtered:barriers}.
Define $\Tilde{a}_-(\phi) \coloneqq a_-\left(\frac{\phi}{1+\phi}\right)$ and notice that, as $a_-$ is a non-increasing function of $y$ and $\phi \mapsto \frac{\phi}{1+\phi}$ is strictly increasing, $\Tilde{a}_-(\phi)$ is non-increasing as well.
Then, as $\Pi_t = \frac{\Phi_t}{1 +\Phi_t}$, we have that
\[
(X^0_t,\Pi_t) \in \cS_- \iff X^0_t \geq a_-(\Pi_t) \iff X^0_t \geq a_-\left( \frac{\Phi_t}{1 + \Phi_t}\right) = \Tilde{a}_-(\Phi_t),
\]
so that $\hat{\sigma}_- = \inf\{t > 0: \, X^{0}_t \geq \Tilde{a}_-(\Phi_t)\}$.
Consider the process $\underline{X}$ defined in \eqref{eq:filtered:lower_X} and recall that $\P( \underline{X}_t \leq X^0_t \; \forall t \geq 0) = 1$.
Since $\Tilde{a}_-$ is non-increasing, the rectangle $\cR_- \coloneqq \{ (x,\phi) \in \R \times (0,\infty): \, x \geq x_o, \phi \geq \phi_o \}$ is contained in the set $\{(x,\phi) \in \R\times\R_+: \, x \geq \Tilde{a}_-(\phi)\}$ in $\cS_-$ for any $(x_o,y_o) \in \partial \cS_-$.
By setting $\phi_o = \frac{y_o}{1+y_o}$, these two facts imply the following chain of inequalities:
\begin{multline*}
    \underline{\sigma}_{\cR_-} := \inf\{ t > 0: \, \underline{X}_t \geq x_o, \Phi_t \geq \phi_o \} \geq \inf\{ t > 0: \, \underline{X}_t \geq \Tilde{a}_-(\Phi_t) \}\\
    \geq \inf\{ t > 0: \, X^0_t \geq \Tilde{a}_-(\Phi_t) \} = \hat{\sigma}_-.
\end{multline*}
As $\hat{\sigma}_- \leq \underline{\sigma}_{\cR_-}$, it is enough to show $\P(\underline{\sigma}_{\cR_-} = 0) = 1$.
To do so, we make the further change of variable $V_t = \mathrm{log}(\Phi_t)$ to get
\begin{equation}\label{eq:filtered:log_likelihood_ratio}
\begin{aligned}
    dV_t =  \left( \gamma^2 \left(\frac{e^{V_t}}{1+e^{V_t}}- \frac{1}{2} \right) + (e^{-V_t} + 1)(\lambda_2 - \lambda_1 e^{V_t}) \right)dt + \gamma dB_t  \quad V_0 = v_o = \log(\phi_o).
\end{aligned}
\end{equation}
so that $\underline{\sigma}_{\cR_-} = \inf\{t > 0 : \, \underline{X}_t \geq x_o, V_t \geq v_o \}$.
Reasoning as in \cite[Proposition A.4]{ernst2023minimaxwienersequentialtesting} (see also \cite[Appendix B.9]{ferrari2023mertonproblemirreversiblehealthcare}), by using the semi-explicit representations of $\underline{X}$ and $V$, we have
\begin{equation}\label{eq:lemma:probabilistic_regularity:equivalencies}
\begin{aligned}
    & \P\left( \underline{\sigma}_{\cR_-} \leq t \right) = \P\left( \underline{X}_s \geq x_o, { V_s} \geq v_o, \: \text{for some }s \in (0,t]\right) \\
    & = \P\bigg( x_o + \int_0^s(m_2-\delta\underline{X}_r)dr + \sigma B_s  \geq x_o, \\
    & \quad v_o { + } \int_0^s \left( \gamma^2 \Big(\frac{e^{V_r}}{1+e^{V_r}}- \frac{1}{2} \Big) + (e^{-V_r} + 1)(\lambda_2 - \lambda_1 e^{V_r}) \right)dr + \gamma B_s \geq v_o, \: \text{for some }s \in (0,t]\bigg) \\
    & =  \P\left( \frac{B_s}{s} \geq  \frac{1}{s}\int_0^s F_r dr, \, \frac{B_s}{s} \geq  \frac{1}{s}\int_0^s G_r dr, \: \text{for some }s \in (0,t] \right),
\end{aligned}
\end{equation}
where we set
\[
F_r \coloneqq \frac{\delta}{\sigma} \underline{X}_r - \frac{m_2}{\sigma}, \quad G_r \coloneqq \gamma \Big(\frac{1}{2}-\frac{e^{V_r}}{1+e^{V_r}} \Big) + \frac{1}{\gamma}(e^{-V_r} + 1)(\lambda_1 e^{V_r}-\lambda_2).
\]
We claim that the set above has full probability for any $t > 0$.
To see that, we recall that, by the law of iterated logarithms (e.g. \cite[Theorem 2.9.23]{karatzas_shreve}), { it holds $\varlimsup_{t \to 0}\frac{B_t}{t} = \infty$} on an event of full probability.
On the other hand, being the processes $\underline{X}$ and $V$ continuous, the right-hand sides of the inequalities in \eqref{eq:lemma:probabilistic_regularity:equivalencies} converge to some finite values $\alpha$ and $\beta$ on a set of full probability.
We suppose without loss of generality that those two events are equal, and we denote it by $A$.
Fix $t > 0$ and $\omega \in A$. Then, for any $\epsilon > 0$ and $L > \alpha + \beta + \epsilon $, it is possible to find a sequence $(s_n)_{n \geq 1}$ of times, possibly depending on $\omega$ itself, so that 
\[
\frac{B_{s_n}}{s_n} > L > \alpha + \epsilon > \frac{1}{s_n}\int_0^{s_n} F_r dr
\]
for any $s_n < t$.
This implies $A \subseteq \{ \underline{\sigma}_{\cR_-} \leq t \}$ for any $t>0$, so that $\P(\underline{\sigma}_{\cR_-} < t) = 1$.
By taking the limit as $t \to 0$, we get the claim for any $(x_o,y_o) \in \partial\cS_-$.

\smallskip
The regularity of points $(x_o,y_o) \in \partial\cS_+$ is dealt with analogously, by taking into account the pair $(\overline{X},\Phi)$, with $\overline{X}$ defined by \eqref{eq:filtered:upper_X}, instead of $(\underline{X},\Phi)$ and the rectangle $\cR_+ \coloneqq \{ (x,\phi) \in (0,\infty)\times[0,\infty): \, x \leq x_o, \phi \leq \phi_o \}$ instead of $\cR_-$.
The conclusion follows by applying the same arguments as above to the processes $\overline{X}$ and $V$, using the fact that $\varliminf_{s \to 0}\frac{B_s}{s} = -\infty$ $\P$-a.s.
\end{proof}

By the same computations of Lemma \ref{filtered:lemma:probabilistic_regularity}, it is possible to show also the following Lemma:
\begin{lemma}\label{filtered:lemma:probabilistic_regularity_interior}
Every point $(x_o,y_o) \in \partial\cS_\pm$ is probabilistically regular for the interior of $\cS_\pm$ relatively to $(X^0,\Pi)$, i.e. the random times
\begin{equation}\label{eq:def_sigma_check}
    \check{\sigma}_-(x,y) \coloneqq \inf\{t > 0: \, (X^{x,y}_t,\Pi^{x,y}_t) \in \mathring{\cS}_- \}, \quad \check{\sigma}_+(x,y) \coloneqq \inf\{t > 0: \, (X^{x,y}_t,\Pi^{x,y}_t) \in \mathring{\cS}_+ \}
\end{equation}
are such that
\begin{equation}\label{eq:filtered:lemma_regularity_interior:equal_zero}
\P(\check{\sigma}_-(x,y) = 0)=1 \:\: \forall (x_o,y_o) \in \partial\cS_-,  \quad \P(\check{\sigma}_+(x,y) = 0) = 1 \:\:  \forall (x_o,y_o) \in \partial\cS_+.
\end{equation}
\end{lemma}
Lemma \ref{filtered:lemma:probabilistic_regularity_interior} follows by the same computations as in Lemma \ref{filtered:lemma:probabilistic_regularity}, provided that we replace everywhere the  inequality "$\geq$" with the strict inequality "$>$".

We are now ready to prove the continuous differentiability of $U$:
\begin{theorem}\label{filtered:thm:regularity_up_to_boundary}
One has $U \in \dC^1(\cO)$.
\end{theorem}
\begin{proof}
The value function $U$ belongs to $\dC^2$ in the continuation region $\cC$ by Lemma \ref{filtered:lemma:weak_solution} and it belongs to $\dC^\infty$ on the interior of the stopping regions as $U \equiv \pm K_\pm$ on $\cS_\pm$, respectively.
Thus, it only remains to prove that the derivatives $U_x$ and $U_y$ are continuous up to the boundary.
As $U$ is constant in the interior of $\cS_+$ and of $\cS_-$, this amounts to prove that the limits of $U_x$ and $U_y$ as $(x,y)$ approaches $\partial\cC$ are equal to $0$.

\smallskip
We start with $U_x(x,y)$.
Let $(x,y) \in \cC$, $\epsilon > 0$ so that $(x+\epsilon,y) \in \cC$, which is always possible as $\cC$ is open.
Consider $0 \leq \frac{1}{\epsilon}(U(x+\epsilon,y) - U(x,y)) $, as $x \mapsto U(x,y)$ is non-decreasing by Lemma \ref{filtered:lemma_regularity_U}.
We estimate from above $U_x(x,y)$.
Let $(\tau^*,\theta^*)$ be the equilibrium stopping strategies for the Dynkin game starting at $(x,y) \in \cC$ and $(\tau^*_\epsilon,\theta^*_\epsilon)$ be the equilibrium stopping times for the game starting at $(x+\epsilon,y)$.
Recall the definition of $M_{(x,y)}(\tau,\theta)$ in \eqref{eq:filtered:zero_sum_functional} and notice that, as the Dynkin game has a value, we have
\begin{align*}
& U(x,y) = M_{(x,y)}(\tau^*,\theta^*) =  \inf_\tau \sup_\theta  M_{(x,y)}(\tau,\theta) = \sup_\theta M_{(x,y)}(\tau^*,\theta) \geq M_{(x,y)}(\tau^*,\theta^*_\epsilon), \\
& \begin{aligned}
U(x+\epsilon,y) & = M_{(x+\epsilon,y)}(\tau^*_\epsilon,\theta^*_\epsilon) = \sup_\theta \inf_\tau M_{(x+\epsilon,y)}(\tau,\theta) \\
& = \inf_\tau M_{(x+\epsilon,y)}(\tau,\theta^*_\epsilon) \leq M_{(x+\epsilon,y)}(\tau^*,\theta^*_\epsilon).
\end{aligned}
\end{align*}
This implies
\begin{align*}
    0 & \leq \frac{U(x+\epsilon,y) - U(x,y)}{\epsilon} 
    \leq \frac{1}{\epsilon}\E\left[\int_0^{\tau^* \land \theta^*_\epsilon} e^{-\delta t} \left( c' (X^{x+\epsilon,y}_t) - c' (X^{x,y}_t) \right) dt \right] \\
    & = \E\left[\int_0^{\tau^* \land \theta^*_\epsilon} e^{-2\delta t} \int_0^1  c''\big(X^{x,y}_t + r\epsilon e^{-\delta t} \big) dr dt \right],
\end{align*}
where the last equality follows from the fundamental theorem of calculus and by exploiting the linearity of the process $X^{x,y}$, which yields $X^{x+\epsilon,y}_t - X^{x,y}_t = \epsilon e^{-\delta t}$.
Recalling that $\vert c''(x)\vert \leq \alpha_2 (1 + \vert x \vert ^{p - 2})$ by Assumption \ref{application:assumptions_cost}, by using \eqref{eq:filtered:bound} we have the following bound:
\begin{align*}
    & \Bigg\vert \E\left[\int_0^{\tau^* \land \theta^*_\epsilon} e^{-2\delta t} \int_0^1  c''\big(X^{x,y}_t + r\epsilon e^{-\delta t}\big) dr dt \right] \Bigg\vert \leq \kappa \Big(1 +  \E\Big[\int_0^\infty e^{-2\delta t} \vert X^{x,y}_t \vert^{p-2} dt \Big] \Big) < \infty,
\end{align*}
for any $\epsilon \in (0,1]$.
We notice that $\theta^*_\epsilon \to \theta^*$ as $\epsilon \to 0$. This is a consequence of the probabilistic regularity of $\partial\cS_+$ for the interior $\mathring{\cS}_+$ relatively to $(X^0,\Pi)$ and of the strong Markov property of $(X^0,\Pi)$.
Indeed, it is easy to see that $(\theta^*_\epsilon)_{\epsilon > 0}$ is a decreasing sequence of stopping times and that it converges to 
\[
\theta^+ = \inf\{t \geq 0: X^{x,y}_t < a_+(\Pi^{y}_t)  \} = \inf\{ t \geq 0: (X^{x,y}_t,\Pi^{y}_t) \in \mathring{\cS}_+ \}.
\]
We show that $\theta^* = \theta^+$ $\P$-a.s. Since $\mathring{\cS}_+ \subseteq \cS_+$, it follows that $\theta^* \leq \theta^+$. We now show that $\P(\theta^+ > \theta^*) = 0$.
Since the equality $\theta^* = \theta^+$ is always satisfied when $\theta^* = \infty$, { we have $\{\theta^+ > \theta^*\} = \{ \theta^+ > \theta^*, \theta^* < \infty \}$}.
On the event $\{\theta^* < \infty \}$, $\theta^+$ can be expressed in terms of $\theta^*$ as
\[
\theta^+ = \theta^* + \inf\{ t \geq 0: (X^{x,y}_{t + \theta^*},\Pi^{x,y}_{t + \theta^*}) \in \mathring{\cS}_+ \}.
\]
Moreover, if $\theta^+ > \theta^*$, it holds $(X^{x,y}_{\theta^*},\Pi^{x,y}_{\theta^*}) \in \partial\cS_+ = \cS_+ \setminus \mathring{\cS}_+$.
Since $(X^{x,y}_{\theta^*},\Pi^{x,y}_{\theta^*}) \in \partial\cS_+$ on $\{\theta^* < \theta^+\}$, the first entry time of $(X^{x,y}_{t+\theta^*},\Pi^{x,y}_{t+\theta^*})$ in $\mathring{\cS}_+$ must occur at a strictly positive time. Thus, we have
\begin{multline*}
    \{ \theta^+ > \theta^* \} = \{ \theta^+ - \theta^* > 0, \theta^* < \infty, (X^{x,y}_{\theta^*},\Pi^{x,y}_{\theta^*}) \in \partial\cS_+ \} \\
    = \{ \inf\{ t > 0: (X^{x,y}_{t + \theta^*},\Pi^{x,y}_{t + \theta^*}) \in \mathring{\cS}_+ \} > 0, (X^{x,y}_{\theta^*},\Pi^{x,y}_{\theta^*}) \in \partial\cS_+, \theta^* < \infty \}.
\end{multline*}
Therefore,
\begin{multline*}
    \P(\theta^+ > \theta^* ) = \E[ \ind_{\{ \theta^* <\infty, (X^{x,y}_{\theta^*},\Pi^{x,y}_{\theta^*}) \in \partial\cS_+ \}} \ind_{\{ \inf\{ t > 0: (X^{x,y}_{t + \theta^*},\Pi^{x,y}_{t + \theta^*}) \in \mathring{\cS}_+ \} > 0 \}}  ] \\
    = \E[ \ind_{\{ \theta^* <\infty, (X^{x,y}_{\theta^*},\Pi^{x,y}_{\theta^*}) \in \partial\cS_+ \}} \E[ \ind_{\{ \inf\{ t > 0: (X^{x,y}_{t + \theta^*},\Pi^{x,y}_{t + \theta^*}) \in \mathring{\cS}_+ \} > 0 \}} \vert \cF_{\theta^*} ] ]
\end{multline*}
By relying on the strong Markov property of $(X^0,\Pi)$, on $\{\theta^*<\infty\}$, we have
\[
    \E[\ind_{\left\{\inf\{t>0:(X^{x,y}_{t+\theta^*},\Pi^{x,y}_{t+\theta^*})\in\mathring{\cS}_+\}>0\right\}}\,\vert \cF_{\theta^*}] =\P_{X^{x,y}_{\theta^*},\Pi^{x,y}_{\theta^*}}(\check{\sigma}_+>0),\quad \P\text{-a.s.},
\]
where, under $\P_{u,v}$, the stopping time $\check{\sigma}_+$ is defined by
\[
    \check{\sigma}_+:=\inf\{t>0:(X^{u,v}_t,\Pi^{u,v}_t)\in\mathring{\cS}_+\}.
\]
By Lemma \ref{filtered:lemma:probabilistic_regularity_interior}, we have $\P_{u,v}(\check{\sigma}_+>0) = \P(\check{\sigma}_+(u,v) > 0) =0$, for any $(u,v)\in\partial\cS_+$.
Therefore,
\[
    \P(\theta^+>\theta^*)=\E[\ind_{\{\theta^*<\infty,\,(X^{x,y}_{\theta^*},\Pi^{x,y}_{\theta^*})\in\partial\cS_+\}}\P_{X^{x,y}_{\theta^*},\Pi^{x,y}_{\theta^*}}(\check{\sigma}_+>0) ]=0.
\]
Hence $\theta^+=\theta^*$, $\P$-a.s., and so the desired convergence $\theta^*_\epsilon\to\theta^*$, $\P$-a.s., follows.
Since $c''$ is continuous over $(0,\infty)$, the dominated convergence theorem yields
\begin{equation}\label{filtered:thm:regularity_up_to_boundary:x_derivative_bound}
0 \leq  U_x(x,y) = \lim_{\epsilon \to 0} \frac{U(x+\epsilon,y)- U(x,y)}{\epsilon} \leq \E\left[\int_0^{\tau^*\land \theta^*} e^{-2\delta t} \big\vert c''\left(X^{x,y}_t \right) \big\vert dt \right], \;\; \forall (x,y) \in \cC.
\end{equation}
Now we send $(x,y)$ to $(x_o,y_o) \in \partial\cS_+$ (respectively, $\partial\cS_-$). By \eqref{eq:filtered:convergence_stopping_times}, we have $\tau^*(x,y) \to 0$ (respectively, $\theta^*(x,y) \to 0$), so that \eqref{filtered:thm:regularity_up_to_boundary:x_derivative_bound} and dominated convergence theorem imply
\[
0 \leq \varliminf_{(x,y) \to (x_o,y_o) \in \partial\cS_\pm} U_x(x,y) \leq \varlimsup_{(x,y) \to (x_o,y_o) \in \partial\cS_\pm} U_x(x,y) \leq 0,
\]
thus proving that $U_x$ is continuous across $\partial\cS_+$ (respectively, across $\partial\cS_-)$.

\smallskip
We now deal with the continuity of the partial derivative $U_y$ across the boundary of $\cC$.
Let $(x,y) \in \cC$, $\epsilon > 0$ so that $(x,y+\epsilon) \in \cC$, which is always possible as $\cC$ is open.
Consider $\frac{1}{\epsilon}(U(x,y+\epsilon) - U(x,y)) \geq 0$, as $x \mapsto U(x,y)$ is non-decreasing by Lemma \ref{filtered:lemma_regularity_U}.
We estimate from above $U_y(x,y)$.
Let $(\tau^*,\theta^*)$ be the equilibrium stopping strategies for the Dynkin game starting at $(x,y) \in \cC$ and $(\tau^*_\epsilon,\theta^*_\epsilon)$ be the equilibrium stopping times for the game starting at $(x,y+\epsilon)$.
As the times $\theta^*_\epsilon$ are sub optimal for the sup-player when starting at $(x,y)$ and $\tau^*$ is sub optimal for the inf-player when starting from $(x,y+\epsilon)$, we have again
\begin{equation}\label{eq:filtered:thm:regularity_up_to_boundary:bound_y_differences}
\begin{aligned}
    0 & \leq \frac{U(x,y+\epsilon) - U(x,y)}{\epsilon} \leq \frac{1}{\epsilon}\E\left[\int_0^{\tau^* \land \theta^*_\epsilon} e^{-\delta t} \left( c' (X^{x,y+\epsilon}_t) - c' (X^{x,y}_t) \right) dt \right] \\
    & = \E\left[\int_0^{\tau^* \land \theta^*_\epsilon} e^{-\delta t} \frac{X^{x,y+\epsilon}_t - X^{x,y}_t}{\epsilon} \int_0^1 c''\big(X^{x,y}_t + r(X^{x,y+\epsilon}_t - X^{x,y}_t)\big)dr dt \right].
\end{aligned}
\end{equation}
Set $\Delta\Pi^y_t \coloneqq \frac{1}{\epsilon}(\Pi^{x,y+\epsilon}_t-\Pi^{x,y}_t)$ and $\Delta X^{y}_t \coloneqq \frac{1}{\epsilon}(X^{x,y+\epsilon}_t - X^{x,y}_t)$. By direct calculations, their stochastic differentials are given by
\begin{equation*}
\begin{aligned}
    & d\Delta X^{y}_t = \left( \sigma\gamma \Delta\Pi^y_t -\delta \Delta X^{y}_t \right)dt, && \Delta X^{y}_0 = 0, \\
    & d \Delta\Pi^y_t = -(\lambda_1+\lambda_2)\Delta\Pi^y_t dt + \gamma\Delta\Pi^y_t\left(1-\Pi^{x,y+\epsilon}_t - \Pi^{x,y}_t \right) dB_t, && \Delta\Pi^y_0 = 1.
\end{aligned}
\end{equation*}
Let $R^y=(R^y_t)_{t \geq 0}$ be given by
\[
dR^y_t = -(\lambda_1+\lambda_2)R^y_tdt + \gamma R^y_t(1-2{ \Pi^{x,y}_t})dB_t, \quad R^y_0 = 1.
\]
{ Notice that $R^y$ does not depend on $x$ as $\Pi^{x,y}$ does not.}
By \cite[Theorem V.7.39]{protter2005stochastic}, we have $\Delta\Pi^y_t \to R^y_t$, $\P$-a.s., for all $t \geq 0$, as $\epsilon \to 0$, which implies $\Delta X^{y}_t \to e^{-\delta t} \sigma\gamma \int_0^t e^{\delta s} R^y_s ds$, $\P$-a.s., as well.
Then, provided that we can apply the dominated convergence theorem, it holds
\begin{equation}\label{eq:filtered:thm:regularity_up_to_boundary:bound_y_derivative}
    0 \leq U_y(x,y) \leq \sigma\gamma\E\left[\int_0^{\tau^* \land \theta^*} e^{-2\delta t} \Big(\int_0^t R^y_s ds\Big) c''\big(X^{x,y}_t\big) dt \right].
\end{equation}
By taking the limit as $\cC \ni (x,y) \to (x_o,y_o) \in \partial\cS_\pm$ and invoking \eqref{eq:filtered:convergence_stopping_times}, we deduce
\[
0 \leq \varliminf_{(x,y) \to (x_o,y_o) \in \partial\cS_\pm}U_y(x,y) \leq \varlimsup_{(x,y) \to (x_o,y_o) \in \partial\cS_\pm}U_y(x,y) \leq 0,
\]
thus proving continuity across the boundary of $\cC$.

The rest of the proof is dedicated to showing that dominated convergence theorem can be applied in \eqref{eq:filtered:thm:regularity_up_to_boundary:bound_y_differences} to obtain the lower bound on $U_y$ \eqref{eq:filtered:thm:regularity_up_to_boundary:bound_y_derivative}.
As one has $\Delta X^{y}_t = e^{-\delta t} \gamma\sigma\int_0^t e^{\delta s} \Delta\Pi^y_s ds$, we set
\begin{equation}\label{eq:thm:regularity_up_to_boundary:uniformly_integrable_rvs}
\Xi^\epsilon \coloneqq \sigma\gamma \int_0^{\tau^* \land \theta^*_\epsilon} e^{-2\delta t} \Big(\int_0^t \Delta\Pi^y_s ds\Big) \int_0^1 c''\big(X^{x,y}_t + r(X^{x,y+\epsilon}_t - X^{x,y}_t)\big)dr dt.
\end{equation}
We aim at showing that the family $(\Xi^\epsilon)_{\epsilon \in (0,\frac{1}{2}]}$ is bounded in $L^2$-norm, hence uniformly integrable.
By using again { $\vert c''(x) \vert \leq \alpha_2(1+\vert x \vert^{p-2})$} with $p \geq 2$ and the fact that $\Delta\Pi^y_t$ is always positive, we have
\begin{equation}\label{eq:thm:smooth_fit:first_bound_lambda}
\vert \Xi^\epsilon \vert \leq \kappa \int_0^{\infty} e^{-2\delta t} \Big( \int_0^t \Delta\Pi^y_s ds\Big) \left( 1 + \vert X^{x,y}_t \vert^{p-2} + \epsilon^{p-2} e^{-\delta(p-2)t} \Big(\int_0^t \Delta\Pi^y_s ds \Big)^{p-2}   \right) dt,
\end{equation}
{ for some positive constant $\kappa$.}
We suppose first $p \neq 2$.
By taking the expectation, applying Jensen's inequality, Fubini-Tonelli's theorem and using \eqref{eq:filtered:bound}, we deduce
\begin{align*}
& \E[\vert \Xi^\epsilon \vert^2] \leq \kappa_0 + \kappa_1 \E\bigg[  \int_0^{\infty} e^{-2\delta t} \Big( \Big( \int_0^t \Delta\Pi^y_s ds\Big)^2 +  \Big( \int_0^t \Delta\Pi^y_s ds\Big)^4 \Big) dt \\
& \;\; + \int_0^{\infty} e^{-p\delta t} \Big ( \int_0^t \Delta\Pi^y_s ds\Big)^{2(p-1)} dt \bigg] \\
& \leq \kappa_0 + \kappa_1 \bigg(\int_0^{\infty} e^{-2\delta t} \Big( t\int_0^t \E[(\Delta\Pi^y_s)^2] ds +  t^3\int_0^t \E[(\Delta\Pi^y_s)^4] ds \Big) dt \\
& \;\; +\int_0^{\infty} e^{-p\delta t} t^{2p - 3} \int_0^t \E[(\Delta\Pi^y_s)^{2(p-1)}] ds dt \bigg) \\
& \leq \kappa_0 + \kappa_1 \bigg(\int_0^{\infty} e^{-2\delta t} \Big( t\int_0^t \E[(\Delta\Pi^y_s)^2] ds +  t^3\int_0^t \E[(\Delta\Pi^y_s)^4] ds \Big) dt \\
& \;\; + \int_0^{\infty} e^{-p\delta t} t^{\alpha} \int_0^t \E[(\Delta\Pi^y_s)^{2(p-1)}] ds dt \bigg),
\end{align*}
where $\alpha = \lceil 2(p-1) - 1 \rceil$.
Integrating by parts, we finally get
\begin{multline}\label{eq:lemma:smooth_fit:integral_lambda_e}
    \E[\vert \Xi^\epsilon \vert^2 ] \leq \kappa_0 + \kappa_1\bigg(\int_0^{\infty} e^{-2\delta t} \left( \E[(\Delta\Pi^y_t)^2] +  g_2(t) \E[(\Delta\Pi^y_t)^4] \right) dt \\
    + \int_0^{\infty} e^{-p\delta t} g_\alpha(t) \E[(\Delta\Pi^y_t)^{2(p-1)}] dt \bigg),
\end{multline}
where $g_2(t)$ and $g_\alpha(t)$ are two suitable polynomials of degree $2$ and $\alpha$ respectively.
We notice that, for any $q \geq 1$, we have
\begin{equation}\label{eq:filtered:derivatives:delta_pi_power}
\begin{aligned}
& (\Delta\Pi^y_t)^q = \exp\left( -q(\lambda_1+\lambda_2)t +\frac{q(q-1)}{2}\gamma^2\int_0^t(1-\Pi^{y}_s - \Pi^{y-\epsilon}_s)^2 ds \right) M^{(q)}_t
\end{aligned}
\end{equation}
with $M^{(q)}$ a positive martingale, and $(1-\Pi^y_t - \Pi^{y-\epsilon}_t)^2 \leq 2$.
Then, we bound \eqref{eq:lemma:smooth_fit:integral_lambda_e} by
\begin{multline}\label{eq:lemma:smooth_fit:p_neq_2_bound}
    \kappa\int_0^{\infty} e^{-2\delta t} \Big( e^{-2(\lambda_1+\lambda_2)t + 2\gamma^2 t} + { g_2(t)}e^{-4(\lambda_1+\lambda_2)t + 12\gamma^2 t} \Big) dt \\
    +  \int_0^{\infty} e^{-p\delta t} { g_\alpha(t)} e^{-2(p-1)(\lambda_1 + \lambda_2)t + 2(p -1)(2p-3)\gamma^2 t } dt,
\end{multline}
which is finite if and only if condition \eqref{filtered:assumption_delta:p_not_2} in Assumption \ref{application:assumptions_discount} holds.
This implies $\sup_{\epsilon \in (0,1] }\E[\vert \Xi^\epsilon \vert^2 ] < \infty$ which concludes the proof in the case $p > 2$.

If $p=2$, we notice that \eqref{eq:thm:smooth_fit:first_bound_lambda} reduces to
\[
\vert \Xi^\epsilon \vert \leq \kappa \int_0^{\infty} e^{-2\delta t} \Big( \int_0^t \Delta\Pi^y_s ds\Big)dt.
\]
Then, by the same steps as before, we get
\begin{multline}\label{eq:lemma:smooth_fit:p_eq_2_bound}
    \E[\vert \Xi^\epsilon \vert^2]  \leq \kappa \int_0^\infty e^{-2\delta t} \E\left[ \Big( \int_0^t \Delta \Pi^y_t \Big)^2 \right]dt \leq \kappa \int_0^\infty e^{-2\delta t} t \Big( \int_0^t \E\Big[ \big(\Delta \Pi^y_s \big)^2 \Big] ds \Big)dt \\
    \leq \kappa \int_0^\infty e^{-2\delta t} \E\Big[ \big(\Delta \Pi^y_t \big)^2 \Big]dt \leq \kappa \int_0^\infty e^{-t(2\delta + 2(\lambda_1+\lambda_2)-2\gamma^2)}dt < \infty,
\end{multline}
by condition \eqref{filtered:assumption_delta:p_equal_2} in Assumption \ref{application:assumptions_discount}. This shows that $(\Xi^\epsilon)_{\epsilon \in (0,\frac{1}{2}]}$ is uniformly bounded in $L^2$.
\end{proof}

\begin{remark}\label{filtered:rmk:derivatives_bound}
For later use, we notice that the calculations in the proof of Theorem \ref{filtered:thm:regularity_up_to_boundary} imply the following bounds on the first-order derivatives of $U$:
\begin{equation}\label{eq:U_derivatives:final_bounds}
\vert U_x(x,y) \vert \leq \kappa_1(1+\vert x \vert^{p-2}), \qquad 
\vert U_y(x,y) \vert \leq \kappa_2, 
\end{equation}
for some positive constants $\kappa_1$ and $\kappa_2$.
In order to see this, we proceed as follows.
As for the $x$-derivative, recalling that $U_x \equiv 0$ on $\cS_-$ and $\cS_+$ and using \eqref{filtered:thm:regularity_up_to_boundary:x_derivative_bound}, Assumption \ref{application:assumptions_cost} and Lemma \ref{filtered:lemma:finiteness}, we have
\begin{multline*}
\vert U_x(x,y) \vert \leq \int_0^\infty  e^{-2\delta t}\E\left[ \big\vert c''\left(X^{x,y}_t \right) \big\vert \right]dt  \\
\leq \kappa\left( 1 + \int_0^\infty  e^{-2\delta t}\E\Big[  \vert X^{x,y}_t \vert ^{p-2} \Big]dt \right) \leq \kappa(1+\vert x \vert^{p-2}),
\end{multline*}
As for $U_y$, we have $U_y \equiv 0$ on $\cS_\pm$. 
Then, we recall that the random variables $(\Xi^\epsilon)_{\epsilon \in (0,\frac{1}{2}]}$ defined in \eqref{eq:thm:regularity_up_to_boundary:uniformly_integrable_rvs} are uniformly integrable, and that their $L^2$-norm { do not} depend on the initial value $y \in (0,1)$, as showed by \eqref{eq:lemma:smooth_fit:p_neq_2_bound} if $p > 2$ or by \eqref{eq:lemma:smooth_fit:p_eq_2_bound} if $p=2$.
Then, by dominated convergence theorem and \eqref{eq:filtered:thm:regularity_up_to_boundary:bound_y_derivative}, one has 
\begin{equation*}
\begin{aligned}
    & \vert U_y(x,y) \vert \leq \sigma\gamma\E\left[\int_0^{\tau^* \land \theta^*} e^{-2\delta t} \Big(\int_0^t R^y_s ds\Big) c''\big(X^{x,y}_t\big) dt \right] \leq \sup_{\epsilon \in (0,\frac{1}{2}]} \E[\vert \Xi^\epsilon \vert^2]^\frac{1}{2} \leq \kappa_2
\end{aligned}
\end{equation*}
with $\kappa_2$ given by the square root of right-hand side of \eqref{eq:lemma:smooth_fit:p_neq_2_bound} if $p > 2$ or of \eqref{eq:lemma:smooth_fit:p_eq_2_bound} if $p=2$.
\end{remark}

\subsubsection{Regularity refinements: a new coordinate system for the value function $U$}\label{sec:filtered:regularity_refinements}
In Section \ref{sec:application:partial_observation:filtered_problem}, we proved that Hypothesis \ref{conj:main} is satisfied and, moreover, that $U \in \dC^1(\cO)$.
Due to the degeneracy of the operator $\cL_{(X^0,\Pi)}$, we cannot infer more about the regularity of $U$.
In particular, it is not possible to prove that $U \in \bbW^{2,\infty}_{loc}$.
Thus, the existence of a pair $(V,\lambda)$ solution of \eqref{eq:HJB} and of an optimal control $\xi^\star$ will be proved by means of Theorem \ref{thm:HJB_sol_weak}, which requires less regularity.

Theorem \ref{thm:HJB_sol_weak} requires that $\cL_{(X^0,\Pi)}U$ exists almost everywhere and belongs to $L^\infty_{loc}(\cO)$.
In order to recover such regularity, we consider the transformation of the value function $\widehat{U}(x,z)$ associated with the change of coordinates $(X^0,Z)$ \eqref{eq:filtered:definition_z} and investigate its regularity properties.

Recall from Lemma \ref{filtered:lemma:change_of_variables} and equation \eqref{eq:filtered:diffeomorphism} the definition of the process $(X^0,Z)$ and of the diffeomorphism $\Psi$.
Set
\begin{equation}\label{eq:filtered:U_change_of_variables}
\widehat{U}(x,z) \coloneqq U\left( \Psi((x,z))\right) = U\left(x, \frac{e^{\frac{\gamma}{\sigma}(z+x)}}{1 + e^{\frac{\gamma}{\sigma}(z + x) }} \right) 
\end{equation}
so that
\begin{equation}\label{eq:filtered:U_hat_change_of_variables}
U(x,y) = \widehat{U}\left(\Psi^{-1}((x,y))\right) = \widehat{U}\left(x,\frac{\sigma}{\gamma}\log\left(\frac{y}{1-y}\right) - x\right).
\end{equation}
Notice that $\widehat{U}(x,z)$ has the same representation as in \eqref{eq:filtered:U}:
\begin{equation}\label{eq:filtered:dynkin_z}
\widehat{U}(x,z) \coloneqq \inf_{\tau} \sup_ {\theta} \E\left[\int_0^{\tau \land \theta} e^{-\delta t} c' (X^{x,z}_t)dt +K_+ e^{-\delta \tau}\ind_{\tau < \theta} - K_- e^{-\delta \theta}\ind_{\theta < \tau} \right],
\end{equation}
with the pair $(X^0,\Pi)$ given by \eqref{eq:filtering:filtered_dynamics} replaced by $(X^0,Z)$ given by \eqref{eq:filtering:filtered_dynamics_change_of_variables}.
Analogous properties as in Lemma~\ref{filtered:lemma_regularity_U} hold for $\widehat{U}$ as well.
Consider the sets $\widehat{\cC} \coloneqq \Psi^{-1}(\cC)$ and $\widehat{\cS}_\pm \coloneqq \Psi^{-1}(\cS_\pm)$.
Notice that, since $\Psi$ is a diffeomorphism, $\widehat{\cC}$ and $\widehat{\cS}_\pm$ are respectively open and closed sets.
In order to identify the boundaries of these sets, define the functions $b_\pm$ as follows: Let $\hat{a}_\pm(x)$ be the generalized inverse of $a_\pm(y)$, defined by
\[
\hat{a}_+(x) \coloneqq \sup\{ y \in (0,1): \, a_+(y) \leq x \}, \quad \hat{a}_-(x) \coloneqq \inf\{ y\in (0,1): \, a_-(y) \geq x \}.
\]
Set
\begin{equation}\label{eq:filtered:dynzin_z:barriers_b_x}
    b_\pm (x) \coloneqq \frac{\sigma}{\gamma}\log\Big( \frac{\hat{a}_\pm(x)}{1 - \hat{a}_\pm(x)}\Big) - x,
\end{equation}
which are well defined, as $\hat{a}_\pm(\R) \subseteq (0,1)$.
Since $\hat{a}_\pm$ are non-increasing, as $a_\pm$ are so, we deduce that $b_\pm$ are strictly decreasing, so that we can consider the inverses of $b_\pm$:
\begin{equation}\label{eq:filtered:dynzin_z:barriers_b_z}
    \hat{b}_\pm(z) \coloneqq (b_\pm)^{-1}(z).
\end{equation}
By using the barriers $\hat{b}_\pm$, the continuation and stopping regions related to $\widehat{U}$ can be expressed as follows:
\begin{equation}\label{eq:filtered:regions_z}
\begin{aligned}
    \widehat{\cC} & = \{ (x,z): \; -K_+ < \widehat{U}(x,z) < K_- \} = \{ (x,z): \; \hat{b}_+(z) < x < \hat{b}_-(z) \}, \\
    \widehat{\cS}_- & = \{ (x,z): \; \widehat{U}(x,z) \geq K_- \} = \{ (x,z): \; x \geq \hat{b}_-(z) \}, \\
    \widehat{\cS}_+ & = \{ (x,z): \; \widehat{U}(x,z) \leq -K_+ \} = \{ (x,z): \; x \leq \hat{b}_+(z) \}.
\end{aligned}
\end{equation}
Observe that the stopping times $(\tau^*,\theta^*)$ defined by \eqref{eq:filtered:optimal_stopping_times} are a saddle point for the Dynkin game with value function $\widehat{U}(x,z)$ as well, and they can be expressed as the first entry times of the process $( X^0,Z)$ in the sets $\widehat{\cS}_+$ and $\widehat{\cS}_-$ respectively.

\begin{lemma}\label{filtered:lemma:dynkin_z:value_function}
Let $\widehat{U}$ be the value function of the Dynkin game defined by \eqref{eq:filtered:dynkin_z}. Then, $\widehat{U} \in \dC^1(\R^2) \cap \dC^2(\widehat{\cC})$ and it solves
\begin{equation}\label{eq:filtered:dinkin_z:var_ineq}
\begin{dcases}
\cL_{(X^0,Z)}\widehat{U}(x,z) + c'(x) -\delta \widehat{U}(x,z) = 0, &\text{if }  \hat{b}_+(z) < x < \hat{b}_-(z), \\
\widehat{U}(x,z) =  -K_+, &\text{if } x \leq \hat{b}_+(z) , \\
\widehat{U}(x,z) = K_-, &\text{if } x \geq \hat{b}_-(z) .
\end{dcases}
\end{equation}
Moreover, $\widehat{U}_{xx} \in L^\infty_{\text{loc}}(\R^2)$.
\end{lemma}
\begin{proof}
The fact that $\widehat{U} \in \dC^1(\R^2) \cap \dC^2(\widehat{\cC})$ and it solves \eqref{eq:filtered:dinkin_z:var_ineq} follows easily by definition of $\widehat{U}$, $\widehat{\cC}$, Lemma~\ref{filtered:lemma:weak_solution} and Theorem \ref{filtered:thm:regularity_up_to_boundary}.
It remains to show that $\widehat{U}_{xx} \in L^\infty_{loc}(\R^2)$.
To accomplish that, first of all notice that $\widehat{U}_{xx}$ admits a continuous extension to the closure of $\widehat{\cC}$, which we still denote by $\widehat{U}_{xx}$.
As a matter of fact, for any $(x,z) \in \widehat{\cC}$, it holds
\[
\widehat{U}_{xx}(x,z) = \frac{2}{\sigma^2} \left( -q(x,z)\widehat{U}_z(x,z) - \mu(x,z) \widehat{U}_{x}(x,z) + \delta \widehat{U}(x,z) - c'(x) \right),
\]
and the right-hand side of the latter is continuous over $\R^2$ since $\widehat{U} \in \dC^1(\R^2)$.
Then, take $(x_o,z_o) \in \partial\widehat{\cC}$.
For any $x \geq \hat{b}_-(z_o)$, we have
\begin{multline*}
\vert \widehat{U}_x(x,z_o) - \widehat{U}_x(x_o,z_o) \vert \\
\leq \int_{x_o}^{\hat{b}_-(z_o)}\vert \widehat{U}_{xx}(x',z_o) \vert dx' +  \int_{\hat{b}_-(z_o)}^x \vert \widehat{U}_{xx}(x',z_o) \vert dx'  \leq \kappa(x_o,z_o)\vert x - x_0 \vert.
\end{multline*}
The claimed local boundedness of $\widehat{U}_{xx}$ is then proved, being this trivially satisfied in the interior of $\widehat{\cC}$, $\widehat{\cS}_+$ and $\widehat{\cS}_-$.
\end{proof}

Lemma \ref{filtered:lemma:dynkin_z:value_function} allows us to improve the regularity of $U$, in the sense described by the following lemma.
\begin{lemma}\label{filtered:lemma:integral_differential_operator}
$\cL_{(X^0,\Pi)}U$ exists for a.e. $(x,y) \in \cO$ and it belongs to $L^\infty_{loc}(\cO)$.
Moreover, it holds
\begin{equation}\label{eq:filtered:integral_differential_operator}
    \cL_{(X^0,\Pi)}U(x,y) = \cL_{(X^0,Z)}\widehat{U}\big(x,z(x,y)\big), \quad \text{for a.e.} \, (x,y) \in \cO.
\end{equation}
\end{lemma}
\begin{proof}
As $U$ is constant on $\cO \setminus \overline{\cC}$ and invoking Theorem \ref{filtered:thm:regularity_up_to_boundary}, we deduce that $U \in \dC^2(\cO \setminus \partial\cC)$. Analogously, $\widehat{U} \in \dC^2(\R^2 \setminus \partial\widehat{\cC})$.
Moreover, as $\cL_{(X^0,\Pi)}$ and $\cL_{(X^0,Z)}$ are $\Psi$-diffeomorphic, it holds $\cL_{(X^0,\Pi)}U(x,y) = \cL_{(X^0,Z)}\widehat{U}\big(x,z(x,y)\big)$ for any $(x,y) \in \cO \setminus \partial\cC$.
By Lemma \ref{filtered:lemma:dynkin_z:value_function}, $\cL_{(X^0,\Pi)}\widehat{U}$ is defined for a.e. $(x,z) \in \R^2$ and it belongs to $L^\infty_{loc}(\R^2)$.
This concludes the proof.
\end{proof}

\subsubsection{Solution of the ergodic singular control problem}\label{sec:filtered:conclusion}
In this section, we conclude the study of the ergodic singular control problem, by showing that the assumptions of Theorem \ref{thm:HJB_sol_weak} are satisfied.
In particular, in Lemma \ref{filtered:lemma:candidate_potential:derivatives} we build the candidate pair $(V,\lambda)$ as prescribed by \eqref{eq:main:built_potential_weak} and \eqref{eq:main:built_value_weak}, we show that $V \in \bbW^{2,\infty}_{loc}(\cO) \cap \dC^2(\overline{\cC})$ and that $\lambda$ is continuous and bounded.
Then, in Lemma \ref{filtered:lemma:reflection} we build a candidate optimal control $\xi^\star$ as in \eqref{eq:thm:HJB_sol_weak:optimal_ctrl}.
Finally, in Theorem \ref{filtered:thm:optimal_control}, we verify the remaining assumptions, thus completely solving the ergodic singular stochastic control problem \eqref{eq:application:value}.

\smallskip
Recall that Hypothesis \ref{conj:main} is satisfied, so that, in particular, it holds $\sup_{y \in (0,1)}a_+(y) < \inf_{y \in (0,1)}a_-(y)$.
\begin{lemma}\label{filtered:lemma:candidate_potential:derivatives}
Let $\alpha \in (\sup a_+(y),\inf a_-(y) )$.
Consider the function $V$ given by \eqref{eq:main:built_potential_weak}.
It holds $V  \in \dC^1(\cO)$, and $V _{xy}$, $V _{xx}$ belong to $\dC(\cO)$ and $V _{yy} \in L^\infty_{loc}(\cO) \cap \dC(\overline{\cC})$.
In particular, its partial derivatives are given by
\begin{equation}\label{eq:filtered:candidate:derivatives}
\begin{aligned}
    V _{x}(x,y) & = \widehat{U}\big(x,z(x,y)\big) \\
    V _{xx}(x,y) & = \widehat{U}_{x}\big(x,z(x,y)\big) - \widehat{U}_{z}\big(x,z(x,y)\big), \\
    V _{xy}(x,y) & = \widehat{U}_{z}\big(x,z(x,y)\big)\frac{\sigma}{\gamma}\frac{1}{y(1-y)}, \\
    V _{y}(x,y) & =\frac{\sigma}{\gamma}\frac{1}{y(1-y)}\int_\alpha^x \widehat{U}_{z}(x',z(x',y)) dx' , \\
    V _{yy}(x,y) & = \frac{\sigma}{\gamma}\frac{2y-1}{y^2(1-y)^2} \int_\alpha^x \widehat{U}_{z}(x',z(x',y))dx' \\
    & + \frac{\sigma^2}{\gamma^2}\frac{1}{y^2(1-y)^2}\left( \widehat{U}_{z}\big(\alpha,z(\alpha,y)\big)  - \widehat{U}_{z}\big(x,z(x,y)\big) \right) \\
    & + \frac{\sigma^2}{\gamma^2}\frac{1}{y^2(1-y)^2}\left( \widehat{U}_{x}\big(\alpha,z(\alpha,y)\big) - \widehat{U}_{x}\big(x,z(x,y)\big) + \int_{\alpha}^{x}\widehat{U}_{xx}\big(x', z(x',y) \big)dx'\right).
\end{aligned}
\end{equation}
\end{lemma}
\begin{proof}
By exploiting the relationship \eqref{eq:filtered:U_hat_change_of_variables} between $U$ and $\widehat{U}$ and relying on the change of variable $q=z(x',y) = \frac{\sigma}{\gamma}\log\big(\frac{y}{1-y}\big)-x'$, we deduce the following representation for $V$:
\begin{equation}\label{eq:filtered:lemma_regularity_candidate:change_of_variables}
    V (x,y) = \int_\alpha^x \widehat{U}\big(x',z(x',y)\big)dx' = \int_{\frac{\sigma}{\gamma}\log(\frac{y}{1-y}) -x}^{\frac{\sigma}{\gamma}\log(\frac{y}{1-y}) - \alpha}  \widehat{U}\big(\frac{\sigma}{\gamma}\log(\frac{y}{1-y}) - q,q\big)dq.
\end{equation}
We notice that the { limits of integration} in the second equality of \eqref{eq:filtered:lemma_regularity_candidate:change_of_variables} are given by $z(x,y)$ and $z(x,\alpha)$ respectively.
By noticing that $x'\left(z(x,y),y \right) = x$, using the chain rule and the equalities 
\begin{align*}
& z_{x}(x,y) = -1, \quad z_{y}(x,y) = \frac{\sigma}{\gamma} \frac{1}{y(1-y)}, \quad z_{yy}(x,y) = \frac{\sigma}{\gamma} \frac{2y - 1}{y^2(1-y)^2},
\end{align*}
it follows from direct computations that the $x$-partial derivatives take the form described by \eqref{eq:filtered:candidate:derivatives}.
As for the derivatives with respect to $y$, we have
\begin{equation}\label{eq:filtered:candidate:y_derivatives1}
\begin{aligned}
    V _{y} & (x,y) \\
    & = \frac{\sigma}{\gamma}\frac{1}{y(1-y)}\left( \widehat{U}\big(\alpha,z(\alpha,y)\big) - \widehat{U}\big(x,z(x,y)\big) + \int_{z(x,y)}^{z(\alpha,y)}\widehat{U}_{x}\big(\frac{\sigma}{\gamma}\log(\frac{y}{1-y}) -q,q\big)dq\right), \\
    V _{xy} & (x,y) =  \frac{\sigma}{\gamma}\frac{1}{y(1-y)} \widehat{U}_{z}\big(x,z(x,y)\big), \\
    V _{yy} & (x,y) \\
    & = \frac{\sigma}{\gamma}\frac{2y-1}{y^2(1-y)^2}\left( \widehat{U}\big(\alpha,z(\alpha,y)\big) - \widehat{U}\big(x,z(x,y)\big) +\int_{z(x,y)}^{z(\alpha,y)}\widehat{U}_{x}\big(\frac{\sigma}{\gamma}\log(\frac{y}{1-y}) -q,q\big)dq\right) \\
    & + \frac{\sigma^2}{\gamma^2}\frac{1}{y^2(1-y)^2}\left( \widehat{U}_{z}\big(\alpha,z(\alpha,y)\big)  - \widehat{U}_{z}\big(x,z(x,y)\big) \right) \\
    & +\frac{\sigma^2}{\gamma^2}\frac{1}{y^2(1-y)^2}\left( \widehat{U}_{x}\big(\alpha,z(\alpha,y)\big) - \widehat{U}_{x}\big(x,z(x,y)\big) + \int_{z(x,y)}^{z(\alpha,y)}\widehat{U}_{xx}\big(\frac{\sigma}{\gamma}\log(\frac{y}{1-y}) -q,q\big)dq\right).
\end{aligned}    
\end{equation}
To conclude, we notice that the following equality holds:
\begin{align*}
    \int_{z(x,y)}^{z(\alpha,y)} & \widehat{U}_{x}\big(\frac{\sigma}{\gamma}\log(\frac{y}{1-y}) -q,q\big)dq = \int_\alpha^x \widehat{U}_{x}(x',z(x',y))dx'  \\
    & = \int_\alpha^x\Big(U_x(x',y)dx' - \widehat{U}_{z}(x',z(x',y))z_{x}(x',y)\Big)dx' \\
    & = \widehat{U}\big(x,z(x,y)\big)-\widehat{U}\big(\alpha,z(\alpha,y)\big) + \int_\alpha^x\widehat{U}_{z}(x',z(x',y)) dx'.
\end{align*}
Using this identity in \eqref{eq:filtered:candidate:y_derivatives1} yields to \eqref{eq:filtered:candidate:derivatives}.
The regularity properties of $V $ and its derivatives follow directly from the ones of $\widehat{U}$ (cf. Lemma \ref{filtered:lemma:dynkin_z:value_function}).
\end{proof}

We now build an optimal control $\xi^\star=(\xi^{\star,+},\xi^{\star,-}) \in \cB$.
Consider the process $\xi^\star=(\xi^{\star,+},\xi^{\star,-})$ so that the pair $(X^{\xi^\star},\xi^\star)$ solves the Skorohod reflection problem:
\begin{equation}\label{eq:filtered:skorohod_reflection}
\left\{ \begin{aligned}
    & a_+(\Pi_t) \leq X^{\xi^\star}_t \leq a_-(\Pi_t), \quad \P_{x,y}\text{-a.s.,} \;\; \text{for almost all }t \geq 0,\\
    & \xi^{\star,+}_t = \int_0^t \ind_{\{ X^{\xi^\star}_{s-} \leq a_+(\Pi_s) \}}d\xi^{\star,+}_s, \quad \xi^{\star,-}_t = \int_0^t \ind_{\{ X^{\xi^\star}_{s-} \geq a_-(\Pi_s) \}}d\xi^{\star,-}_s, \quad \P_{x,y}\text{-a.s.,} \;\; \forall t \geq 0, \\
    & \int_0^{\Delta\xi^{\star,+}_t} \ind_{\{(X^{\xi^{\star,+}}_t + z,\Pi_t) \in \cC\}} dz + \int_0^{\Delta\xi^{\star,-}_t} \ind_{\{(X^{\xi^\star}_t - z,\Pi_t) \in \cC\}} dz = 0, \quad \P_{x,y}\text{-a.s.,} \;\; \forall t \geq 0.
\end{aligned} \right.
\end{equation}
where $\Delta\xi^{\star,\pm}_t = \xi^{\star,\pm}_t - \xi^{\star,\pm}_{t-}$.

\smallskip
By using the similar techniques as in \cite[Section 4.3]{federico_pham2014sicon} and \cite[Section 6.1]{federico2023inventory_unknown_demand}, it is possible to explicitly build the solution to the Skorohod reflection problem \eqref{eq:filtered:skorohod_reflection}, as shown by the following Lemma:
\begin{lemma}\label{filtered:lemma:reflection}
There exists a solution to  \eqref{eq:filtered:skorohod_reflection}. Moreover, $\xi^\star$ is admissible.
\end{lemma}
\begin{proof}
For any $\xi \in \cB$, define $(\Lambda,\bar{\xi})$ by setting
\begin{equation}\label{filtered:lemma:reflection:alternative_state}
d\bar{\xi}^\pm_t = e^{\delta t} d\xi^\pm_t, \quad \Lambda_t = x + \bar{\xi}^+_t - \bar{\xi}^-_t.
\end{equation}
By using $\bar{\xi}$, we represent $X^\xi$ as 
\[
X^{\xi}_t = e^{-\delta t}x + \cX^0_t + e^{-\delta t} \left( \int_0^t e^{\delta s} d\xi^+_t - \int_0^t e^{\delta s} d\xi^+_t \right) =  e^{-\delta t}x + \cX^0_t + e^{-\delta t} \bar{\xi}^+_t - e^{-\delta t} \bar{\xi}^-_t ,
\]
where $\cX^0$ denotes the uncontrolled process starting from $x=0$.
We then express the constraint $a_+(\Pi_t) \leq X^{\xi^\star}_t \leq a_-(\Pi_t)$ as
\begin{equation}\label{eq:filtered:reflection:stochastic_constraint}
    \nu^+_t \coloneqq e^{\delta t}\left( a_+(\Pi_t) - \cX^0_t \right)  \leq x + \bar{\xi}^{\star,+}_t - \bar{\xi}^{\star,-}_t \leq e^{\delta t}\left( a_-(\Pi_t) - \cX^0_t \right) =: \nu^-_t,
\end{equation}
where $\nu^\pm=(\nu^\pm_t)_{t \geq 0}$ are respectively left-continuous and right-continuous adapted processes.
Thus, the reflection problem \eqref{eq:filtered:skorohod_reflection} can be then restated equivalently in terms of $(\Lambda,\bar{\xi}^\star)$ as
\begin{equation}\label{filtered:lemma:reflection:alternative_reflection}
\left\{ \begin{aligned}
    & \nu^+_t \leq \Lambda_t \leq \nu^-_t, \quad \P_{x,y}\text{-a.s.,} \;\;  \text{for almost all }t \geq 0,\\
    & \bar{\xi}^{\star,+}_t = \int_0^t \ind_{\{ \Lambda_{s-} \leq \nu^+_s \}}d\bar{\xi}^{\star,-}_s, \quad \bar{\xi}^{\star,-}_t = \int_0^t \ind_{\{ \Lambda_{s-} \geq \nu^-_s \}}d\bar{\xi}^{\star,-}_s, \quad \P_{x,y}\text{-a.s.,} \;\;  \forall t \geq 0, \\
    & \int_0^{\Delta\bar{\xi}^{\star,+}_t} \ind_{\{ \nu^+_t <  \Lambda_t + z < \nu^-_t \}} dz + \int_0^{\Delta\bar{\xi}^{\star,-}_t} \ind_{\{ \nu^+_t <  \Lambda_t - z < { \nu^-_t }\}} dz = 0, \quad \P_{x,y}\text{-a.s.,} \;\; \forall t \geq 0.
\end{aligned} \right.
\end{equation}
We build the process $\bar{\xi}^\star = (\bar{\xi}^{\star,+},\bar{\xi}^{\star,-})$ so that \eqref{filtered:lemma:reflection:alternative_reflection} is satisfied.
Then, it will be enough to set
\[
\xi^{\star,\pm}_t = \int_0^t e^{-\delta s} d\bar{\xi}^{\star,\pm}_s,
\]
as the pair $(X^{\xi^\star},\xi^\star)$ satisfies \eqref{eq:filtered:skorohod_reflection} by construction.

\smallskip
To this extent, consider the stopping times
\begin{equation}\label{filtered:lemma:existence_skorohod_reflection:tau_0}
\begin{aligned}
    & \tau^+_0 \coloneqq \inf\{ t \geq 0 : \, x < \nu^+_t \} = \inf\{ t \geq 0 : \,  { e^{-\delta t} x + \cX^0_t <  a_+(\Pi_t) }\}, \\
    & \tau^-_0 \coloneqq \inf\{ t \geq 0 : \, x > \nu^-_t \} = \inf\{ t \geq 0 : \, { e^{-\delta t} x + \cX^0_t  > a_-(\Pi_t) }\} , \\
    & \tau_0 \coloneqq \tau^+_0 \land \tau^-_0.
\end{aligned}
\end{equation}
Notice that, because $\inf_{y \in (0,1)}\big( a_-(y) - a_+(y) \big) >0$ by Lemma \ref{filtered:lemma:regularity_boundary}, we have $\{ \tau_0^+ = \tau_0^-\}=\{\tau_0=\infty\}$.
Let $\Omega_+ = \{\tau^+_0 < \tau^-_0 \}$, $\Omega_- = \{\tau^-_0 < \tau^+_0 \}$ and $\Omega_\infty = \{\tau_0 = \infty \}$.
We recursively define $\bar{\xi}^\star$.
Set $\Lambda^0_t = x$ for every $t \geq 0$, and for every $k \geq 1$, define:
\begin{align*}
\text{If $k\geq 1$ is odd,} \quad \Lambda^k_t &:=
\begin{cases}
x, \,\, & \text{on}\,\, \Omega_{\infty}, \\
x + \max_{s \in [\tau_{k-1},t]}\big(\nu^+_s - x \big)^+, \,\, & \text{on}\,\, \Omega_{+}, \\
x + \min_{s \in [\tau_{k-1},t]}\big(\nu^-_s - x \big)^-, \,\, & \text{on}\,\, \Omega_{-}, 
\end{cases} \\
\text{with} \quad \tau_k &:=
\begin{cases}
\infty, \,\, & \text{on}\,\,\Omega_{\infty}, \\
\inf\{t\geq \tau_{k-1}:\, \Lambda^k_t > \nu^-_t \},\,\, & \text{on}\,\, \Omega_{+},\\
\inf\{t\geq \tau_{k-1}:\, \Lambda^k_t < \nu^+_t \},\,\, & \text{on}\,\, \Omega_{-}.
\end{cases} \\
\text{If $k\geq 2$ is even,} \quad \Lambda^k_t &:=
\begin{cases} 
x, \,\, & \text{on}\,\, \Omega_{\infty}, \\
x + \max_{s \in [\tau_{k-1},t]}\big(\nu^+_s - x \big)^+, \,\, & \text{on}\,\, \Omega_{-}, \\
x + \min_{s \in [\tau_{k-1},t]}\big(\nu^-_s - x \big)^-, \,\, & \text{on}\,\, \Omega_{+}, 
\end{cases} \\
\text{with} \quad \tau_k &:=
\begin{cases} 
\infty, \,\, & \text{on}\,\, \Omega_{\infty}, \\
\inf\{t\geq \tau_{k-1}:\, \Lambda^k_t > \nu^-_t \},\,\, & \text{on}\,\, \Omega_{-},\\
\inf\{t\geq \tau_{k-1}:\, \Lambda^k_t < { \nu^+_t} \},\,\, & \text{on}\,\, \Omega_{+}.
\end{cases}
\end{align*}
In light of these definitions, and setting $\tau_{-1} \coloneqq 0$, one can then proceed as in \cite[Section 4.3]{federico_pham2014sicon} in order to conclude that the pair
\begin{equation}
    \Lambda^{\bar{\xi}^\star}_t = \sum_{k=0}^\infty \ind_{[\tau_{k-1},\tau_k)}(t)\Lambda^k_t, \quad \bar{\xi}^\star_t = \Lambda^{\bar{\xi}^\star}_t - x
\end{equation}
is a solution to the reflection problem \eqref{filtered:lemma:reflection:alternative_reflection}.
{ The minimality properties in the second and third equation of \eqref{filtered:lemma:reflection:alternative_reflection}} follow directly from \cite[Section 4.3]{federico_pham2014sicon} (see in particular Lemma 4.11 therein).
Finally, we notice that, thanks to the boundedness of $U$, we have $\vert V(x,y) \vert \leq \kappa (1+\vert x \vert)$.
Moreover, as the control $\xi^\star$ keeps the process in $[a_+(\Pi_t),a_-(\Pi_t)]$ and $\underline{a}_+ \leq a_+(y) < a_-(y) \leq \overline{a}_-$ for any $y \in (0,1)$ by Lemma \ref{filtered:lemma:regularity_boundary}, we have that $X^{\xi^\star}_T$ is uniformly bounded, $\P_{x,y}$-a.s. This concludes the proof.
\end{proof}

Before showing that $\xi^\star$ is optimal for the original singular control problem, we show that $\Pi$ admits a stationary distribution. By Remark \ref{rmk:ergodicity}, this entails that the problem admits a unique value $\lambda^\star$ regardless of the initial data.

\begin{lemma}\label{filtered:lemma:stationary_distribution}
The process $\Pi$ admits a stationary distribution.
\end{lemma}
\begin{proof}
By \cite[Lemma 23.19]{kallenberg_foundations} it is enough to verify that the density of the speed measure $m'(x)$ of the process $\Pi$ is integrable over the state space $[0,1]$.
By definition, we have
\[
m'(x) = \frac{2}{\gamma^2 x^2(1-x)^2}e^{ 2\int_a^x \frac{\lambda_2 - (\lambda_1+\lambda_2)y}{\gamma^2 y^2(1-y)^2} dy}.
\]
where $a \in (0,1)$.
Without loss of generality, we choose $a=\frac{1}{2}$.
We first verify integrability in a neighborhood of $0$.
For $\epsilon > 0$, we have
\begin{align*}
  \int_0^\epsilon & m'(x) dx = \int_0^\epsilon \frac{2}{\gamma^2 x^2(1-x)^2} e^{ -\frac{2\lambda_2}{\gamma^2}\int_x^{\frac{1}{2}} \frac{1}{ y^2(1-y)^2} dy } e^{\frac{2(\lambda_1+\lambda_2)}{\gamma^2}\int_x^{\frac{1}{2}}\frac{1}{y (1-y)^2} { dy}} dx\\
  & \leq \int_0^\epsilon \frac{2}{\gamma^2 x^2(1-x)^2} e^{ -\frac{2\lambda_2}{\gamma^2}\int_x^{\frac{1}{2}} \frac{1}{ y^2(1-y)^2} dy } e^{ { \frac{8(\lambda_1+\lambda_2)}{\gamma^2} }\int_x^{\frac{1}{2}}\frac{1}{y}{ dy}} dx \\
  & \leq \int_0^\epsilon \frac{2}{\gamma^2 x^2(1-x)^2} e^{ -\frac{2\lambda_2}{\gamma^2}\int_x^{\frac{1}{2}} \frac{1}{ y^2(1-y)^2} dy } e^{{ \frac{8(\lambda_1+\lambda_2)}{\gamma^2} } ( \ln( \frac{1}{2}) - \ln(x) )  } dx \\
  & \leq \kappa \int_0^\epsilon \frac{2}{\gamma^2 x^2(1-x)^2} e^{ -\frac{2\lambda_2}{\gamma^2}\int_x^{\frac{1}{2}} \frac{1}{ y^2(1-y)^2} dy } e^{\ln \big( x^{-{ \frac{8(\lambda_1+\lambda_2)}{\gamma^2} }} \big ) } dx \\
  & = \kappa\int_0^\epsilon \frac{2}{\gamma^2 x^2(1-x)^2} e^{ -\frac{2\lambda_2}{\gamma^2}\int_x^{\frac{1}{2}} \frac{1}{ y^2(1-y)^2} dy } x^{-{ \frac{8(\lambda_1+\lambda_2)}{\gamma^2} } } dx
\end{align*}
for some positive constant $\kappa$, where we used the inequality { $(1-y)^2 \geq \tfrac{1}{4}$} for $y \in (0,\tfrac{1}{2})$ in the first estimate. By exploiting the inequality $(1-y)^2 \leq 1$, we bound the last term above with
\begin{align*}
  \kappa\int_0^\epsilon & \frac{2}{\gamma^2 x^2(1-x)^2} e^{ -\frac{2\lambda_2}{\gamma^2}\int_x^{\frac{1}{2}} \frac{1}{ y^2 } dy } x^{-{ \frac{8(\lambda_1+\lambda_2)}{\gamma^2} }} dx \leq \kappa \int_0^\epsilon \frac{2}{\gamma^2 x^2(1-x)^2} x^{-{ \frac{8(\lambda_1+\lambda_2)}{\gamma^2} } } e^{ -\frac{2\lambda_2}{\gamma^2} \frac{1}{x} } dx,
\end{align*}
which is finite.
As for integrability in a neighborhood of $1$, we notice that the process $\Pi$ is symmetric, in the sense that $\Pi$ and $1-\Pi$ solve the same equation \eqref{eq:filtering:filtered_dynamics} with $\lambda_1$ and $\lambda_2$ inverted.
Then, the same computations show that the density of the speed measure of $1 - \Pi$ is integrable in $0$, and so is the density of the speed measure of $\Pi$ in $1$.
\end{proof}

Let $\alpha \in (\sup a_+(y),\inf a_-(y) )$, and recall the definition of $(V,\lambda)$ from \eqref{eq:main:built_potential_weak} and \eqref{eq:main:built_value_weak}.
We are finally ready to show that the control $\xi^\star$ solves the original singular stochastic control problem.
\begin{theorem}[Optimal control]\label{filtered:thm:optimal_control}
The policy $\xi^\star$ solution to the Skorohod reflection problem \eqref{eq:filtered:skorohod_reflection} is optimal for the ergodic stochastic singular control problem.
Moreover, the problem has a value, given by
\[
\lambda^\star = \varlimsup_{T \uparrow \infty}\frac{1}{T}\E_{y}\left[\int_0^T \lambda(\Pi_t)dt\right] = \inf_{\xi \in \cB} \varlimsup_{T \to +\infty} \frac{1}{T} \, \E_{x,y}\left[\int_0^T c(X_t^\xi) \, dt + K_+ \xi^+_T + K_- \xi^-_T\right].
\]
\end{theorem}
\begin{proof}
We verify that the assumptions of Theorem \ref{thm:HJB_sol_weak} are satisfied.
By Lemma \ref{filtered:lemma:candidate_potential:derivatives}, $V \in \bbW^{2,\infty}_{loc}(\cO) \cap \dC^2(\cC)$.
We now verify that equality \eqref{thm:HJB_sol_weak:differential_identity} is satisfied.
By explicit computation using the derivative of $V$ given by \eqref{eq:filtered:candidate:derivatives}, it holds
\begin{align*}
    \frac{1}{2} & \sigma^2V_{xx}(x,y) + \sigma\gamma y(1-y)V_{xy}(x,y)  + \frac{1}{2}\gamma^2y^2(1-y)^2V_{yy}(x,y) \\
    & =  \frac{1}{2}\sigma^2\Big( \widehat{U}_{z}\big(\alpha,z(\alpha,y)\big) + \widehat{U}_{x}\big(\alpha,z(\alpha,y)\big) \Big) \\
    & \:\: + \int_{\alpha}^{x}\Big( \frac{1}{2}\sigma^2\widehat{U}_{xx}\big(x',z(x',y)\big) +\frac{1}{2}\sigma\gamma(2y-1)\widehat{U}_{z}\big(x', z(x',y) \big) \Big)dx'.
\end{align*}
As for the first order derivatives, by using the fundamental theorem of calculus, the identity $b(x,y) = \mu(x,z(x,y))$ and \eqref{eq:filtered:candidate:derivatives}, we deduce
\begin{align*}
    b & (x,y) V_{x}(x,y) = b(x,y) U(x,y) = \int_\alpha^x \frac{d}{dx'}\Big[ b(x',y) U(x',y) \Big]dx' + b(\alpha,y)U(\alpha,y)  \\
    & = \int_\alpha^x \Big( -\delta \widehat{U}(x',z(x',y)) + \mu(x',z(x,y))\widehat{U}_{x}(x',z(x',y))  \\
    & \quad + \mu(x',z(x',y))z_{x}(x',y)\widehat{U}_{z}(x',z(x',y)) \Big)dx' \\
    & \quad + \mu\big(\alpha,z(\alpha,y)\big)\widehat{U}\big(\alpha,z(\alpha,y)\big)
\end{align*}
and
\begin{multline*}
    \big( \lambda_2 - (\lambda_1 + \lambda_2)y \big) V_{y} (x,y) = \big( \lambda_2 - (\lambda_1 + \lambda_2)y \big)\int_\alpha^x U_y(x',y)dx' \\
    = \int_\alpha^x \big( \lambda_2 - (\lambda_1 + \lambda_2)y \big)z_{y}(x',y) \widehat{U}_{z}(x',z(x',y))dx'.
\end{multline*}
Therefore, we have
\begin{equation}\label{eq:thm:optimal_control:PDE1}
\begin{aligned}
    & \cL_{(X^0,\Pi)}V(x,y) + c(x) \\
    & =  \mu\big(\alpha,z(\alpha,y)\big)\widehat{U}\big(\alpha,z(\alpha,y)\big)+c(\alpha) +\frac{1}{2}\sigma^2\Big(\widehat{U}_{z}\big(\alpha,z(\alpha,y)\big) +\widehat{U}_{x}\big(\alpha,z(\alpha,y)\big) \Big) \\
    & + \int_\alpha^x \bigg( -\delta \widehat{U}(x',z(x',y)) + c'(x') + \mu(x',z(x,y))\widehat{U}_{x}(x',z(x',y)) + \frac{1}{2}\sigma^2\widehat{U}_{xx}\big(x', z(x',y) \big)  \\
    & + \Big(\big( \lambda_2 - (\lambda_1 + \lambda_2)y \big)z_{y}(x',y) + \mu(x',z(x',y))z_{x}(x',y) +\frac{1}{2}\sigma\gamma(2y-1) \Big) \widehat{U}_{z}(x',z(x',y)) \bigg)dx'.
\end{aligned}
\end{equation}
Notice that it holds
\begin{align*}
q\big(x,z(x,y)\big) & = \big( \lambda_2 - (\lambda_1 + \lambda_2)y \big)z_{y}(x,y) + \mu(x,z(x,y))z_{x}(x,y) +\frac{1}{2}\sigma\gamma(2y-1),
\end{align*}
so that \eqref{eq:thm:optimal_control:PDE1} can be rewritten as
\begin{equation}\label{eq:thm:optimal_control:PDE2}
\begin{aligned}
    & \cL_{(X^0,\Pi)}V(x,y) + c(x) =  \mu\big(\alpha,z(\alpha,y)\big)\widehat{U}\big(\alpha,z(\alpha,y)\big)+c(\alpha) \\
    & +\frac{1}{2}\sigma^2\Big(\widehat{U}_{z}\big(\alpha,z(\alpha,y)\big) +\widehat{U}_{x}\big(\alpha,z(\alpha,y)\big) \Big) \\
    & + \int_\alpha^x \bigg( -\delta \widehat{U}(x',z(x',y)) + c'(x') + \cL_{(X^0,Z)}\widehat{U}(x',z(x,y)) \bigg)dx'.
\end{aligned}
\end{equation}
By identity \eqref{eq:filtered:U_hat_change_of_variables} { and} Lemma \ref{filtered:lemma:integral_differential_operator}, noticing that $\widehat{U}_x(x,z(x,y))=U_x(x,y) + \frac{\gamma}{\sigma}y(1-y)U_y(x,y)$ and $\widehat{U}_z(x,z(x,y))=\frac{\gamma}{\sigma}y(1-y)U_y(x,y)$, we finally get that \eqref{eq:thm:optimal_control:PDE2} is equivalent to 
\[
    \cL_{(X^0,\Pi)}V(x,y) + c(x)  - \lambda(y) = \int_{\alpha}^x \bigg( \cL_{(X^0,\Pi)}U(x',y)  -\delta U(x',y) + c'(x') \bigg)dx'.
\]
As $\cL_{(X^0,\Pi)}U \in L^\infty_{loc}(\cO)$ and $(x',y) \in \cC$ for any $(x',y) \in \{(x,y) \in \cO: \, a_+(y) < x < \alpha \}$, we invoke \eqref{eq:U_var_ineq} to conclude
\[
    \cL_{(X^0,\Pi)}V(x,y) + c(x) - \lambda(y) = \int_{a_+(y)}^x \bigg( \cL_{(X^0,\Pi)}U(x',y)  -\delta U(x',y) + c'(x') \bigg)dx'.
\]
i.e. \eqref{thm:HJB_sol_weak:differential_identity}.

As for $\lambda(y)$, we notice that it is continuous, as $b(x,y)$, $\gamma y (1-y)$ are so and $U \in \dC^1(\cO)$. Thus, $ \lambda \in L^\infty_{loc}\big((0,1)\big)$.
By employing the bounds \eqref{eq:U_derivatives:final_bounds} on $U_x$ and $U_y$, we deduce that $\lambda(y)$ is bounded.
Therefore, the process $(\lambda(\Pi_t))_{t \geq 0}$ is bounded as well, hence $d\P_{y}\otimes dt$ integrable for any $T>0$.

Arguing as in \cite[Lemma A.1]{federico2023inventory_unknown_demand}, one has $\P( (X^{\xi^\star}_t,\Pi_t) \in \cC ) = 1$, for all $t \geq 0$.
As $\xi^\star$ is admissible by Lemma \ref{filtered:lemma:reflection}, all the assumptions of Theorem \ref{thm:HJB_sol_weak}. We conclude that $\xi^\star$ is an optimal control and that \eqref{eq:HJB_sol_weak:value_representation} holds.
{ Since $\Pi$ admits a stationary distribution by Lemma \ref{filtered:lemma:stationary_distribution}}, Remark \ref{rmk:ergodicity} implies that $\lambda^\star(y) = \varlimsup_{T \uparrow \infty}\frac{1}{T}\E_{y}[\int_0^T \lambda(\Pi_t)dt]$ is constant in $y$, equal to the value of the ergodic stochastic singular control problem $\lambda^\star$.
\end{proof}

\subsection{Inventory control with observable mean-reversion level}\label{sec:application:full_observation}
Consider a complete filtered probability space $(\Omega,\cF,\bbF = (\cF_t)_{t \geq 0},\P)$, equipped with two { correlated Brownian motions $W^1$ and $W^2$, with correlation factor $\rho$ such that $\vert \rho \vert < 1$.
The joint dynamics of the inventory process $X$ and the mean-reversion level $Y$ are given by
\begin{equation}\label{application:elliptic:dynamics}
\left\{ \begin{aligned}
    & d X^{\xi}_t = \big(Y_t -\delta X^\xi_t \big)dt + \sigma_1 dW^1_t +d\xi^+_t - d\xi^-_t, && X^\xi_0 = x, \\
    & d Y_t = \big(m - b Y_t \big)dt + \sigma_2 dW^2_t, && Y_0 = y,
\end{aligned} \right.
\end{equation}
with $\xi=(\xi^+,\xi^-)$ in $\cB$, $m \in \R$ and $b$, $\sigma_1$ and $\sigma_2$ positive constants.}
Notice that, as all parameters are positive, the pair $(X^0,Y)$ is ergodic.
The firm aims at finding an optimal control which realizes
\begin{equation}\label{eq:application:elliptic:value}
    \inf_{ \xi \in \cB } \varlimsup_{T \to \infty}\frac{1}{T}\E_{x,y}\left[\int_0^T c(X^{\xi}_t)dt +K_+\xi^+_T +K_-\xi^-_T \right],
\end{equation}
where the instantaneous cost $c:\R \to \R$ depends only on the inventory level.
To simplify and shorten the analysis, we assume that the cost function $c$ is quadratic and it is given by
\[
c(x) = \frac{1}{2}x^2,
\]
although a more general function, in the spirit of Section \ref{sec:application:partial_observation}, could be considered.
In the following sections, we show that we can apply Theorem \ref{thm:HJB_sol} and we build an optimal control $\xi^\star$.

\smallskip
We make the following assumption on the parameters:
\begin{assumption}\label{elliptic:assumption:parameters}
We have $b > \delta$.
\end{assumption}
This assumption is needed to ensure the Lipschitz-property of the free-boundaries $a_\pm$ of the associated Dynkin game, which in turn will ensure the integrability of $(\lambda(Y_t))_{t \geq 0}$.

\subsubsection{The associated Dynkin game}

Here, we state the structure of the Dynkin game associated to the control problem \eqref{eq:application:elliptic:value}.
As the volatility coefficients in \eqref{application:elliptic:dynamics} do not depend on the state process $X^\xi$, the Markov process underlying the auxiliary Dynkin game is just given by $(X^0,Y)$ solution of \eqref{application:elliptic:dynamics} with constant control $\xi \equiv 0$.
When needed, in order to simplify the notation, we write $(X^{x,y},Y^{x,y})$ to stress the dependence on the initial value $(x,y) \in \R^2$.

The Dynkin game associated to the control problem is given by
\begin{equation}\label{eq:elliptic:dynkin}
U(x,y) \coloneqq \inf_{\tau} \sup_ {\theta} \E\left[ \int_0^{\tau \land \theta} e^{-\delta t} c'(X^{x,y}_t) dt +K_- e^{-\delta \tau}\ind_{\tau < \theta} - K_+ e^{-\delta \theta}\ind_{\theta < \tau} \right],
\end{equation}
where $c'(x) =  x$.
{ We stress that the infinitesimal generator $\cL_{(X^0,Y)}$ is uniformly elliptic, as the volatility matrix of $(X^0,Y)$ is constant, $\rho \neq  \pm 1$, the drifts are linear and the two obstacles are constant functions. Then, 
by Theorems 3.2, 3.4 and 4.1 in \cite{friedman1982variational} (upon also using Exercises 2 and 5), we have there exists a unique solution $\Tilde{U} \in \bbW^{2,\infty}_{loc}(\R^2)$ of the pointwise variational inequality \eqref{eq:U_var_ineq}.
Define the sets $\cS_+$ and $\cS_-$ by
\[
\cS_+ \coloneqq \{ (x,y) \in \R^2: \, \Tilde{U}(x,y) \leq -K_+ \}, \quad \cS_- \coloneqq \{ (x,y) \in \R^2: \, \Tilde{U}(x,y) \geq K_- \},
\]
and $\cC \coloneqq \R^2 \setminus (\cS_+ \cup \cS_-)$.
By Sobolev embedding, $\Tilde{U} \in \dC^{1}(\R^2)$. Moreover, due to the uniform ellipticity of the generator $\cL_{(X^0,Y)}$, classical Schauder's estimates imply that $\Tilde{U} \in \dC^{\infty}(\cC)$, as the running cost function $c'(x)= x$ belongs to $\dC^{\infty}(\R^2)$ (see \cite[Theorem 6.13]{gilbarg_trudinger2001elliptic}). Furthermore, it is clear that $\Tilde{U} \in \dC^{\infty}(\mathcal{S}_+ \cup \mathcal{S}_-),$ being $\Tilde{U}$ constant therein.}

We verify that the solution $\Tilde{U}$ coincides with the value function of the Dynkin game $U$.
To this extent, take $\tau^* = \inf\{t \geq 0: \, (X^{x,y}_t,Y^{y}_t) \in \cS_-\}$ and $\theta^* = \inf\{t \geq 0: \, (X^{x,y}_t,Y^{y}_t) \in { \cS_+}\}$.
Take any $(x,y) \in \R^2$. By applying a weak version of Dynkin's formula to $e^{-\delta t } \Tilde{U}(X^{x,y}_{t},Y^{y}_{t})$ (see, e.g., \cite{bensoussan_lions1982variational}, Lemma 8.1 and Theorem 8.5, pp. 183-186), for any stopping time $\theta$ we have
\begin{multline*}
\E[e^{-\delta (t \land \tau^* \land \theta) }\Tilde{U}(X^{x,y}_{t \land \tau^* \land \theta},Y^{y}_{t \land \tau^* \land \theta})] \\
= \Tilde{U}(x,y) + \E\left[ \int_0^{t \land \tau^* \land \theta} e^{-\delta s}(\cL_{(X^0,Y)}\Tilde{U}(X^{x,y}_s,Y^{y}_s) -\delta \Tilde{U}(X^{x,y}_s,Y^{y}_s) )ds \right].
\end{multline*}
By using the fact that $\cL_{(X^0,Y)}\Tilde{U}(x,y) -\delta \Tilde{U}(x,y) \leq -c'(x)$ for any $(x,y) \in \R^2 \setminus \cS_-$, we get 
\begin{equation*}
\Tilde{U}(x,y) \geq  \E\left[ \int_0^{t \land \tau^* \land \theta} e^{-\delta s}c'(X^{x,y}_s) ds + e^{-\delta (t \land \tau^* \land \theta) } \Tilde{U}(X^{x,y}_{t \land \tau^* \land \theta},Y^{x,y}_{t \land \tau^* \land \theta}) \right]
\end{equation*}
By sending $t \to \infty$, we deduce
\begin{align*}
\Tilde{U}(x,y) & \geq  \E\left[ \int_0^{\tau^* \land \theta} e^{-\delta s}c'(X^{x,y}_s) ds + e^{-\delta (\tau^* \land \theta) } \Tilde{U}(X^{x,y}_{\tau^* \land \theta},Y^{x,y}_{\tau^* \land \theta}) \right] \\
& \geq \E\left[ \int_0^{\tau^* \land \theta} e^{-\delta s}c'(X^{x,y}_s) ds -K_+ e^{-\delta \theta } \ind_{\theta < \tau^*} + K_- e^{-\delta \tau^* }\ind_{\tau^* < \theta} \right]
\end{align*}
where we used $\Tilde{U}(x,y) \geq -K_+$ on $\R^2$ and $\Tilde{U}(x,y) = K_-$ on $\cS_-$.
Relying on the fact that $\cL_{(X^0,Y)}\Tilde{U}(x,y) -\delta \Tilde{U}(x,y) \geq -c'(x)$ for any $(x,y) \in \R^2 \setminus \cS_+$, the same computations show
\begin{equation}\label{eq:elliptic:saddle_point}
    \Tilde{U}(x,y) \leq \E\left[ \int_0^{\tau \land \theta^*} e^{-\delta s}c'(X^{x,y}_s) ds -K_+ e^{-\delta \theta^* } \ind_{\theta^* < \tau} + K_- e^{-\delta \tau }\ind_{\tau < \theta^*} \right] 
\end{equation}
for any $(\tau,\theta)$ stopping times.
Finally, the exploiting $-\delta \Tilde{U}(x,y) + \cL_{(X^0,Y)}\Tilde{U}(x,y) + c'(x) = 0$ for $(x,y) \in \cC$, we get that \eqref{eq:elliptic:saddle_point} holds with equality for $(\tau^*,\theta^*)$.
This proves that $(\tau^*,\theta^*)$ identifies a saddle point for the Dynkin game \eqref{eq:elliptic:dynkin}, and thus $\Tilde{U}(x,y)$ coincides with the value function $U(x,y)$.

\smallskip
It remains to express the sets $\cS_+$ and $\cS_-$ in terms of the free-boundaries $a_\pm:\R \to \R$ and to investigate their properties.
To this extent, we first notice that, for any fixed $y \in \R$, $x \mapsto U(x,y)$ is non-decreasing.
To see this, fix $y \in \R$, let $x_1 \leq x_2$.
Recalling that $b > \delta$, by linearity we have,
\begin{equation}\label{eq:elliptic:linear_representation}
Y^y_t = e^{-b t}y + Y^0_t, \quad { X^{x,y}_t = e^{-\delta t}x + \frac{1}{\delta - b} (e^{-b t}-e^{-\delta t}) y + X^{0,0}_t,}
\end{equation}
for any $t \geq 0$ $\P_{x,y}$-a.s., where $(X^{0,0},Y^{0,0})$ denotes the solution of \eqref{application:elliptic:dynamics} starting from $(x,y) = (0,0)$.
This clearly implies that $X^{x_1,y}_t \leq  X^{x_2,y}_t$, $\P$-a.s., for all $t \geq 0$. 
As $c'$ is strictly increasing, we also have $c'(X^{x_1,y}_t) \leq  c'(X^{x_2,y}_t)$ so that we deduce $U(x_1,y) \leq U(x_2,y)$.
By the same reasoning as above, { recalling that $b > \delta$ by Assumption \ref{elliptic:assumption:parameters} }, one has that $X^{x,y_1}_t \leq  X^{x,y_2}_t$ whenever $y_1 \leq y_2$, which in turn implies $U(x,y_1) \leq U(x,y_2)$.

\smallskip
We then set
\begin{equation}\label{eq:elliptic:barriers}
    a_-(y) \coloneqq \inf\{ x \in \R: \; U(x,y) \geq K_- \}, \quad a_+(y) \coloneqq \sup\{ x \in \R: \; U(x,y) \leq -K_+ \},
\end{equation}
with the conventions $\sup \emptyset = -\infty$, $\inf\emptyset = +\infty$.
Then, continuity and monotonicity of $U$, upon exploiting the bounds $-K_+ \leq U(x,y) \leq K_-$, yields
\begin{equation}\label{eq:elliptic:regions_U}
\begin{aligned}
    \cS_+ & = \{ (x,y) \in \R^2: \, x \leq a_+(y) \}, \quad \cS_- = \{ (x,y) \in \R^2: \, x \geq a_-(y) \}, \\
    \cC & = \{ (x,y) \in \R^2: \, a_+(y) < x < a_-(y) \}.
\end{aligned}
\end{equation}

\begin{lemma}\label{elliptic:lemma:regularity_boundary}
\mbox{}
\begin{enumerate}[resume,label=(\roman*)]
    \item \label{elliptic:lemma:regularity_boundary:monotonicity} The maps $a_\pm(y)$ are non-increasing. Moreover, $a_-$ is right-continuous and $a_+$ is left-continuous.
    \item \label{elliptic:lemma:regularity_boundary:separated} 
    For any $y \in \R$, it holds $a_+(y) \leq (c')^{-1}(-K_+\delta) < (c')^{-1}(K_-\delta) \leq a_-(y)$.
    \item \label{elliptic:lemma:regularity_boundary:finiteness} For any $y \in \R$, $a_+(y) > -\infty$ and $a_-(y) < +\infty$.
    \item \label{elliptic:lemma:regularity_boundary:lipschitz} $a_+$ and $a_-$ are Lipschitz continuous.
\end{enumerate}
\end{lemma}
\begin{proof}
Points \ref{elliptic:lemma:regularity_boundary:monotonicity} and \ref{elliptic:lemma:regularity_boundary:separated} can be proven exactly as in Lemma \ref{filtered:lemma:regularity_boundary}.

As for point \ref{elliptic:lemma:regularity_boundary:finiteness}, 
suppose there exists $y_o \in \R$ so that $a_+(y_o) = -\infty$.
Then, as $a_+$ is non-increasing, we also have $a_+(y) = -\infty$ for any $y \geq y_o$.
Take $y_o < y_1 < y_2$ and $\gamma < \inf_{y \in [y_1,y_2]} a_-(y)$.
{ Let $Q \coloneqq (-\infty,\gamma) \times (y_1,y_2)$ and notice that $Q \subseteq \cC$.
Consider again $\theta^* = \inf \{t \geq 0: \, (X^{x,y}_t,Y^{x,y}_t) \in \cS_+\}$, $\tau^* = \inf \{t \geq 0: \, (X^{x,y}_t,Y^{x,y}_t) \in \cS_-\}$, and define $\tau^Q \coloneqq \inf \{t \geq 0: \, (X^{x,y}_t,Y^{x,y}_t) \notin Q\}$.
Let $(x,y) \in Q$.
By the same reasoning as in Lemma \ref{filtered:lemma_regularity_U} in Section \ref{sec:application:partial_observation:filtered_problem}, the semi-harmonic characterization of $U$ (as given in \cite[Theorem 2.1]{peskir2008optimal_stopping_games}) and Doob's stopping theorem yield that the process
\[
    \Big( e^{-\delta (t \land \theta^* \land \tau^* \land \tau^Q)}U(X^{x,y}_{t \land \theta^* \land \tau^* \land \tau^Q},Y^{x,y}_{t \land \theta^* \land \tau^* \land \tau^Q}) + \int_0^{t\land \theta^* \land \tau^* \land \tau^Q}e^{-\delta s}c'(X^{x,y}_s)ds \Big)_{t \geq 0}
\]
is a martingale. Moreover, as $Q \subseteq \cC$,  for any $(x,y) \in Q$ we have $\tau^Q \leq \theta^* \land \tau^* = \inf \{ t \geq 0: (X^{x,y}_t,Y^{x,y}_t) \notin \cC \}$.
Thus, for any $(x,y) \in Q$, we get
\begin{align*}
    & - K_+ < U(x,y) = \E\left[ U(X^{x,y}_{t\land \tau^Q},Y^{y}_{t\land \tau^Q}) \right] +  \E\left[\int_0^{t\land \tau^Q}e^{-\delta s}c'(X^{x,y}_s)ds  \right] \\
    & \leq K_- +  \E\left[\int_0^{t\land \tau^Q}e^{-\delta s}c'(X^{x,y}_s)ds  \right] \\
    & = K_-  +  x \E\left[\int_0^{t\land \tau^Q}e^{-2\delta s} \right]+ \E\left[\int_0^{t\land \tau^Q}e^{-\delta s} X^{0,y}_s ds \right] \\
    & \leq K_-  +  x \E\left[\int_0^{t\land \tau^Q}e^{-2\delta s} \right]+ \int_0^{\infty}e^{-\delta s}\E\left[ \vert X^{0,y}_s \vert \right]ds \leq \kappa +  x \E\left[\int_0^{t\land \tau^Q}e^{-2\delta s} \right],
\end{align*}
where $\kappa$ is a positive constant independent of $x$.
We now take the limit as $x \to -\infty$: noticing that $\tau^Q \to \inf\{t \geq 0: Y_t^{y} \notin (y_1,y_2)\} =: \tau^{(y_1,y_2)}$ as $x \to -\infty$, the expectations on the right-hand side converge to finite values.
In particular, $\lim_{x \to -\infty}\E\left[\int_0^{t\land \tau^Q}e^{-2\delta s} \right] = \frac{1}{2\delta}(1-\E[e^{-2\delta (t \land \tau^{(y_1,y_2)})}]) > 0$, which in turn implies that the right-hand side converges to $-\infty$ as $x \to -\infty$, which leads to a contradiction.
}
Finiteness of $a_-$ is dealt with analogously.

\smallskip
We now show that $a^\epsilon_+$ is Lipschitz continuous, by following \cite{deangelis_stabile2019lipschitz}.
We show the Lipschitz continuity of the lower boundary $a_+$, as the upper boundary $a_-$ can be treated analogously.
For $\epsilon \in (0,1]$, set
\begin{equation*}
    a^{\epsilon}_+(y) \coloneqq \sup \{ x \in \R: \, U(x,y) \leq -K_+ + \epsilon \}.
\end{equation*}
We notice that $a^\epsilon_+(y)$ converges to $a_+(y)$ pointwise as $\epsilon \downarrow 0$.
Indeed, as $x \mapsto U(x,y)$ is increasing for any fixed $y$ and so is $y \mapsto U(x,y)$ for fixed $x$, we have $(x,a^\epsilon_+(y)) \in \cC$ for any $(x,y)$ and $\epsilon > 0$, so that $a^\epsilon_+(y) \geq a_+(y)$. Moroever, by definition, $(a^\epsilon_+(y))_{\epsilon > 0}$ is non-increasing in $\epsilon$. This implies $\lim_{\epsilon \downarrow 0} a^\epsilon_+(y) \geq a_+(y)$.
By using the continuity of $U$, we have $U(\lim_{\epsilon \downarrow 0} a^\epsilon_+(y),y) = \lim_{\epsilon \downarrow 0} U(a^\epsilon_+(y),y) = \lim_{\epsilon \downarrow 0} (-K_+ +\epsilon) = -K_+$, which implies that $(\lim_{\epsilon \downarrow 0} a^\epsilon_+(y),y)$ belongs to $\cS_+$ and thus $\lim_{\epsilon \downarrow 0} a^\epsilon_+(y) \leq a_+(y)$.

We have that $a^\epsilon_+$ is continuously differentiable over $\R$.
This can be proved as follows: we first notice that, by the same calculations of the proof of Theorem \ref{filtered:thm:regularity_up_to_boundary}, and exploiting the linearity of the dynamics of $(X^0,Y)$ and of instantaneous cost $c'(x)$, the partial derivatives of $U$ at the points $(a^\epsilon_+(y),y)$ can be represented as
\begin{equation}\label{eq:elliptic:derivatives}
\begin{aligned}
& U_x(a^\epsilon_+(y),y) = \E\left[\int_0^{\tau^*_\epsilon\land \theta^*_\epsilon} e^{-2\delta t} dt \right], \\
& U_y(a^\epsilon_+(y),y) =  { \frac{1}{\delta - b}} \E\left[\int_0^{\tau^*_\epsilon\land \theta^*_\epsilon} e^{-\delta t}\left( e^{-b t} - e^{-\delta t} \right) dt \right],
\end{aligned}
\end{equation}
where $(\tau^*_\epsilon,\theta^*_\epsilon)$ are the exit times of $(X^{(a^\epsilon_+(y),y)},Y^{(a^\epsilon_+(y),y)})$ from $\cS_+$ and $\cS_-$ respectively.
As $(a^\epsilon_+(y),y) \in \cC$ for every $y \in \R$ and $\epsilon > 0$ and $U \in \dC^2(\cC)$, the implicit function theorem implies
\begin{equation}\label{eq:elliptic:lemma:regularity_boundary:derivative_a_epsilon}
\begin{aligned}
\frac{d a^\epsilon_+}{dy}(y) = -\frac{ U_y(a^\epsilon_+(y),y) }{U_x(a^\epsilon_+(y),y)} = -{ \frac{1}{\delta - b}} \frac{ \int_0^\infty \P(\tau^*_\epsilon\land \theta^*_\epsilon \leq t) e^{-\delta t}\left( e^{-b t} - e^{-\delta t} \right) dt }{\int_0^\infty \P(\tau^*_\epsilon\land \theta^*_\epsilon \leq t) e^{-2\delta t}  dt }.
\end{aligned}
\end{equation}
Thanks to Assumption \ref{elliptic:assumption:parameters}, we can bound the derivative of $a^\epsilon_+$ uniformly by a constant $L$ independent of $\epsilon$ and $y$.
Indeed, we have
\[
\left \vert \frac{d a^\epsilon_+}{dy}(y) \right\vert \leq { \frac{1}{b - \delta}} \frac{ \int_0^\infty \P(\tau^*_\epsilon\land \theta^*_\epsilon \leq t) e^{-2\delta t}\left\vert e^{-(b-\delta) t} -1 \right\vert dt }{\int_0^\infty \P(\tau^*_\epsilon\land \theta^*_\epsilon \leq t) e^{-2\delta t}  dt } \leq { \frac{2}{b - \delta}}.
\]
Next, fix $y \in \R$ and let $y_1$, $y_2$ so that $ y_1 < y < y_2$.
By point \ref{elliptic:lemma:regularity_boundary:finiteness} and monotonicity of $a^\epsilon_+$ for any $\epsilon > 0$, we have $-\infty < a_+(y_1) \leq a^\epsilon_+(y) \leq a^1_+(y_2) < \infty$ for any $y$ in the compact interval $[y_1,y_2]$.
Thus, $(a^\epsilon_+)_{\epsilon \in (0,1]}$ is a sequence of equicontinuous uniformly bounded functions defined on the compact $[y_1,y_2]$.
Ascoli-Arzelà's theorem then implies that, up to a subsequence, $a^\epsilon_+$ converges uniformly to $a_+$ on $[y_1,y_2]$.
This allows us to deduce that $a_+$ is Lipschitz continuous on $[y_1,y_2]$ with constant $L$ independent of the interval. As the interval is arbitrary, this implies that $a_+$ is Lipschitz over the whole line.
\end{proof}

\subsubsection{Solution to the ergodic singular control problem}

As prescribed by Theorem \ref{thm:HJB_sol}, we now build the optimal control $\xi^\star=(\xi^{\star,+},\xi^{\star,-}) \in \cB$ which keeps the process inside the open set $\cC$.
Consider the process $\xi^\star=(\xi^{\star,+},\xi^{\star,-})$ so that the pair $(X^{\xi^\star},\xi^\star)$ solves the Skorohod reflection problem:
\begin{equation}\label{eq:elliptic:skorohod_reflection}
\left\{ \begin{aligned}
    & a_+(Y_t) \leq X^{\xi^\star}_t \leq a_-(Y_t), \quad \P_{x,y}\text{-a.s.,} \;\; \text{for almost all }t \geq 0,\\
    & \xi^{\star,+}_t = \int_0^t \ind_{\{ X^{\xi^\star}_{s-} \leq a_+(Y_s) \}}d\xi^{\star,+}_s, \quad \xi^{\star,-}_t = \int_0^t \ind_{\{ X^{\xi^\star}_{s-} \geq a_-(Y_s) \}}d\xi^{\star,-}_s, \quad \P_{x,y}\text{-a.s.,} \;\; \forall t \geq 0,
\end{aligned} \right.
\end{equation}
By the same techniques of Lemma \ref{filtered:lemma:reflection} in Section \ref{sec:application:partial_observation}, it can be proved that there exists a solution to \eqref{eq:elliptic:skorohod_reflection}.
As the proof is essentially the same, we omit it.
With respect to \eqref{eq:filtered:skorohod_reflection}, we notice that we dropped the third request about minimality of the jump size, as in this case the boundaries $a_\pm$ are proven to be continuous.
Moreover, $\xi^\star$ is admissible: as the control $\xi^\star$ keeps the process in $[a_+(Y_t),a_-(Y_t)]$ and $a_\pm$ are Lipschitz continuous by Lemma \ref{elliptic:lemma:regularity_boundary}, we have $\vert X^{\xi^\star}_T \vert \leq \kappa(1 + \vert Y_T \vert )$.
By exploiting the explicit representation of $Y$, we deduce $\lim_{T \to \infty}\frac{1}{T} \E_{y}[\vert Y_T \vert] = 0$, which gives the admissibility of $\xi^\star$ as required by Definition \ref{def:admissible_policies}.

\smallskip
We finally apply Theorem \ref{thm:HJB_sol} to guarantee that the control $\xi^\star$ is optimal.
\begin{theorem}\label{elliptic:thm:solution_control}
Let $\xi^\star$ be given by \eqref{eq:elliptic:skorohod_reflection}. Then, $\xi^\star$ is optimal and the stochastic singular control problem has a value $\lambda^\star$, given by
\[
\lambda^\star = \varlimsup_{T \to \infty} \frac{1}{T} \E_{y}\left[\int_0^T\lambda(Y_t)dt \right],
\]
with $\lambda$ defined accordingly to \eqref{eq:main:built_value}.
\end{theorem}
\begin{proof}
By Lemma \ref{elliptic:lemma:regularity_boundary}, there exists $\alpha \in \R$ so that $\sup_{y \in \R} a_+(y) < \alpha < \inf_{y \in \R} a_-(y)$, so that we can define $V$ as in \eqref{eq:main:built_potential}.
Moreover, as $U$ belongs to $\bbW^{2,\infty}_{loc}(\R^2) \cap \dC^2(\cC)$, $V \in \bbW^{2,\infty}_{loc}(\R^2) \cap \dC^2(\cC)$ as well.
Recall from \eqref{eq:main:built_value} that $\lambda(y)$ is given by
\[
\lambda(y) = c(\alpha) + b(\alpha,y)U(\alpha,y) + \frac{1}{2}\sigma_1^2 U_x(\alpha,y) + \rho\sigma_1\sigma_2 U_y(\alpha,y).
\]
By \eqref{eq:elliptic:derivatives}, $U_x(x,y)$ and $U_y(x,y)$ are bounded over $\R^2$.
As $U(x,y)$ is bounded as well and $b(x,y)$ is linear, we have $\vert \lambda(y) \vert \leq \kappa (1 + \vert y \vert )$, which implies that $(\lambda(Y_t))_{t \geq 0}$ belongs to $L^1(\Omega \times [0,T])$ for any $T \geq 0$.
Finally, $\xi^\star$ is admissible and, arguing as in \cite[Lemma A.1]{federico2023inventory_unknown_demand}, $\P( (X^{\xi^\star}_t,Y_t) \in \cC ) = 1$, for all $t \geq 0$.
Thus, all assumptions of Theorem \ref{thm:HJB_sol} are verified, and we can conclude that $\xi^\star$ is an optimal control and that \eqref{eq:HJB_sol:value_representation} holds.
As $Y$ admits a stationary distribution, Remark \ref{rmk:ergodicity} implies that $\varlimsup_{T \uparrow \infty}\frac{1}{T}\E_{y}[\int_0^T \lambda(\Pi_t)dt]$ is constant in $y$ and equal to the value of the ergodic stochastic singular control problem $\lambda^\star$.
\end{proof}

\section*{Acknowledgments}
\noindent 
Steven Campbell, Georgy Gaitsgori, Ioannis Karatzas, Gechun Liang, Nizar Touzi, and Renyuan Xu are acknowledged by the authors for stimulating discussions.

Funded by the Deutsche Forschungsgemeinschaft (DFG, German Research Foundation) – Project-ID 317210226 – SFB 1283.

This work started during the visits at the Center for Mathematical Economics (IMW) at Bielefeld University of Alessandro Calvia, who thanks IMW and the SFB 1283 for the support and hospitality.

The first author is a member of the Gruppo Nazionale per l’Analisi Matematica, la Probabilità e le loro Applicazioni (GNAMPA) of the Istituto Nazionale di Alta Matematica "Francesco Severi" (INdAM).

\bibliographystyle{abbrv}
\bibliography{biblio}

\end{document}